\documentclass[reqno,12pt]{amsart}
\usepackage{amsmath,amssymb,amsthm,graphicx,mathrsfs,url}
\usepackage[usenames,dvipsnames]{color}
\usepackage[colorlinks=true,linkcolor=Red,citecolor=Green]{hyperref}
\usepackage[super]{nth}
\usepackage[open, openlevel=2, depth=3, atend]{bookmark}
\hypersetup{pdfstartview=XYZ}
\usepackage[font=footnotesize]{caption}
\usepackage{a4wide}
\usepackage{xcolor}

\newcommand{\p}{\partial}

\newcommand{\cjd}{\rangle}
\newcommand{\cjg}{\langle}

\newcommand\CC{\mathbb{C}}
\newcommand\RR{\mathbb{R}}
\newcommand\R{\mathbb{R}}
\newcommand\rr{\mathbb{R}}
\newcommand\NN{\mathbb{N}}

\newcommand\Id{\operatorname{Id}}

\newcommand{\pl}{\partial}
\newcommand{\x}{\times}

\newcommand{\til}{\widetilde}
\newcommand{\bbar}{\overline}
\newcommand{\la}{\lambda}
\newcommand{\eps}{\epsilon}
\newcommand{\mc}{\mathcal}

\newcommand\xin{\overline{\xi}_0}

\newtheorem{theorem}{Theorem}
\newtheorem{lemma}[theorem]{Lemma}
\newtheorem{proposition}[theorem]{Proposition}
\newtheorem{corollary}[theorem]{Corollary}
\newtheorem{definition}[theorem]{Definition}

\theoremstyle{remark}

\numberwithin{equation}{section}
\numberwithin{theorem}{section}

\title{X-Ray Transform in Asymptotically Conic Spaces}
\author{Colin Guillarmou}
\address{Universit\'e Paris-Saclay, CNRS, Laboratoire de math\'ematiques d'Orsay, 91405, Orsay, France.}
\email{colin.guillarmou@math.u-psud.fr}

\author{Matti Lassas}
\address{Department of Mathematics and Statistics
P.O. Box 68 (Gustaf Hällströmin katu 2b)
FI-00014 University of Helsinki
Finland}
\email{matti.lassas@helsinki.fi}

\author{Leo Tzou}
\address{Leo Tzou\newline\indent School of Mathematics and Statistics, University of Sydney, NSW 2006, Australia}
\email{leo@maths.usyd.edu.au}

\begin{document}
\begin{abstract}
In this article, we study the properties of the geodesic X-ray transform for asymptotically Euclidean or conic Riemannian metrics and show injectivity under non-trapping and no conjugate point assumptions. We also define a notion of lens data for such metrics and study the associated inverse problem.
\end{abstract}

\maketitle
\section{Introduction}
On Euclidean space $(\R^n,g_0)$, the linear operator $I_0$ mapping a function $f$ to the set of its integrals  
$$
I_0f(\gamma)=\int_\R f(\gamma(r))\,dr
$$
along lines $\gamma$ is called X-ray transform, or Radon transform in dimension $2$. It is known to be invertible using Radon inversion formula and is the basis of X-ray tomography. This operator, when acting on functions supported in a fixed convex and compact set, say for example the unit ball $B$, is also the linearisation of the following natural geometric 
inverse problem: is there a metric $g=e^{2f}g_0$ conformal to the Euclidean metric $g_0$ (with $f\in C_c^\infty(B)$) so that there is a conjugation $\psi: S_{g_0}\R^n\to S_g\R^n$ between the geodesic flow of $g_0$ and $g$ on their respective unit tangent bundles, which is equal to the Identity outside $S_{g_0}B$, i.e. 
\[\varphi^{g}_t(\psi(x,v))=\psi(\varphi_{t}^{g_0}(x,v)) ,\quad \psi(x,v)=(x,v) \textrm{ if }|x|\geq 1\]
This conjugation property can also be written in terms of the equality between two functions
called scattering map and rescaled lengths, that we shall introduce below. 
More generally, one can define an $X$-ray transform on symmetric tensors of order $m\in\NN$ by the formula
$$
I_mf(\gamma)=\int f_{\gamma(t)}(\otimes^m \dot{\gamma}(t))dt
=\int_\R  \sum_{i_1,i_2,\dots, i_m=1}^n f_{i_1i_2\dots i_m}(\gamma(t)) \dot{\gamma}^{i_1}(t)
\dot{\gamma}^{i_2}(t)\dots\dot{\gamma}^{i_m}(t)\, dt.
$$
The case $m=2$ corresponds to the same linearised problem as above but replacing conformal metrics $g=e^{2f}g_0$ by any possible metric $g$ on $\R^n$ which is a compact perturbation of $g_0$. This rigidity problem was solved by Gromov \cite{Gr}, with an alternate proof by Croke \cite{Cr}, for the class of metrics $g$ with no conjugate points. This property is also called boundary rigidity of the Euclidean metric on $B$.\\

In this paper, we investigate a similar problem but for non-compact perturbations of $\R^n$, and more generally non-compact perturbations of metric cones. A metric cone is a warped product 
$(0,\infty)_r\x N$ with metric 
\[g_0(r,y)=dr^2+r^2h_0(y,dy)\]
 where $(N,h_0)$ is a closed Riemannian manifold of dimension $n-1$. Here we work with smooth metrics and we can take our model to be any smoothing of 
the metric cone at the cone tip $r=0$, and more generally any Riemannian metric which is asymptotic near infinity to the region $r\geq 1$ of the cone. To be precise, our Riemannian manifold $(M,g)$ metric will be called \emph{asymptotically conic} if $M$ is the interior of a smooth compact manifold with boundary $\bbar{M}$ and there is a smooth boundary defining function of $\pl\bbar{M}$ such that, in a product decomposition $[0,\eps)_\rho\x \pl\bbar{M}$ near the boundary,
\[ g(\rho,y)= \frac{d\rho^2}{\rho^4}+\frac{h_\rho(y)}{\rho^2}\]
where $h_\rho$ is a smooth $1$-parameter family of metrics on $\pl\bbar{M}$, $y$ being coordinates on $\pl\bbar{M}$; see subsection \ref{subsec normal form}. Here $\pl\bbar{M}$ plays the role of the section $N$ of the cone described above and $\rho$ plays the role of $1/r$. Using the $r$ variable, this means that $g$ has an asymptotic expansion in powers of $1/r$ near infinity, and with leading term the exact conic metric. 
We say that $g$ is \emph{non-trapping} if each complete geodesic $\gamma(t)$ of $g$ tends to $\pl\bbar{M}$ as $t\to \pm \infty$; more generally such geodesics $\gamma$ is said to be non-trapped. 
For example small perturbations of the Euclidean metric on $\RR^n$ are non-trapping.
These types of metrics have been studied intensively in scattering theory for the wave equation 
\cite{Me,MeZw,HaVa, JoSB1, JoSB2,SBWu}.
For asymptotically conic metrics we define the X-ray transform $I_m$ on symmetric tensors by 
\[I_m : C_c^\infty(M;S^m(T^*M))\to L_{\rm loc}^\infty(\mc{G}), \quad 
I_mf(\gamma):=\int_{-\infty}^\infty f_{\gamma(t)}(\otimes^m \dot{\gamma}(t))dt
\]
where $\mc{G}$ is the set of complete non-trapped geodesics of $g$. The kernel of $I_m$ contains the space of exact $(m-1)$-tensors $\{Dh; h\in C^\infty_c(M;S^{m-1}(T^*M))\}$ where $D$ is the symmetrized covariant derivative, that is, $Dh=\mc{S}(\nabla h)$ where $\mc{S}:(T^*M)^{\otimes m}\to (T^*M)^{\otimes m}$ is the  symmetrization operator,
$$
\mc{S}(\sum_{i_1,\dots i_m}h_{i_1i_2\dots i_m}dx^{i_1}\otimes \dots \otimes dx^{i_m} )=\sum_{i_1,\dots i_m} \frac 1{n!}(\sum_\sigma h_{i_{\sigma(1)}i_{\sigma(2)}\dots i_{\sigma(m)}})dx^{i_1}\otimes\dots \otimes dx^{i_m},
$$ 
where $\sigma$  runs over all permutations of the set $\{1,2,\dots,m\}$. We say that $I_m$ is solenoidal injective if its kernel is precisely $\{Dh; h\in C^\infty_c(M;S^{m-1}(T^*M))\}$.

A smooth symmetric scattering tensor of order $m$ is a smooth symmetric tensor $h\in C^\infty(M; S^mT^*M)$ on $M$ which can be written near $\rho=0$ under the form 
\[h=\sum_{j=0}^m \mc{S}\Big(\frac{h_j}{\rho^{j}}\otimes \Big(\frac{d\rho}{\rho^2}\Big)^{\otimes (m-j)}\Big)\]
where $h_j\in C^\infty(\bbar{M}; S^{j}(T^*\pl\bbar{M}))$ and $\mc{S}$ is the symmetrization operator on tensors. Note that they have bounded pointwise norm with respect to $g$. The space of smooth scattering symmetric tensors will be denoted $C^\infty(\bbar{M};S^m({^{\rm sc}T}^*\bbar{M}))$.

\begin{theorem}
\label{injectivity of tensors}
Let $(M,g)$ be an asymptotically conic manifold. 
Assume that $(M,g)$ is non-trapping and has no conjugate points and let $k>n/2+1$. Then:\\
i) If $f\in \rho^k C^\infty (\overline M)$  and $I_0f=0$, then $f=0$.\\
ii) If $f\in \rho^k C^\infty (\overline M; {}^{\rm sc} T^*\bbar{M})$ and $I_1f=0$ then there exists $u \in  \rho^{k-1} C^\infty (\overline M)$ such that $f = du$.\\
iii) Assume in addition that $(M,g)$ has non-positive sectional curvature. 
For each tensor $f \in \rho^k C^\infty (\overline M; S^m ({^{\rm sc} T}^*M))$ satisfying $I_m f = 0$ with $m>1$, there is a tensor $u\in C^\infty (M; S^{m-1}T^*M))\cap \rho^{k-1}L^\infty$ such that $Du = f$.\\
Finally, if instead $(M,g)$ has negative curvature and a non-empty trapped set, then $i),ii),iii)$ are also satisfied. 
\end{theorem}
Assuming that the boundary has additional geometric properties, we obtain similar results with weaker assumptions on the function $f$. Indeed, the geodesics that stay very far near infinity (i.e. close to $\pl\bbar{M}$) approach in the compactification $\bbar{M}$ the geodesics traveling for time $\pi$ in the boundary $(\pl\bbar{M},h_0)$. This leads to some kind of time-$\pi$ ray transform on the boundary, which turns out to be injective if for example the radius of injectivity of $(\pl\bbar{M},h_0)$ is larger than $\pi$; see Proposition \ref{I0 boundary determination}.

We emphasize that the geometric assumption on $(M,g)$ in Theorem \ref{injectivity of tensors} does not include non-trivial decaying perturbations of the Euclidean metrics, as is shown in the companion paper \cite{GMT}: indeed, asymptotically Euclidean metrics on $\RR^n$ with no conjugate points must be Euclidean. However there are plenty of examples satisfying the assumptions of Theorem \ref{injectivity of tensors}, for example some with non-positive curvature. The boundary $\pl\bbar{M}$ could typically be a round sphere of curvature less than $1$; see section \ref{Examples with no conjugate points}.

In the case of Cartan-Hadamard surfaces with curvature decaying at order $\mc{O}(d^{-\kappa})$ for some $\kappa>2$ ($d$ being to distance to a fixed point), Lehtonen-Railo-Salo \cite{LRS}  proved injectivity of X-ray for tensors with the same decay assumptions on $f$. Their result is more general in terms of the behaviour at infinity than ours, but it is weaker in the sense that it requires non-positive curvature for functions and $1$-forms.

We remark that one possible application of that work could be to apply our injectivity results for getting stability in the inverse scattering problem for the wave-equation on such manifolds, since stability estimates for the wave equation can often be reduced to X-ray stability estimates, and the proof of injectivity easily brings stability estimates. \\

In a second part of the article, we define the notion of lens data in that setting, namely the scattering map $S_g$ and the renormalized length $L_g$ of geodesics, in a way similar to the work \cite{GGSU} for asymptotically hyperbolic manifolds. The renormalized length of a non-trapped complete geodesic  $\gamma$ is 
\[ L_{g}(\gamma):=\lim_{\eps\to 0}\ell_g(\gamma\cap \{\rho\geq \eps\})-2/\eps\]
and the scattering map is a symplectic map $S_g: T^*\pl\bbar{M}\to T^*\pl\bbar{M}$ which encodes the asymptotics at 
$t\to \pm \infty$ for $(\gamma(t),\dot{\gamma}(t))$ on $S^*M$; it is defined using a time rescaling of the geodesic flow of $g$, see Section \ref{scatteringsec}. 
We show that being in the kernel of the linearisation of this pair $(L_g,S_g)$ at a given metric $g_0$ means that the X-ray transform $I_2$ (for $g_0$) of the variation of metrics must vanish, implying a deformation rigidity result. 
A non-trivial aspect in this analysis is the determination, in certain geometric cases, of the jets of the metrics at $\pl\bbar{M}$ from the scattering map: we show for example that the scattering map determines the full metric asymptotics if the boundary metric $h_0$ is such that its geodesic flow at time $\pi$ is ergodic on $S^*\pl\bbar{M}$ or more generally if the sectional curvatures of $h_0$ are negative. 
This analysis is also strongly used in our companion paper \cite{GMT} on the non-existence of non-trivial asymptotically Euclidean metrics without conjugate points on $\RR^n$.\\

We conclude with a deformation rigidity result in negative curvature: 
\begin{theorem}
\label{def rigid}Let $(M,g(s))$, $s\in (-1,1),$ be a smooth $1$-parameter family of non-trapping negatively curved asymptotically conic manifolds whose metric $g(s)$ depends smoothly on $s$, and such that near $\pl\bbar{M}$ the metric has the form
$g(s) = \frac{d\rho^2}{\rho^4} + \frac{h(s)}{\rho^2}$, where $h(s)|_{\pl\bbar{M}}=h(0)|_{\pl\bbar{M}}$ and $(\pl\bbar{M},h(0))$ has negative curvature.
Assume the family has constant lens data, that is, 
\[S_{g(s)} = S_{g(0)}, \quad L_{g(s)} = L_{g(0)}.\] 
Then there exists a family of diffeomorphisms 
$\psi(s):M\to M$ with $\psi|_{\pl\bbar{M}} = {\rm Id}$ which satisfies $g(0) = \psi(s)^* g(s)$.
\end{theorem}

\noindent\textbf{Acknowledgement.} We thank the referees for their careful reading, in particular one of them for spotting a small error in a proof.
This project has received funding from the European Research Council (ERC) under the European Union’s Horizon 2020 research and innovation programme (grant agreement No. 725967). C. Guillarmou thanks the Math Dept. of Sydney Univ., where part of this work was done, for its hospitality and funding of the visit. This material is based upon work supported by the National Science Foundation under Grant No. DMS-1440140 while C. Guillarmou and L.Tzou were in residence at the Mathematical Sciences Research Institute in Berkeley, California, during the Fall 2019 semester.
M. Lassas is supported by the Academy of Finland, grants 284715, 312110. L. Tzou is suported by ARC Grants no DP190103302 and DP190103451.

\section{Geometric Preliminaries}

\subsection{Asymptotically conic metrics and normal form}\label{subsec normal form}
Asymptotically conic or scattering manifolds are complete Riemannian manifolds $(M,g)$ where 
$M$ is the interior of a smooth compact manifold with boundary $\bbar{M}$ and $g$ is a smooth metric on $M$ satisfying a certain asymptotic structure at $\pl\bbar{M}$ that we now describe.
Let $\rho_0\in C^\infty(\bbar{M};\rr^+)$ be a smooth boundary defining function, i.e. $\pl\bbar{M}=
\{\rho_0=0\}$ and $d\rho_0|_{\pl M}(y)\not=0$ for all $y\in \pl\bbar{M}$. 
First, following Melrose \cite{Me}, we recall that there is a smooth bundle, called the \emph{scattering tangent bundle} and denoted by ${^{\rm sc}T}\bbar{M}\to\bbar{M}$, 
whose space of smooth sections can be identified with the space of smooth vector fields of the form
$\rho_0 V$, where $V$ are smooth vector fields on $\bbar{M}$ that are tangent to the boundary 
$\pl\bbar{M}$. If $(y_1,\dots,y_{n-1})$ are local coordinates on $\pl \bbar{M}$, we have induced local coordinates $(\rho_0,y_1,\dots,y_{n-1})$ near $\pl \bbar{M}$ in ${\bbar{M}}$, and a local basis of  ${^{\rm sc}T}\bbar{M}\to\bbar{M}$ is given by $\rho_0^2\pl_{\rho_0},\rho_0\pl_{y_1},\dots,\rho_0\pl_{y_{n-1}}$.
The vector bundle dual to ${^{\rm sc}T}\bbar{M}$ will be denoted by ${^{\rm sc}T}^*\bbar{M}$. Near $\partial\bbar{M}$, $d\rho_0/\rho_0^2, dy_1/\rho_0, \dots,dy_{n-1}/\rho_0$ is a local frame of 
${^{\rm sc}T}^*\bbar{M}$. 

\begin{definition}\label{defAE} 
A Riemannian metric $g$ on $M$ is called \emph{asymptotically conic} if 
$g\in C^{\infty}(\bbar{M};S^2({^{\rm sc}T}^*\bbar{M}))$ and there exists a boundary defining function $\rho_0$ such that $\rho_0^{-2}|\nabla^g\rho_0|_{g}=1+\mc{O}(\rho_0^2)$. We say that it is asymptotically conic to order $m\geq 1$ if moreover $g-g_0\in \rho_0^mC^{\infty}(\bbar{M};S^2({^{\rm sc}T}^*\bbar{M}))$ for some smooth metric
$g_0$ on $M$ that is equal to an exact conic metric $g_0=d\rho_0^2/\rho_0^4+h_0/\rho_0^2$ near $\pl\bbar{M}$, $h_0$ being a smooth Riemannian metric on $\pl\bbar{M}$.
\end{definition}
In particular, near $\pl\bbar{M}$, one has 
\begin{equation}\label{g-g0}
g-g_0=\rho_0^m\Big(a\frac{d\rho_0^2}{\rho_0^4}+\frac{\sum_j b_jd\rho_0\, dy_j}{\rho_0^3}+
\frac{\sum_{i,j=1}^{n-1}\ell_{ij}dy_i\,dy_j}{\rho_0^2}\Big)
\end{equation} 
where $a,b_j,\ell_{ij}\in C^{\infty}(\bbar{M})$ with
$\rho_0^ma=\mc{O}(\rho_0^{\max(2,m)})$. A metric cone is $(0,\infty)_r\x N$ with metric 
$dr^2+r^2h_0$ if $(N,h_0)$ is a compact Riemannian manifold. Setting $\rho=1/r$, the metric becomes $d\rho^2/\rho^4+h_0/\rho^2$ outside $r=0$, thus smoothing the tip $r=0$ of the cone indeed gives an asymptotically conic metric. In \cite{JoSB2}, Joshi-Sa Barreto proved that an asymptotically conic metric admits an approximate normal form near the boundary $\pl\bbar{M}$. An exact normal form can in fact easily be obtained by reducing to a non-characteristic first order PDE, this suggestion appears for example in Graham-Kantor \cite{GrKa}: we give here a short self-contained proof based on this argument (we also need to compare the exact cone case to the perturbed one). The approximate normal form of \cite{JoSB2} correspond to the form \eqref{exactnormalform}
up to an error $\mc{O}(s^{\infty})$ as $s\to 0$.

\begin{lemma}\label{normalform}
Let $g$ be asymptotically conic to order $m\geq 1$. Then
there is a boundary defining function $\rho\in C^\infty(\bbar{M})$  satisfying
\begin{equation}
\label{|drho|=1}
\frac{|\nabla^g\rho|_g}{\rho^2}=1,
\qquad
\rho=\rho_0(1+\mc{O}(\rho_0^m)).
\end{equation}
If $m\geq 2$, the function $\rho$ is uniquely determined near the boundary by the equation \eqref{|drho|=1}, while
if $m=1$, such a  $\rho$ is not unique: for each $\omega_0\in C^\infty(\pl\bbar{M})$ there is 
a function $\rho=\rho_0+\omega_0\rho_0^2+\mc{O}(\rho_0^3)$ such that 
$\rho^{-2}|\nabla^g\rho|_g=1$, uniquely defined near $\pl\bbar{M}$.
For each such $\rho$, there is a smooth diffeomorphism 
\[\psi:[0,\eps)_s \x \pl\bbar{M}\to U\subset \bbar{M}\] 
onto a collar neighborhood $U$ of $\pl\bbar{M}$ such that $\psi(0,\cdot)|_{\partial\bbar{M}}=\Id$, $\psi^*\rho=s$ and 
\begin{equation}\label{exactnormalform}
\psi^*g=\frac{ds^2}{s^4}+\frac{h_s}{s^2},
\end{equation} 
where $h_s$ is a smooth family of Riemannian metrics on $\pl\bbar{M}$ such that
\[h_s-h_0\in s^m C^\infty([0,\eps)\x\pl\bbar{M};S^2(T^*\pl\bbar{M})).\] 
The diffeomorphism $\psi$ is given by the expression 
$\psi(s,y)=e^{sZ_\rho}(y)$, where $Z_\rho:=\rho^{-4}\nabla^g\rho\in C^\infty(\bbar{M};T\bbar{M})$.
When $m=1$, if $h_s$ and $\hat{h}_s$ are the smooth family of Riemannian metrics on $\pl\bbar{M}$ associated with two boundary defining function $\rho,\hat{\rho}$ with $\hat{\rho}=\rho+\omega_0\rho^2+\mc{O}(\rho^3)$, then 
\begin{equation}\label{changeofrho}
\hat{h}_s=h_s+s(\mc{L}_{\nabla^{h_0}\omega_0}h_0+2\omega_0h_0)+\mc{O}(s^2).
\end{equation}
\end{lemma}
\begin{proof} 
We search for $\rho:=\rho_0e^{\rho_0 \omega}$ with $\omega\in C^\infty(\bbar{M})$ such that 
$|d\rho/\rho^2|_g=1$ near $\pl\bbar{M}$. This is equivalent to the equation
\begin{equation}\label{HJ}
 2\rho_0^{-3}(\nabla^g\rho_0)(\rho_0\omega)=e^{2\rho_0\omega}-1-\rho_0^2\Big|\frac{d\omega}{\rho_0}\Big|^2_g-\rho_0^2\omega^2\Big|\frac{d\rho_0}{\rho_0^2}\Big|^2_g-2\rho_0^2\omega g\Big(\frac{d\omega}{\rho_0},\frac{d\rho_0}{\rho_0^2}\Big)+\rho_0^2F
\end{equation}
with $F:=\rho_0^{-2}(1-|d\rho_0/\rho_0^2|_g)\in \rho^{\max(m-2,0)}C^{\infty}(\bbar{M})$.
A direct computation gives 
\[ \nabla^g\rho_0=\rho_0^{4}\pl_{\rho_0}+\rho_0^{m+3}V\]
for some smooth vector field on $\bbar{M}$ that has the form $V=\alpha\rho_0\pl_{\rho_0}+\sum_j\beta_j\pl_{y_j}$ near $\pl\bbar{M}$ with $\alpha,\beta_j\in C^{\infty}(\bbar{M})$. Then   
\eqref{HJ} can be rewritten as
\begin{equation}\label{NLODE}
2\pl_{\rho_0}\omega=\rho_0^{-2}\Big(e^{2\rho_0\omega}-1-2\rho_0\omega\Big)+
G_g(\rho_0,y,\omega,\pl_{\rho_0}\omega,\pl_y\omega). 
\end{equation} 
where $G_g$ is $C^{\infty}$ in the variables $\rho_0,y$ and polynomial of degree $2$ in the last $3$ variables, and 
$G_g(0,y,\omega,U,V)$ is independent of $U$. Note that $(e^{x}-1-x)=x^2\sum_{j\geq 0}x^j/(j+2)!$ is smooth, then we see that the equation  \eqref{NLODE} with boundary condition $\omega|_{\rho_0=0}=\omega_0$ at $\pl\bbar{M}$ 
is a non-characteristic first order non-linear PDE with $C^{\infty}$ coefficients. 
It thus has a unique solution $\omega$ near $\pl\bbar{M}$ that is $C^{\infty}(\bbar{M})$. 
If we choose $\omega|_{\pl \bbar{M}}=0$, we obtain a particular solution of \eqref{HJ}. Notice that 
$\omega=0$ is the solution of \eqref{HJ} with $g$ is replaced by $g_0$ and $\omega|_{\pl\bbar{M}}=0$,  and that
\[G_g(\rho_0,y,\omega,\pl_{\rho_0}\omega,\pl_y\omega)=G_{g_0}
(\rho_0,y,\omega,\pl_{\rho_0}\omega,\pl_y\omega)+F+\rho_0^{m-1}\til{G}(\rho_0,y,\omega,\rho_0\pl_{\rho_0}\omega,\pl_y\omega)\]
for some $\til{G}$ that is $C^{\infty}$ in the variables $(\rho_0,y)$, polynomial of order $2$ in the last $3$ variables. 
To simplify notations, we write $G_g(\omega)$ instead of $G_g(\rho_0,y,\omega,\pl_{\rho_0}\omega,\pl_y\omega)$: using the expression of $G_{g_0}(\omega)$ obtained from \eqref{HJ} with $g_0$ instead of $g$,
\begin{equation}\label{omeganear0}
\begin{split}
2\pl_{\rho_0}\omega=& \rho_0^{m-1}\til{G}(\omega)+F+\rho_0^{-2}\Big(e^{2\rho_0\omega}-1-2\rho_0\omega\Big)+G_{g_0}(\omega)\\
=& -|\pl_y\omega|^2_{h_0}+\omega^2+ \mc{O}(|\rho_0\pl_{\rho_0}\omega|(|\omega|+|\rho_0\pl_{\rho_0}\omega|))+\rho_0^{m-1}\til{G}(\omega)+F.
\end{split}\end{equation}
Since $\omega\in C^{\infty}$, we can write its Taylor expansion, and assuming that
\[\omega=\mc{O}(\rho_0^{\ell}), \quad \pl_y\omega=\mc{O}(\rho_0^{\ell}), \quad \pl_{\rho_0}\omega=\mc{O}(\rho_0^{\ell-1})
\] 
for some $\ell \geq 1$  and plugging this into  
 \eqref{omeganear0}, we deduce that $\pl_{\rho_0}\omega=\mc{O}(\rho_0^{m-2})+\mc{O}(\rho_0^{2\ell})$, which shows that $\omega=\mc{O}(\rho_0^{m-1})+\mc{O}(\rho_0^{2\ell+1})$. By induction, this implies that
 $\omega=\mc{O}(\rho_0^{m-1})$ near $\rho_0$.  
Finally, let $Z_\rho:=\nabla^g\rho/\rho^4=\nabla^g\rho/|\nabla^g\rho|_g^2$. This is a $C^{\infty}$ vector field near $\pl\bbar{M}$ that is transverse to $\pl\bbar{M}$.
We consider its flow $\phi_s$ and the diffeomorphism $\psi:[0,\eps)_s\x \pl\bbar{M}\to \bbar{M}$
defined by $(s,y)\mapsto\phi_s(y)$ is $C^{\infty}$ and it is direct to check that it satisfies the desired properties by using that $\rho=\rho_0(1+\mc{O}(\rho_0^m))$. Moreover one has uniqueness of such a defining function if $m\geq 2$ since $\rho$ is determined by $\omega_0$ (which needs to be $0$ in order to get $\rho-\rho_0=\mc{O}(\rho_0^2)$). 
 Now, for the case where $m=1$, we see that the choices of $\rho$ are determined by the boundary value $\omega_0$: assume $\hat{\rho}=\rho+\omega_0\rho^2+\mc{O}(\rho^3)$ are two normal forms and $g=d\rho^2/\rho^4+h_{\rho}/\rho^2$ and $h_{\rho}=h_0+\rho h_1+\mc{O}(\rho^2)$. 
Let $\hat{\phi}_s=e^{sZ_{\hat{\rho}}}$; to check \eqref{changeofrho}, we compute 
 that $\pl_s (s^2\hat{\phi}_s^*g)(\pl_{y_i},\pl_{y_j})|_{s=0}=(h_1+\mc{L}_{Z_{\hat{\rho}}}h_0+2\omega_0h_0)(\pl_{y_i},\pl_{y_j})$. Now
an easy computation shows that $Z_{\hat{\rho}}|_{\pl\bbar{M}}=\pl_{\rho}+\nabla^{h_0}\omega_0$ and this gives \eqref{changeofrho}.
\end{proof}

In what follows, we will study $g$ in normal form, i.e. $g=d\rho^2/\rho^4+h_\rho/\rho^2$ near $\pl \bbar{M}$, for some smooth family $h_\rho$ of metrics on $\pl\bbar{M}$.

\subsection{Curvature tensor}

Let us first compute the decay of the curvature tensor near infinity.
\begin{proposition}
\label{curvature decay}
Let  $V, W\in \ker d\rho$ of length $g_p(V,V)=g_p(W,W)=1$ and $g_p(V,W)=0$ and let $Z=\rho^2\pl_\rho$.
For $p=(\rho,y)\in M$ sufficiently close to $\partial M$, the sectional curvature $K_p$ and Riemann tensor $R_p$  satisfy:\\ 
1) $|K_{p}(V,W)| =|g(R_p(V,W)W,V)|=\mc{O}(\rho^2)$.\\ 
2) $|K_{p}(V, Z)|= |g(R_p(V,W)W,V)|=\mc{O}(\rho^4)$.\\
3) $g(R_p(V, W) W, Z)=\mc{O}(\rho^3)$.\\
If in addition $(\pl\bbar{M},h_0)$ has sectional curvature $+1$, then 
\[K_p(V,W)=g(R_p(V,W)W,V)
=\mc{O}(\rho^3).\]
\end{proposition}
\begin{proof}
Let us define the subbundle $E=\ker d\rho\subset T\bbar{M}$ and 
$\mc{E}=C^\infty(\bbar{M};E)$.
We have $V=\rho\bbar{V}$ and $W=\rho\bbar{W}$ for some smooth $\bbar{V},\bbar{W}\in \mc{E}$, and let $Z=\rho^2\pl_\rho=\rho^2\bbar{Z}$.
Here the vector field $\bbar{Z} \in C^\infty(\bbar{M};T\bbar{M})$ is simply the smooth vector field $\partial_{\rho}$ which is smooth up to the boundary $\partial M$. Using $(\rho,y_1,\dots,y_{n-1})$ local coordinates near $\pl \bbar{M}$ with $(y_i)$ local coordinates on $\pl\bbar{M}$ as above, a direct computation by writing down $V = \sum_j \rho \bar v_j \partial_{y_j}$ and $W = \sum_j \rho \bar w_j \partial_{y_j}$ gives for all $V,W\in \rho\mc{E}$
\[ [Z,V]-\rho V\in \rho^3\mc{E}, \quad [V,W]\in \rho^2\mc{E}\]
(the commutators above belong to $\ker d\rho$ due to the absence of the $\partial_{\rho}$ vector in the resulting coordinate calculation).

Using Koszul formula, the Levi-Civita connection satisfies for all $V,W\in \rho\mc{E}$ 
\[\nabla_{V} W= \rho^2 \nabla^h_{\bbar{V} }\bbar{W} +a(V,W)Z, \quad \nabla_ZZ=0, \quad g(\nabla_ZV,Z)=0, \]
where $\nabla^h$ is the connection of $h(\rho)$ on the hypersurfaces given by level sets of $\rho$ and $a(V,W)$ some smooth function on $\bbar{M}$. 

We claim that the curvature in the direction $V,W\in \mc{E}$ (with $\cjg V,W\cjd=0$ and $|V|_g=1=|W|_g$) is given by
\begin{equation}
\label{slice curvature}
\begin{split}
K(V, W)= & \rho^2( K^h (\bbar{V}, \bbar{W})-1)  + \frac{\rho^3}{2} ((\partial_\rho h)(\bbar{V}, \bbar{V})+(\pl_\rho h)(\bbar{W},\bbar{W}))  \\
& -\frac{\rho^4}{4}\big((\partial_\rho h)(\bbar{V}, \bbar{V})(\partial_\rho h)(\bbar{W}, \bbar{W}) - ((\partial_\rho h)(\bbar{V}, \bbar{W}))^2\big)
\end{split}
\end{equation}
Indeed, we compute for an orthonormal basis $(Y_j)_{j=1}^{n-1}\in \rho\mc{E}$ of $\ker d\rho$ for the metric $g$
\begin{eqnarray*}
g( \nabla_{W} \nabla_{V} W, V)& =& g(\nabla_{W}  g(\nabla_{V} W, Z)Z, V) + \sum_{j=1}^{n-1}g( \nabla_{W}  g(\nabla_{V} W, Y_j) Y_j, V)\\
&=&g( \nabla_{V}W, Z)g( \nabla_{W} Z, V) + \rho^2 h(\nabla^h_{\bbar{W}}
\nabla^h_{\bbar{V}} \bbar{W}, \bbar{V})\\
&=&-g(\nabla_{V} W, Z) g(Z ,\nabla_{W}V) + \rho^2 h(\nabla^h_{\bbar{W}}\nabla^h_{\bbar{V}} \bbar{W}, \bbar{V})\\
&=&-(g(\nabla_{V} W, Z))^2  + \rho^2 h(\nabla^h_{\bbar{W}}\nabla^h_{\bbar{V}} \bbar{W}, \bbar{V})\\
&=&-\frac{1}{4}\rho^4 (\pl_\rho h(\bbar{V},\bbar{W}))^2+ \rho^2 
h(\nabla^h_{\bbar{W}}\nabla^h_{\bbar{V}} \bbar{W}, \bbar{V})
\end{eqnarray*}
The last equality comes from using Koszul formula. 
Also
\begin{eqnarray*}
g( \nabla_{V} \nabla_{W} W, V) &=& g( \nabla_{V} \sum_j g(\nabla_{W}W, Y_j) Y_j, V) + g( \nabla_{V} g(\nabla_{W} W, Z) Z, V)\\
&=& \rho ^2 h( \nabla^h _{\bbar{V}}  \nabla^h_{\bbar{W}} \bbar{W}, \bbar{V}) - g(W ,\nabla_{W} Z) g([V, Z], V)\\
&=& \rho^2  h( \nabla^h_{\bbar{V}} \nabla_{\bbar{W}} \bbar{W}, \bbar{V})- g (W,[Z, W]) g([Z, V],V)\\
&=& \rho^2 h( \nabla^h_{\bbar{V}} \nabla^h_{\bbar{W}} \bbar{W}, \bbar{V}) - g(W, \rho W + \rho^3 [\pl_\rho,\bbar{W}]) g (V, \rho V + \rho^3 [\pl_\rho,\bbar{V}] )
\end{eqnarray*}
Now using the fact that $\partial_\rho (h(\bbar{W},\bbar{W}))= 0=\partial_\rho (h(\bbar{V},\bbar{V}))$, we see that 
$2h([\pl_\rho,\bbar{W}],\bbar{W})=-\pl_\rho h(\bbar{W},\bbar{W})$ and the same with $\bbar{V}$ replacing $\bbar{W}$. We thus obtain
\[\begin{split}
g( \nabla_{V} \nabla_{W} W, V) =& \rho^2 g( \nabla^h_{\bbar{V}} \nabla^h_{\bbar{W}} \bbar{W}, \bbar{V}) -\rho^2-\frac{\rho^4}{4}\pl_\rho h(\bbar{V},\bbar{V})\pl_\rho h(\bbar{W},\bbar{W})\\
&+ \frac{1}{2}\rho^3(\pl_\rho h(\bbar{V},\bbar{V}) +\pl_\rho h(\bbar{W},\bbar{W}))
\end{split}\]
Since also $g( \nabla_{[V,W]}V,W)=\rho^2h (\nabla^h_{[\bbar{V},\bbar{W}]}\bbar{V},\bbar{W})$, we can 
 combine these computations to deduce \eqref{slice curvature}.

Next, we compute the curvature in mixed direction $K(Z,V)$.
\begin{equation}\label{curvature first term}
\begin{split}
g( \nabla_{Z} \nabla_{V}V, Z) =& Z g(\nabla_{V} V,Z)= - Z g( V, \nabla_{V} Z)=Z(\rho h(\bbar{V}, [\partial_\rho,\rho\bbar{V}]) \\
 =&  \rho^2-\rho^3\pl_\rho h(\bbar{V}, \bbar{V})-\frac{1}{2}\rho^4 \pl_\rho(\pl_\rho h(\bbar{V}, \bbar{V})).
\end{split}
\end{equation}
We can choose $V$ to be parallel with respect to $Z$ in the collar $\rho\in (0,\eps)$, so that
\[\begin{split} 
g( \nabla_{V} \nabla_{Z}V, Z) =0.
\end{split}\]
The equation $\nabla_ZV=0$ can be written as (here $S=h^{-1}\pl_\rho h$ denotes the shape operator for the hypersurfaces $\rho={\rm const}$)
\[ [\pl_\rho,\bbar{V}]=-\frac{1}{2}S\bbar{V}.\]
And finally, we have 
\[\begin{split} 
-g(\nabla_{[Z,V]} V,Z) = & g( V, \nabla_{[Z,V]}Z)=g(V,\nabla_Z [Z,V]-
[Z,[Z,V]])\\
=& h( \bbar{V}, \nabla_{\rho\pl_\rho}[Z,V] - 2\rho^2[\pl_\rho,V]-\rho^3[\pl_\rho,[\pl_\rho,V]])\\
=&  \rho^3h( \bbar{V}, \nabla_{\pl_\rho}[\pl_\rho,V] -[\pl_\rho,[\pl_\rho,V]])\\
=& \rho^3h( \bbar{V}, \nabla_{\pl_\rho}\bbar{V}-[\pl_\rho,\bbar{V}]+\rho\nabla_{\pl_\rho}[\pl_\rho,\bbar{V}]  -\rho[\pl_\rho,[\pl_\rho,\bbar{V}]])\\
= & -\rho^2+\frac{\rho^3}{2}\pl_\rho h(\bbar{V},\bbar{V})+\rho^4h( \bbar{V}, \nabla_{\pl_\rho}[\pl_\rho,\bbar{V}]  -[\pl_\rho,[\pl_\rho,\bbar{V}]]).
\end{split}\]
But we also have, using $\nabla_{\pl_\rho}V=0$, 
\[\begin{split} 
\rho^4h( \bbar{V}, \nabla_{\pl_\rho}[\pl_\rho,\bbar{V}]) = & \rho^6(\pl_\rho(g(\bbar{V},[\pl_\rho,\bbar{V}]))+\rho^{3}h(\bbar{V},[\pl_\rho,\bbar{V}])\\
=& -\rho^3h(\bbar{V},[\pl_\rho,\bbar{V}])+\rho^{4}\pl_\rho(h(\bbar{V},[\pl_\rho,\bbar{V}]))\\
=& \frac{1}{2}\rho^3\pl_\rho h(\bbar{V},\bbar{V})-\frac{\rho^4}{2}\pl_\rho(\pl_\rho h(\bbar{V},\bbar{V}))
\end{split}\]
and, by differentiating twice $h(\bbar{V},\bbar{V})=1$ with respect to $\pl_\rho$,
\[ -h([\pl_\rho^2,\bbar{V}],\bbar{V})-\frac{1}{2}\pl_\rho^2h(\bbar{V},\bbar{V})=h([\pl_\rho,\bbar{V}],[\pl_\rho,\bbar{V}])+2\pl_\rho h([\pl_\rho,\bbar{V}],\bbar{V}).\]
Thus, combining with \eqref{curvature first term} we obtain that 
\begin{equation}\label{KZV}
\begin{split}
K(Z,V)=& -\frac{\rho^4}{2} \pl_\rho^2h(\bbar{V}, \bbar{V})+\rho^4h([\pl_\rho,\bbar{V}],[\pl_\rho,\bbar{V}])\\
=& -\frac{\rho^4}{2} \pl_\rho^2h(\bbar{V}, \bbar{V})+\frac{\rho^4}{4}h(S\bbar{V},S\bbar{V}).
\end{split}\end{equation}
This implies that $K(Z,V)=\mc{O}(\rho^4)$ and $K(V,W)=\mc{O}(\rho^2)$. Note that if $h$ has curvature $+1$, then $K(V,W)=\mc{O}(\rho^3)$.  

For the last statement, taking $V,W$ such that $\nabla_ZW=\nabla_ZV=0$, then 
\[\begin{split}
R(Z,W,V,W)=&g(\nabla_Z\nabla_WV-\nabla_{[Z,W]}V,W)=Zg(\nabla_WV,W)-
\rho^2h(\nabla^h_{[\pl_\rho,W]}\bbar{V},\bbar{W})\\
=& Zg([W,V],W)-\rho^2h(\nabla^h_{\bbar{W}+\rho[\pl_\rho,\bbar{W}]}\bbar{V},\bbar{W})\\
=& \rho^2h([\bbar{W},\bbar{V}],\bbar{W})-\rho^2h(\nabla^h_{\bbar{W}}\bbar{V},\bbar{W})+\mc{O}(\rho^3)=\mc{O}(\rho^3)
\end{split}\]
concluding the proof.
\end{proof}

\subsection{Examples with no conjugate points}\label{Examples with no conjugate points}

Let us also give an example of a simply connected asymptotically conic manifold with no conjugate points that is not the Euclidian metric. Take $\rr^n$ and use the radial coordinates $(r,\theta)$ so that the Euclidean metric is $g_0=dr^2+r^2d\theta^2$ with $d\theta^2$ the metric of curvature $+1$ on the round sphere $\mathbb{S}^{n-1}$. Take now a new metric on $\rr^n$ given by in polar coordinates 
\[ g=dr^2+f(r)^2d\theta^2, \quad f''(r)\geq 0, \quad f(r)=r \textrm{ in }[0,1], \quad f(r)=ar \textrm{ in }[4,\infty)\]
where $a>1$ and $f\in C^\infty([0,\infty))$. Then the metric is smooth on $\rr^n$ and is asymptotically conic, as can be seen by setting $\rho=1/r$ outside the ball $\{r\leq 1\}$. Moreover, a standard computation gives that the sectional curvature of $g$ is for any $V,W$ of unit length and tangent to 
the level sets $r={\rm const}$
\[ K(\pl_r,V)=-\frac{f''(r)}{f(r)}\leq 0, \quad K(V,W)=\frac{1}{f^2(r)}(1-(f'(r))^2)\leq 0.\]
In particular this metric has no conjugate points and is a Cartan-Hadamard manifold.\\

In fact, we can choose $f$ so that $f(r)=\sinh(r)$ in $r\in [2,3]$ in which case the curvature is $-1$ in $r\in [2,3]$ and by a result of Gulliver \cite[Theorem 3]{Gu}, there is a new metric $g'$ on $\rr^n$ such that $g'=g$ outside a compact subregion $\Omega$ of $\{r\in (2,3)\}$ , the curvature of $g'$ is positive in an open subset $U\subset \Omega$ and $g'$ has no conjugate points. This shows that there are asymptotically conic metrics on $\rr^n$ that have positive curvature and no conjugate points, they are therefore non-trapping but not Cartan-Hadamard manifolds and do not enter in the class studied by Lehtonen-Railo-Salo \cite{LRS}.\\

In fact, for asymptotically conic metrics $g$ on $\rr^2$, it can be shown that if the boundary metric 
$h(0)$ is such that $\pl M=\mathbb{S}^1$ has length less than $2\pi$, then $g$ must have conjugate points, see \cite{GMT}. 

\subsection{The manifold $\overline{S^*M}$, geodesic vector field, Liouville measure}
The unit cotangent bundle $S^*M=\{ (x,\xi)\in T^*M; |\xi|_g=1\}$ is non-compact. There is a natural compactification by taking the scattering unit cotangent bundle 
${^{\rm sc}S}^*M=\{ (x,\xi)\in {^{\rm sc}T}^*M; |\xi|_g=1\}$ but the geodesic flow does not behave  well near the boundary of that compactification. It is more convenient to work on a non-compact manifold with boundary on which the geodesic vector field is transversal to the boundary. This manifold is denoted by $\bbar{S^*M}$, its interior is $S^*M$ and its boundary consists of two connected components diffeomorphic to $T^*\pl\bbar{M}$. The manifold $\bbar{S^*M}$ is defined using the following procedure. The unit scattering cotangent bundle is    
\[ 
{^{\rm sc}S}^*\bbar{M}:=\big\{ (x,\xi)\in {^{\rm sc}T}^*\bbar{M}; \, |\xi|_{g}=1\big\}.
\]
It is a compact smooth manifold with boundary.
The local coordinates $(\rho,y)$ on $\bbar{M}$ near $\pl\bbar{M}$ induce local coordinates $(\rho,y,\bbar{\xi}_0,\bbar{\eta})$ on ${^{\rm sc}T}^*\bbar{M}$ by writing each $\xi\in{^{\rm sc}T}^*\bbar{M}$ under the form 
\begin{equation}\label{scatxi}
\xi=\bbar{\xi}_0\frac{d\rho}{\rho^2}+\sum_{i=1}^{n-1}\bbar{\eta}_i\frac{dy_i}{\rho},
\end{equation} 
and $\xi\in{^{\rm sc}S}^*\bbar{M}$ means that $\bbar{\xi}_0^2+|\bbar{\eta}|_{h_\rho}^2=1$. 
A quick study of the integral curves of the geodesic vector field on $S^*M$ shows that 
geodesics going to infinity have their $(\bbar{\xi}_0,\bbar{\eta})$ components tending to 
$(\pm 1,0)$, which suggest to use polar coordinates around the incoming/outgoing sets 
$L_{\mp}:=\{\rho=0,\bbar{\xi}_0=\pm 1,\bbar{\eta}=0\}$ in ${^{\rm sc}S}^*\bbar{M}$.
We thus consider the manifold 
\[ [{^{\rm sc}S}^*\bbar{M}; L_\pm] \]
obtained by performing a radial blow-up of ${^{\rm sc}S}^*\bbar{M}$ at the incoming/outgoing submanifolds $L_\pm=\{ \bbar{\xi}_0=\mp 1, \rho=0\}\subset {^{\rm sc}S}^*\bbar{M}$. This is
obtained by replacing $L_\pm$ by the inward pointing spherical normal bundle $NL_\pm/\R^+$ of the submanifold $L_\pm$ of ${^{\rm sc}S}^*\bbar{M}$. The blow-down map $\beta:[{^{\rm sc}S}^*\bbar{M}; L_\pm]\to {^{\rm sc}S}^*\bbar{M}$ is the identity outside $NL_\pm/\R^+$ and is the natural 
projection $NL_\pm/\R^+\to L_\pm$ when restricted to $NL_\pm/\R^+$. A function on the blow-up is smooth if outside $NL_\pm/\R^+$ it is the pull back by $\beta$ of a smooth function, and 
if near $NL_\pm/\R^+$ it can be written as a smooth function in polar coordinates around $L_\pm$, i.e. it is a smooth function of the variables (here $|\bbar{\eta}|:=|\bbar{\eta}|_{h_\rho}$)
\begin{equation}\label{coordblowup} 
y,\, R:=\sqrt{|\bbar{\eta}|^2+\rho^2}, \, Y:=\frac{\bbar{\eta}}{\sqrt{|\bbar{\eta}|^2+\rho^2}}, \, Z:=\frac{\rho}{\sqrt{|\bbar{\eta}|^2+\rho^2}}.
\end{equation}
Note that in these coordinates, $\beta$ can be described near $NL^\pm/\R^+$ as 
\[\begin{split} 
\beta \left(y,\sqrt{|\bbar{\eta}|^2+\rho^2}, \frac{\bbar{\eta}}{\sqrt{|\bbar{\eta}|^2+\rho^2}},  \frac{\rho}{\sqrt{|\bbar{\eta}|^2+\rho^2}}\right)&=(\rho,y,\bbar{\xi}_0,\bbar{\eta})\\
&= (ZR,y,\mp \sqrt{1-R^4Z^2|Y|^2},YR).
\end{split}\]
We refer to the lecture notes \cite[Chapter 5]{Melrose:1996ij} for details about blow-ups. 
The manifold $[{^{\rm sc}S}^*\bbar{M}; L_\pm]$ is a smooth manifold with codimension $2$ corners. The boundary hypersurface of $[{^{\rm sc}S}^*\bbar{M}; L_\pm]$ 
corresponding to the pull-back of $\{\rho=0,\bbar{\eta}\not=0\}$ to $[{^{\rm sc}S}^*\bbar{M}; L_\pm]$ by the blow-down map $\beta:[{^{\rm sc}S}^*\bbar{M}; L_\pm]\to {^{\rm sc}S}^*\bbar{M}$ is denoted by ${\rm bf}$. In other words, ${\rm bf}=\hbox{cl}(\beta^{-1}(\{\rho=0,\bbar{\eta}\not=0\}))\subset [{^{\rm sc}S}^*\bbar{M}; L_\pm]$.
Then the boundary of $ [{^{\rm sc}S}^*\bbar{M}; L_\pm] \setminus {\rm bf}$
is the union of two boundary faces obtained from the blow-up, 
they are the half-sphere bundles $NL_\pm/\R^+$ and are defined by $R=0$ in the smooth coordinates \eqref{coordblowup} on $[{^{\rm sc}S}^*\bbar{M}; L_\pm]$, with interior denoted by $\pl_\pm S^*\bbar{M}$. These two new boundary faces are isomorphic to $T^*\pl\bbar{M}$: using that in the region $\pm\bbar{\xi}_0>0$, one has
$L_\mp=\{\rho=0,\bbar{\eta}=0\}$, 
 the projective coordinates    
\begin{equation}\label{projcoord} 
\rho, y,\eta:=\bbar{\eta}/\rho 
\end{equation}
are smooth coordinates in a neighborhood of the interior of these new boundary faces. Moreover $\pl_\pm S^*M=\{\rho=0\}$ in this neighborhood (in the projective coordinates): then $(y,\eta)$ restricted to $\pl_\pm S^*M$ provide a diffeomorphism with $T^*\pl M$. The variable $\bbar{\xi}_0$ is determined by $(\rho,y,\eta)$ near $\pl_\pm S^*M$ by the equation
\begin{equation}\label{bbarS^*M} 
\bbar{\xi}_0^2+ \rho^2 |\eta|_{h_\rho}^2=1.
\end{equation}

We then define the non-compact manifold with boundary
\[\bbar{S^*M}:= [{^{\rm sc}S}^*\bbar{M}; L_\pm] \setminus {\rm bf}. \]
The coordinates $(\rho,y,\xi_0,\eta)$ satisfying   
the condition \eqref{bbarS^*M} provide well-defined smooth coordinates on $\bbar{S^*M}$. We also note that $\rho$ is a smooth boundary defining function of $\pl_\pm S^*M$
in $\bbar{S^*M}$.
More informally, the space $\bbar{S^*M}$ corresponds to $S^*M$ with two copies of $T^*\pl\bbar{M}$ glued at $\{\rho=0,\bbar{\xi}_0=\pm 1\}$ in a way that smooth functions on $\bbar{S^*M}$ correspond to smooth functions 
$f\in C^\infty(S^*M)$ which can be written under the form $f(x)=F(\rho,y,\bbar{\xi}_0,\eta)$ near $\rho=0$ by using the coordinates \eqref{projcoord}, with $F\in C^\infty([0,\eps)\x \pl\bbar{M}\x[-1,1]\x T^*\pl\bbar{M})$. 
Recall that two normal forms with functions $\rho$ and $\hat\rho$ are related by $\hat{\rho}=\rho+\omega_0\rho^2+\mc{O}(\rho^3)$ for some $\omega_0\in C^\infty(\pl\bbar{M})$ in Lemma \ref{normalform}. Thus the induced coordinates $(\hat{\rho},\hat{y},\hat{\bbar{\xi}}_0,\hat{\eta})$ are related to $(\rho,y,\bbar{\xi}_0,\eta)$ by 
\begin{equation}\label{changeofcoord}
\hat{y}=y+\mc{O}(\rho), \quad \hat{\bbar{\xi}}_0=\bbar{\xi}_0+\mc{O}(\rho),\quad \hat{\eta}=\eta+d\omega_0+\mc{O}(\rho^2).
\end{equation}

 The canonical Liouville $1$-form on $T^*M$ is denoted $\alpha$, in local coordinates near 
$\pl \bbar{M}$ it is given by 
\[\alpha =\xi_0 d\rho+\eta.dy,\] 
and $d\alpha$ is the canonical Liouville symplectic form of $T^*M$.  
The geodesic vector field $X$ is the Hamilton vector field of the energy functional 
\[ 
H(x,\xi)=\frac{1}{2} |\xi|^2_{g_x}= \frac 12 (\bbar{\xi}_0^2 + \rho^2 |\eta|^2_{h_\rho})
\] 
As usual, we can restrict $\alpha$ to $S^*M$ as a contact form satisfying $\alpha(X)=1$ and $i_Xd\alpha=0$, so that $X$ is  the Reeb vector field of $\alpha$.  
A direct computation yields
\begin{equation}\label{formluaX}
X =   \rho^2 \xin  \p_\rho   +  \rho^2 \sum_{i,j}h^{ij}\eta_i
\p_{y_j} -  \big( \rho^2|\eta|_{h_\rho}^2 + \frac{1}{2}\rho^3 \p_\rho |\eta|^2_{h_\rho}\big) \rho\p_{\xin} -
\frac 12 \rho^2 \sum_{j}\p_{y_j}(|\eta|^2_{h_\rho}) \p_{\eta_j}.
\end{equation}
In particular, 
\begin{equation}
\label{bbar X}
\bbar{X}:=\rho^{-2}X
\end{equation}
 extends smoothly down to $\pl\bbar{S^*M}=\{\rho=0\}$ in $\bbar{S^*M}$ and 
is equal at $\{\rho=0\}$ to 
\begin{equation}\label{limitingcase}
\bbar{X}|_{\pl \bbar{S^*M}}=\xin  \p_\rho   +   \sum_{ij}h^{ij}\eta_i
\p_{y_j}-\frac{1}{2} \sum_{j}\p_{y_j}(|\eta|^2_{h_0}) \p_{\eta_j}=\xin  \p_\rho+Y
\end{equation} 
 where $Y$ is the Hamilton field of $\frac{1}{2}|\eta|^2_{h_0}$ on $(\pl \bbar{M},h_0)$. We deduce
\begin{lemma}\label{factorX}
The vector field $X$ extends smoothly to $\bbar{S^*M}$ and can be factorized under the form $X=\rho^2\bbar{X}$ with $\bbar{X}\in C^\infty(\bbar{S^*M};T\bbar{S^*M})$ a smooth vector field transverse to the boundary $\pl\bbar{S^*M}$. 
\end{lemma}
There is a natural volume form on $S^*M$, namely the Liouville measure 
given by 
\[\mu= \alpha\wedge (d\alpha)^{n-1}.\]
The boundary $\pl \bbar{S^*M}$ identifies to two copies of $T^*\pl \bbar{M}$ and is thus a natural symplectic manifold with the canonical Liouville form $\sum_{j}d\eta_j\wedge dy_j$.
\begin{lemma}\label{volform}
The forms $\rho^{2}\alpha$ and $d\alpha$ extend smoothly to $\bbar{S^*M}$ and $\iota^*_\pl d\alpha$ restricts to the canonical Liouville symplectic form of $\pl_\pm S^*M\simeq T^*\pl\bbar{M}$ if $\iota_\pl:\pl \bbar{S^*M}\to \bbar{S^*M}$ is the inclusion map. The volume form $\mu$ is such that $\rho^2\mu$ extends smoothly on $\bbar{S^*M}$ and $\iota_\pl^*i_X\mu$ is the canonical Liouville symplectic volume form on $\pl \bbar{S^*M}$.
\end{lemma}
\begin{proof}
We can write $\alpha=\bbar{\xi}_0\frac{d\rho}{\rho^2}+\sum_{j}\eta_jdy_j$ so $\rho^{2}\alpha$ extends smoothly to $\bbar{S^*M}$. Now $d\alpha=\rho^{-2}d\bbar{\xi}_0\wedge d\rho+\sum_{j}d\eta_j\wedge dy_j$ and differentiating $1=\bbar{\xi}_0^2 + \rho^2 |\eta|^2_{h_\rho}$ gives 
$\bbar{\xi}_0d\bbar{\xi}_0=-\rho |\eta|^2_{h_\rho} d\rho+\mc{O}(\rho^2)$ on $T(S^*M)$ thus $d\alpha$ extends smoothly to $\pl \bbar{S^*M}$ and $i^*_\pl d\alpha$ restricts to the canonical Liouville symplectic form of $\pl_\pm S^*M$. From the discussion above, $\rho^2\mu$ extends smoothly, and $i_X\mu=(d\alpha)^{n-1}$, thus pulls-back to the symplectic measure on $T^*(\pl M)\simeq \pl_\pm S^*M$.  
\end{proof}
Since also $\rho^2\mu=\bbar{\xi}_0d\rho\wedge (\sum_{j}\eta_jdy_j)^{n-1}+\mc{O}(\rho)$, 
we see that the orientation induced by $\rho^2\mu$ and $(\sum_{j}\eta_jdy_j)^{n-1}$ agree (resp. are opposite) on $\pl_+S^*M$ (resp. on $\pl_-S^*M)$.\\

\subsection{Connection map and Sasaki metric on $S^*M$}\label{connectionmap}
The manifold $T^*M$ has a natural metric structure called the Sasaki metric defined so that the horizontal space, defined through the Levi-Civita connection, and the vertical space are orthogonal;  we refer to \cite[Chapter 1.3]{Pa} for details. The Sasaki metric will be denoted by $G$. 
The projection on the base $\pi:T^*M\to M$ allows to define the vertical bundle
$\mc{V}:=\ker d\pi\subset T(T^*M)$ and there is a map called connection map $\mc{K}:T(T^*M)\to T^*M$ and $\mc{H}:=\ker \mc{K}$ is called the horizontal space if $z\in T^*M$. The maps
\[ \mc{K}: \mc{V}\to T^*M ,\quad d\pi: \mc{H}\to TM\]
are isomorphisms and we will call horizontal (resp. vertical) lift $\xi^h\in \mc{H}_{(x,\xi)}$ 
(resp. $\xi^v\in \mc{V}_{z}$) of a point $z=(x,\xi)\in T^*M$ the element $\xi^h\in\mc{H}_{z}$ 
(resp. $\xi^v\in \mc{V}_{z}$) so that $d\pi(\xi^h)=\xi$ (resp. $\mc{K}(\xi^v)=\xi$).
We have a similar decomposition $T(S^*M)=\mc{H}\oplus \mc{V}$ over $S^*M$ but $\mc{V}$ becomes $(n-1)$-dimensional.  We let $\mc{Z}\to S^*M$ be the bundle with fibers $\mc{Z}_{(x,\xi)}=\{ v\in T_xM: \xi(v)=0\}$; then 
the maps $d\pi|_{\mc{H}}: \mc{H}\cap \ker \alpha \to \mc{Z}$ and $\mc{K}|_{\mc{V}}:\mc{V}\to
\mc{Z}$ are isomorphisms (see \cite{PSU}). We will call horizontal lift $\xi^h\in \mc{H}_{(x,\xi)}$ of a point  $(x,\xi)\in S^*M$ the element $\xi^h\in\mc{H}_{(x,\xi)}$ so that $d\pi(\xi^h)=\xi$.
The vector field $X$ also  acts on  sections of $\mc{Z}$ by using parallel transport 
along geodesics of $g$: for $v\in \Gamma(\mc{Z})$, 
$v(\varphi_t(x,\xi))$ is a vector field along the geodesic $\pi(\varphi_t(x,\xi))$, 
and we define 
\[ Xv(x,\xi):=\nabla_{\pl_t}v(\varphi_t(x,\xi))|_{t=0}
\]
where $\nabla$ is the Levi-Civita connection pulled-back to $\mc{Z}$ over $S^*M$.

If an element $z=(x,\xi) \in T^*M$ is expressed as $x=(\rho,y)$ and 
$\xi = \bbar{\xi}_0 \frac{d\rho}{\rho^2} + \sum_j\eta_jdy_j=\xi_0d\rho + \sum_j\eta_jdy_j$, and if we use the coordinate on $T^*M$ given by $(\rho, y ,\xi_0, \eta)$ 
and the coordinate frame for $T(T^*M)$ 
is $\{ \partial_\rho, \partial_{y}, \partial_{\xi_0}, \partial_{\eta}\}$, 
the vertical lift of $\xi$ is given by 
\[ \xi^v=\xi_0\pl_{\xi_0}+\sum_j\eta_j\pl_{\eta_j} \in \mc{V}.\] 
Using the coordinates $(\rho, y ,\bbar{\xi}_0, \eta)$ we have $\xi^v=\bbar{\xi}_0\pl_{\bbar{\xi}_0}+\sum_j\eta_j\pl_{\eta_j}$.
The horizontal lift of a vector $Z=Z_0\pl_\rho +\sum_{j=1}^{n-1}Z_j\pl_{y_j} \in T_xM$ to a vector over the element $z = (x,\xi) \in T^*M$ is given by 
\[ Z^h=Z+\sum_{i,j,k}\xi_k\Gamma_{ij}^kZ_j\pl_{\xi_i} \in \mc{H}\]
where $\xi_i=\eta_i$ for $i\geq 1$ and $\Gamma_{ij}^k = dx_k( \nabla_{\partial_{x_i}}\partial_{x_j})$. In particular, some computations give 
\begin{equation}\label{horverlifts} 
\begin{gathered}
\pl_{\rho}^h=\pl_\rho  - 2 \rho^{-1} \bbar{\xi}_0\partial_{\bbar{\xi}_0}+\rho^{-1}L^{0}_x(\eta,\pl_\eta)\\
\pl_{y_j}^h=\pl_{y_j}+\rho^{-1}\bbar{\xi}_0J_x^{j}(\pl_\eta)+\rho K_x^{j}(\eta)\pl_{\bbar{\xi}_0}
+L^{j}(\eta,\pl_\eta)
\end{gathered}\end{equation}
where $L^{j}_x(a,b)$ are smooth in $x=(\rho,y)$ (down to $\rho=0$) 
and bilinear in $(a,b)$, $J_x(b)$, $K_x(a)$ are smooth in $(\rho,y)$ (down to $\rho=0$) and linear in $a$ and $b$.
The Sasaki metric satisfies 
\begin{equation}\label{sasakimet}\begin{split}
G(\xi^v,\xi^v)=& g^{-1}(\xi,\xi)=\rho^4\xi_0^2+\rho^2|\eta|^2_{h_\rho}=\bbar{\xi}_0^2+\rho^2|\eta|^2_{h_\rho},\\
G(Z^h,Z^h)=& g(Z,Z)=\rho^{-4}Z_0^2+\rho^{-2}\Big|\sum_{j=1}^{n-1}Z_j\pl_{y_j}\Big|^2_{h_\rho}=
\rho^{-4}Z_0^2+\rho^{-2}\sum_{ij=1}^{n-1}h_{ij} Z_jZ_i,\\
G(Z^h,\xi^v)=& 0,
\end{split}
\end{equation}
in particular $\mc{H}\oplus \mc{V}$ is an orthogonal decomposition for the metric $G$.
Using the local frame $\{\partial_\rho^h, \partial_{y_j}^h, 
\pl_{\bbar{\xi}_0}, \pl_{\eta_j}\}$ of $T(T^*M)$, the Sasaki metric can be expressed as 
\[G = \rho^{-4} \delta_\rho^2 + \rho^{-2} \sum_{jk}h_{kj} \delta_{y_j} \delta_{y_k} + 
d \bbar{\xi}_0^2 + \rho^2 \sum_{jk}h^{jk}d{\eta_j} d\eta_k \]
where $\{\delta_\rho,\delta_{y_j},d\bbar{\xi}_0,d\eta_j\}$ denotes the dual frame ($1$-forms on $T^*M$).

To compute the gradient $\nabla^{G}u$ of $u:S^*M\to \rr$ with respect to $G$, it suffices to compute $\til{\nabla}^G(u\circ p)|_{S^*M}$ where $p:T^*M\to S^*M$ is defined by $p(x,\xi)=(x,\xi/|\xi|)$ and $\til{\nabla}^G$ denotes the gradient with respect to $G$ in $T^*M$.
We get, writing $\til{u}=u\circ p$, 
\[\til{\nabla}^G\til{u}=\rho^4\pl_\rho^h\til{u}\pl_\rho^h+\rho^2\sum_{j,k}h^{kj}\pl_{y_j}^h\til{u}\pl_{y_k}^h+\pl_{\bbar{\xi}_0}\til{u}\pl_{\bbar{\xi}_0}+\rho^{-2}\sum_{k,j}h_{jk}\pl_{\eta_{k}}\til{u}\pl_{\eta_j} \]
with the condition $\bbar{\xi}_0\pl_{\bbar{\xi}_0}\til{u}+\sum_{j}\eta_j\pl_{\eta_j}\til{u}=0$. We deduce that 
\begin{equation}\label{normnablaGu} 
\| \nabla^G u\|^2_{G}=\rho^4 (\pl_\rho^hu)^2+\rho^2\sum_{k,j}h^{kj}\pl_{y_j}^hu \pl_{y_k}^hu+(\pl_{\bbar{\xi}_0}u)^2+ \rho^{-2}\sum_{k,j}h_{kj}\pl_{\eta_j}u \pl_{\eta_k}u
\end{equation}
We will also write $\nabla^hu$ and $\nabla^vu$ for the horizontal and vertical gradient, which are 
the orthogonal projections of $\nabla^Gu$ onto $\mc{H}$ and $\mc{V}$ with respect to $G$.

\subsection{Lift of tensors}
Denote by $C^\infty(\bbar{M}; S^m ({^{\rm sc}T}^*\bbar{M}))$ 
the smooth symmetric scattering tensors, i.e. smooth sections of the bundle of symmetric tensors in ${^{\rm sc}T}^*\bbar{M}$.
There exists a natural lift 
\begin{equation}\label{liftpim}
\pi_m^* : C^\infty(M; S^m T^*M) \to C^\infty(S^*M),  \quad 
\pi_m^* f (x,\xi) := f(x)(\otimes^m \xi^\sharp)
\end{equation} 
where $\xi^\sharp\in T_xM$ is the dual to $\xi\in T^*_xM$ with respect to $g$. Let us consider the action of $\pi^*_m$ on $C^\infty(\bbar{M}; S^m ({^{\rm sc}T}^*\bbar{M}))$. When viewed as a function on the smooth compact manifold with boundary ${^{\rm sc}S}^*\bbar{M}:=\{(x,\xi)\in {^{\rm sc}T}^*\bbar{M}; |\xi|_g=1\}$, it is direct to see that $\pi_m^*f\in C^\infty({^{\rm sc}S}^*\bbar{M})$. Now we also need to view $\pi_m^*$ as a function on $\bbar{S^*M}$.
This can be computed explicitly in the coordinate given by $\xi = \bar\xi_0 \frac{d\rho}{\rho^2} + \sum_j\eta_j d{y^j}$ where $\bbar{\xi}_0^2 + \rho^2 |\eta|_{h}^2 =1$: since 
 $\xi=\bbar{\xi}_0 \frac{d\rho}{\rho^2} + \sum_j\rho\eta_j \frac{dy_j}{\rho}$ and 
 $\xi^\sharp=\bbar{\xi}_0\rho^2\pl_\rho+\rho^2\sum_{ij}h^{ij}\eta_j\pl_{y_i}$
 we see that if $f\in C^\infty(\bbar{M}; S^m ({^{\rm sc}T}^*\bbar{M}))$ then for $\ell=0,\dots,m-1$
 \begin{equation}\label{liftest}
  \pi_m^*f \in C^\infty(\bbar{S^*M}), \quad \textrm{ and }
 \pi^*_mf\in \rho^{m-\ell}C^\infty(\bbar{S^*M}) \textrm{ on }\ker (i_{\rho^2\pl_\rho})^{\ell+1}
 \end{equation}
where $i_{\rho^2\pl_\rho}$ denotes interior product. On smooth sections of 
$S^m({^{\rm sc}T}^*\bbar{M})\cap \ker (i_{\rho^2\pl_\rho})^{\ell+1}$ we get
\begin{equation}\label{pimf}
|\pi_m^* f (x,\xi)| \leq  C\|f\|_{C^0(\bbar{M}; S^m ({^{\rm sc}T}^*\bbar{M}))}\rho^{m-\ell}|\eta|^{m-\ell}_{h_\rho}\leq C\|f\|_{C^0(\bbar{M}; S^m ({^{\rm sc}T}^*\bbar{M}))}
\end{equation}
for some uniform $C>0$, and so the same estimate thus holds on $C^0(\bbar{M}; S^m({^{\rm sc}T}^*\bbar{M}))$.
 Similarly, by direct computation using \eqref{normnablaGu} and \eqref{horverlifts},
 \begin{equation}
 \label{gradient of lift estimate}
 \|\nabla^G \pi^*_mf\|_G^2 \leq  C\|f\|^2_{C^1(\bbar{M}; S^m ({^{\rm sc}T}^*\bbar{M}))}| \rho\eta|_h^{2m-2-2\ell}\leq C\|f\|^2_{C^1(\bbar{M}; S^m ({^{\rm sc}T}^*\bbar{M}))}
 \end{equation} 
on smooth sections of $S^m({^{\rm sc}T}^*\bbar{M})\cap \ker (i_{\rho^2\pl_\rho})^{\ell+1}$ for some uniform $C>0$.
 
 It is well known that on a manifold with smooth metric, trace-free elements of $S^m T^*M$ lift to fiberwise homogeneous harmonic polynomials on $T^*M$ with respect to the vertical Laplacian $\Delta^v:=(\nabla^v)^*\nabla^v$  which then restrict to spherical harmonics on $S^*M$ (see \cite{Sh}) which we denote by $\Omega_m$.

\subsection{Integral curves of $X$ pointing towards infinity}
Recall that the rescaled vector field $\bbar{X}$ was defined as $\rho^{-2}X$ where $X$ is the geodesic vector field on $S^*M$ given by the expression (near $\rho=0$)
\eqref{formluaX}. 
We first prove a lemma about the flow $\varphi_t$ of $X$, which in turn is a rescaling of the flow of $\bbar{X}$.  In what follows, we let
\[W_\epsilon^\pm := \{ z = (\rho,y,\bbar{\xi}_0, \eta) \in \bbar{S^*M}; \rho \leq \epsilon, \pm \bbar{\xi}_0 \leq 0\}.\]

\begin{lemma}
\label{x asymptotic}
For $\epsilon>0$ small enough, there is $C>0$ such that for all $z \in W_\epsilon^+$ and $t\geq 0$
\[ \frac{\rho(z)}{1+ \rho(z) t}\leq \rho(\varphi_t(z)) \leq \frac{C\rho(z)}{1+ \rho(z) t},\ \  
|\eta(z)|_h e^{-C\rho(z)}\leq|\eta(\varphi_t(z))|_h \leq |\eta(z)|_h e^{C\rho(z)}.\]
Furthermore, 
\[0\leq1 - \bbar\xi_0(\varphi_t(z))^2 \leq  C\big(1 - \bbar\xi_0(z)^2\big)\Big(\frac{1}{1+ \rho(z) t}\Big)^{2} e^{C\rho(z)}.\]
The same holds for the reverse time flow $\varphi_{-t}(z)$ for $z\in W^-_{\epsilon}$.
\end{lemma}
\begin{proof}
We write $\varphi_t(z)=(\rho(t),y(t),\bbar{\xi}_0(t),\eta(t))$. From the form of $X$ in \eqref{formluaX}, we get the flow equations 
\[ \dot \rho = \rho^2\bar\xi_0,\ \ \dot{\overline\xi_0} =  - \rho^3 (|\eta|^2 +\frac{\rho}{2}\pl_\rho h(\eta,\eta)),\  \dot{y}_j(t)=\rho^2\sum_{i,j}h^{ij}\eta_i, \  \dot{\eta}_i=-\frac{1}{2}\rho^2\pl_{y_i} h(\eta,\eta)\]
We first note  that $\pl_t (\frac{1}{\rho(t)}) \leq -\bbar\xi_0 \leq 1$, which gives that
\begin{equation}
\label{x lower bound}
\rho(t) \geq \frac{\rho(0)}{1+\rho(0) t}.
\end{equation}
Furthermore for $\rho(0)\leq \epsilon$ small enough and $\bbar\xi_0(0)\leq 0$, one has that $\dot \rho \leq 0$ and $\bbar\xi_0(t) \leq 0$ for all $t>0$.
We now establish the upper bound on $|\eta(t)|$. We also get from the flow equation
that $\pl_t |\eta|^2=2\sum_{i,j}h^{ij}\dot{\eta}_i\eta_j+\sum_i\dot{y}_i\pl_{y_i}h(\eta,\eta)+\dot{\rho}\pl_\rho h(\eta,\eta)=\dot{\rho}\pl_\rho h(\eta,\eta)$ so 
\[ -C\rho^2 |\eta|^2\leq \pl_t |\eta|^2 \leq C \rho^2 |\eta|^2.\]
By Gr\"onwall's inequality one gets 
\begin{eqnarray} 
\label{prelim eta estimate}
 |\eta(0)| e^{-C\int_0^\infty \rho^2(s) ds}\leq |\eta(t)| \leq |\eta(0)| e^{C\int_0^\infty \rho^2(s) ds}.
\end{eqnarray}
We see here that it is natural to consider an estimate for the $L^2$ norm of $\rho(t)$. To this end we first observe that since $\rho(t)$ is decreasing and $\rho(0) \leq \epsilon$, thus for sufficiently small $\epsilon>0$ 
\[\rho(t)^2 |\eta(t)|^2 + \rho(t)^3 \pl_\rho h(\eta(t), \eta(t)) \geq\frac{\rho(t)^2}{2} |\eta(t)|^2 = \frac{1}{2}(1-|\bbar\xi_0(t)|^2).\] 
Combining this with \eqref{formluaX} and \eqref{x lower bound}, we have
\begin{eqnarray}
\label{xi derivative estimate}
\dot{\overline\xi_0} \leq  \frac{-1}{2}\frac{\rho(0)}{1+ \rho(0)t} (1-|\bbar\xi_0(t)|^2) &=&   \frac{-1}{2}\frac{\rho(0)}{1+ \rho(0)t}  (1+ \bar\xi_0)(1-\bar\xi_0)\\\nonumber &\leq&\frac{-1}{2}\frac{\rho(0)}{1+ \rho(0)t}  (1+ \bbar\xi_0) .\end{eqnarray}
The last inequality uses the fact that $\bbar\xi_0(t) \leq 0$ for all $t>0$. Applying Gr\"onwall again we see that $1+ \bar\xi_0(t) \leq \frac{(1+ \bar\xi_0(0))}{\sqrt{1+ \rho(0)t}}$ or 
\begin{eqnarray}
\label{xi bounds}
-1\leq \bar\xi_0(t) \leq -1 + \frac{1+ \bbar\xi_0(0)}{\sqrt{1 + \rho(0)t}}.
\end{eqnarray}
Inserting this into $\pl_t(\frac{1}{\rho}) = - \bbar\xi_0$ we see that 
\[\frac{1}{\rho(t)}\geq \frac{1}{\rho(0)} + t - 2 \frac{1+ \bbar\xi_0(0)}{\rho(0)} \sqrt{1+ \rho(0)t} + 2\frac{1+ \bar\xi_0(0)}{\rho(0)}\]
or
\begin{eqnarray}
\label{x bound}
 \rho(t) \leq \frac{\rho(0)}{1 + \rho(0)t (1-2\frac{(1+\bbar\xi_0(0))(\sqrt{1+ \rho(0)t}-1)}{\rho(0)t})}.\end{eqnarray}
This gives the integrability estimate
\[\int_0^\infty \rho(t)^2 dt \leq  \rho(0) \int_0^\infty \frac{1}{\Big(1 + s (1-2\frac{(1+\bbar\xi_0(0))(\sqrt{1+ s}-1)}{s})\Big)^2}ds.\]
Using that $\bbar\xi_0(0) \leq 0$ and $0\leq 1- 2\frac{\sqrt{1+s} -1}{s} \leq 1$ is strictly increasing we see that $\int_0^\infty \rho(t)^2 dt \leq C\rho(0)$ where $C$ is independent of the choice of 
$z\in W^+_\epsilon$. Going back to \eqref{prelim eta estimate} we deduce that
\[|\eta(t)| \leq |\eta(0)| e^{C\rho(0)}\ \ \forall t>0.\]
Furthermore if we let $c_0>0$ be the number such that $2\frac{(\sqrt{1+ s}-1)}{s} < 1/2$ for all $s> c_0$ we have from \eqref{x bound} that if $\rho(0)t \geq c_0$ then $\rho(t) \leq \frac{2\rho(0)}{1+ \rho(0)t}$. If $\rho(0)t \leq c_0$ then 
$$(1+c_0) \frac{\rho(0)}{1+ \rho(0)t} \geq \rho(0) \geq \rho(t)$$
since $\dot \rho(t) \leq 0$. This estimate along with \eqref{x lower bound} gives the estimates for $\rho(t)$. The estimate for $1- \bbar\xi_0(t)^2 = \rho(t)^2 |\eta(t)|^2$ follows immediately. \end{proof}

\subsection{Rescaled Dynamic}

If $\varphi_t$ is the flow generated by $X$, one can see that for $z_0 \in S^*M$,
$\bbar{\varphi}_\tau (z_0) := \varphi_{t(\tau)} (z_0)$ is the flow for the rescaled vector field $\bbar{X}$ by setting 
\begin{equation}
\label{tau change of var}
\tau(t) := \int_0^t \rho^2(\varphi_s(z_0))ds,\quad \quad  t(\tau) := \int_0^\tau \rho^{-2}(\bbar{\varphi}_s(z_0)) ds.
\end{equation} 
In what follows, for a fixed $z_0\in S^*M$ we will denote by $z(\tau) := \bbar{\varphi}_\tau(z_0)$ whereas $z(t) := \varphi_t(z_0)$ will denote the unscaled dynamic with $\tau$ and $t$ related by the change of variable given by \eqref{tau change of var}.
We define the rescaled arrival times for $z\in \bbar{S^*M}$
\begin{equation}\label{taupm}
\begin{gathered} 
\tau^+(z):= \sup\{ \tau\geq 0; \forall s<\tau, \bbar{\varphi}_s(z)\notin \pl_+S^*M\}\in [0,\infty],\\ 
\tau^-(z):=\inf\{ \tau\leq 0; \forall s>\tau, \bbar{\varphi}_s(z)\notin \pl_-S^*M\}\in [-\infty,0].
\end{gathered}\end{equation}
We notice that this quantity depends on the choice of $\rho$.

\subsection{Geodesics arbitrarily close to infinity}
We first discuss the exact conic case near infinity, i.e. when  $h=h_0$ is constant in $\rho$. 
In this case the vector field  $\bbar{X}$ is given by the formula \eqref{bbar X}. 
We then obtain
\begin{eqnarray}
\label{ddot rho(tau)}
\ddot \rho (\tau) = - \rho |\eta|_{h_0}^2 ,
\end{eqnarray}
and a direct computation shows that $|\eta(\tau)|_{h_0}=|\eta_0|_{h_0}$ is constant with $\dot \rho(0) = \xin(0) =  1,$ so we obtain
\[
\rho(\tau) = \frac{1}{|\eta_0|_{h_0}} \sin(\tau |\eta_0|_{h_0}).
\]
We see that for large $|\eta_0|$, the geodesic stays close to infinity.
We would like to assert that for large initial speed $|\eta_0|$ the rescaled dynamics aproximate the one given by a warped product metric. To this end we first prove a lemma about the rescaled dynamic for short time.

\begin{lemma}
\label{rho is tau}
Let $\epsilon >0$ be sufficiently small and suppose that $\bbar{\varphi}_\tau(z_0)$ is a flow line contained in $W^+_\epsilon \cup W^-_\epsilon$ for $z_0\in \pl_-S^*M$. There exist positive constants $c$ and $C$ such that $c\tau \leq \rho(\tau) \leq C \tau$ for all $\tau$ such that $\bbar{\varphi}_\tau(z_0) \in W^-_\epsilon$ and $c(\tau^+(z_0) - \tau) \leq \rho(\tau) \leq  C(\tau^+(z_0) - \tau)$ for all $\tau$ such that $\bbar{\varphi}_\tau(z_0) \in W^+_\epsilon$.
\end{lemma}
\begin{proof}
Let $\tau_0>0$ be such that $z_1:= \bbar{\varphi}_{\tau_0}(z_0) \in W^-_\epsilon$. Consider the unscaled backward flow $\varphi_{-t}(z_1)$ along the same trajectory. By \eqref{tau change of var} one has that $\tau_0 = \int_0^\infty \rho^2(\varphi_{-t}(z_1)) dt$, which by Lemma \ref{x asymptotic} can be estimated above and below by $c\rho(z_1)\leq \tau_0 \leq C \rho(z_1)$ and this completes the proof for the $W^+_\epsilon$ case. The $W^{-}_\epsilon$ case can be dealt with the same way.
\end{proof}

\begin{lemma}
\label{limiting dynamic}
There is $C>0$ such that for each $z_0=(y_0,\eta_0)\in\pl_-S^*M$, if $|\eta_0|$ is sufficiently large we have $\rho(\tau) = \tau + \mc{O}(\tau^3)$ for all $\tau \in [0,\tau^+(y_0,\eta_0)]$ where $|\eta_0| \tau^{+}(y_0,\eta_0) = \pi + \mc{O}(|\eta_0|^{-1})$, and
$$\sup\limits_{\tau \in [0,\tau^+(y_0,\eta_0)]} \rho(\tau) \leq \frac{C}{|\eta_0|},\quad  \sup\limits_{\tau \in [0,\tau^+(y_0,\eta_0)]}\Big|\frac{|\eta(\tau)|}{|\eta_0|} - 1\Big| \leq \frac{C}{|\eta_0|}.$$
\end{lemma}
\begin{proof} Let $z_0:=(y_0,\eta_0) \in \partial_- S^*M$. Applying Lemma \ref{x asymptotic} backwards one sees that $|\eta(\tau)| \leq C |\eta_0|$ for all $\tau>0$ such that $\bbar{\varphi}_\tau(z_0) \in W^-_\epsilon$. Since $\bbar{\xi}_0^2 + \rho^2 |\eta|_h^2 = 1$ along the trajectory, taking $|\eta_0|<\eps/C$, we see that the first time $\tau_{\max}$ 
so that $\bbar{\varphi}_{\tau}(z_0)\notin W_{\eps}^-$ for $\tau>\tau_{\max}$ needs to satisfy 
$\bbar{\xi}_0(\tau_{\max})=0$ and $\varphi_{\tau_{\max}}(z_0)\in W_\eps^+$. 
By Lemma \ref{x asymptotic} (and its proof), $\bbar{\varphi}_{\tau}(z_0)\in W_\eps^+$ for all $\tau>\tau_{\max}$ and $|\eta(\tau)|\leq C|\eta(\tau_{\max})|\leq C'|\eta_0|$ for some $C,C'$ uniform.
We also see that for $|\eta_0|$ sufficiently large, $c|\eta_0|^{-1}\leq \rho_{\max} \leq C|\eta_0|^{-1}$ where $\rho_{\max}=\rho(\tau_{\max})$ is the maximum value of $\rho(\tau)$ for $\tau \in [0,\tau^+(y_0,\eta_0)]$ and $c,C$ independent of $\eta_0$. 
By Lemma \ref{rho is tau} one has that
\[ c''|\eta_0|^{-1}\leq c'\rho_{\max} \leq \tau_{\max} \leq C'\rho_{\max} \leq C''|\eta_0|^{-1}.\]
for some positive $c',c'',C',C''$ independent of $\eta_0$.
The same argument yields that 
\begin{equation}\label{prelim tau+ bound}
\frac{c}{|\eta_0|} \leq \tau^+(y_0,\eta_0) \leq \frac{C}{|\eta_0|}
\end{equation}
for $|\eta_0|$ sufficiently large.
We introduce polar coordinates on $S^*M$ given by 
\begin{eqnarray}\label{polar coord}
\bbar{\xi}_0 = \cos\theta,\ \rho|\eta|_h = \sin\theta.
\end{eqnarray} 
Differentiating the cosine term we have (using the expression of $\bbar{X}$ from \eqref{formluaX})
\[-\Big( \rho |\eta|_h^2 + \frac{\rho^2}{2} \partial_\rho h(\eta,\eta)\Big)=\frac{d}{d\tau} \bbar{\xi}_0 = -\dot \theta \sin \theta = -\dot\theta \rho|\eta|\] 
which becomes 
\[\dot\theta(\tau) = |\eta|_h + \frac{\rho}{2} \partial_\rho h(\eta,\frac{\eta}{|\eta|}) = |\eta_0| + \int_0^\tau \partial_s |\eta(\bbar{\varphi}_s(z_0))|_h ds + \frac{\rho}{2} \partial_\rho h(\eta,\frac{\eta}{|\eta|}) .\]
Again from \eqref{formluaX} and $|\eta|^2_h=(1-\bbar{\xi}_0^2)\rho^{-2}$, it is direct to check that $\partial_s |\eta(\bbar{\varphi}_s(z_0)| = \frac{1}{2}\bbar{\xi}_0 \partial_\rho h(\eta,\frac{\eta}{|\eta|})$. Inserting this expression into the integral we obtain the double-sided bound on $\dot\theta$:
\[ |\eta_0| - C|\eta_0|\tau^+  \leq \dot\theta \leq |\eta_0| + C |\eta_0|\tau^+\]
where we have simplified the notation $\tau^+ := \tau^+(y_0,\eta_0)$. Going back to \eqref{polar coord} one sees that $z(\tau) \in \partial_- S^*M$ precisely when $\bbar{\xi}_0 = -1$ or $\theta = \pi$. Integrating the two-sided differential inequality and evaluating at $\tau = \tau^+$ we arrive at 
\[ |\eta_0| \tau^+ - C|\eta_0| (\tau^+)^2 \leq \theta(\tau^+) = \pi \leq  |\eta_0| \tau^+ + C|\eta_0| (\tau^+)^2 .\]
Dividing through by $|\eta_0|$ and applying \eqref{prelim tau+ bound} we see that 
\[\tau^+ \in \Bigg[\frac{\pi}{|\eta_0|} - \frac{C}{|\eta_0|^2},  \frac{\pi}{|\eta_0|} + \frac{C}{|\eta_0|^2}\Bigg].\] 
The estimate for $\Big|\frac{|\eta(\tau)|}{|\eta_0|} - 1\Big|$ follows from these estimates and Lemma \ref{x asymptotic}.\end{proof}
We would like to get more detailed description of the asymptotics of the dynamics as $|\eta_0|\to\infty$. To this end, fix $(y_0,\eta_0)\in \partial_-S^*M$ with $|\eta_0|_{h_0} = 1$ and consider the rescaled dynamics $\bbar{\varphi}_{\tau}(z_0):=(\rho(\tau),y(\tau),\bbar{\xi}_0(\tau),\eta(\tau))$ with initial condition $z_0:=(y_0,\epsilon^{-1}\eta_0)\in \partial_-S^*M$. It is convenient to rescale again and introduce variables 
\[\til{\rho} (s) := \epsilon^{-1}\rho (\epsilon s), \,\, \til{\xi}_0(s) := \bbar{\xi}_0(\epsilon s),\,\, 
 \til{\eta}(s) := \epsilon\eta(\epsilon s), \,\,  \til{y} (s) = y(\epsilon s).\]
In coordinates, we get equations 
\begin{equation}
\label{tilde dynamic}
\begin{gathered}
\dot{\til{\rho}} = \til{\xi}_0,  \quad \dot{\til{\xi_0}} = -\til{\rho} (|\til{\eta}|^2 + \frac{\epsilon\til{\rho}}{2} \big(\partial_\rho h_{\epsilon\til{\rho}}\big)(\til{\eta},\til{\eta}) ), \\
\dot{\til{y}}_j = \sum_{k}h^{jk}_{\epsilon\til{\rho}} \til{\eta}_k, 
\quad \dot{\til{\eta}}_j = -\frac{1}{2} \pl_{y_j}
|\til{\eta}|^2_{\eps \til{\rho}}.
\end{gathered}
\end{equation}
Note that using these relations we can derive a convenient representation for $|\til\eta|^2$. Indeed, differentiating the relation $|\til{\eta}|^2 = \frac{1- \til{\xi}_0^2}{\til{\rho}^2}$ we have using \eqref{tilde dynamic} that $\pl_s |\til{\eta}|^2 =   \epsilon \til{\xi}_0 \pl_\rho h(\til{\eta},\til{\eta})$. Therefore,
\begin{equation}
\label{|tilde eta|}
|\til{\eta}(s)|^2= 1 + \epsilon\int_0^s \til{\xi}_0 \pl_\rho h(\tilde\eta,\til{\eta}) dt
\end{equation}
is a way to keep track of the evolution of $|\til{\eta}|^2$ (and the same holds for $|\til{\eta}|$).
This expression also allows us to obtain a convenient representation for $\epsilon^{-1} \tau^+_g(y,\epsilon^{-1}\eta_0)$.
First, we consider the variable $\theta(\cdot): [0,\epsilon^{-1}\tau^+_g(y,\epsilon^{-1}\eta_0)] \to [0,\pi]$ defined by 
\[\til{\xi}_0(s) = \cos \theta(s),\ \quad  \til{\rho}(s) |\til{\eta}(s)|_h = \sin\theta(s).\] 
Observe that $\theta(0) = 0$ and $\theta(\epsilon^{-1}\tau^+_g(y,\epsilon^{-1}\eta_0)) = \pi$. Differentiating the second equation and using \eqref{tilde dynamic} we obtain the integral relation
\[\theta(s) = \int_0^s |\til{\eta}| + \frac{\epsilon \til{\rho}}{2} \pl_\rho h_{\epsilon\til{\rho}} 
(\til{\eta}, \frac{\til{\eta}}{|\til{\eta}|}) dt = s + \epsilon \int_0^s \Big( \int_0^t \frac{\til{\xi}_0}{2} \pl_\rho h_{\epsilon\til{\rho}}(\til{\eta}, \frac{\til{\eta}}{|\til{\eta}|}) du  
+ \frac{\til{\rho}}{2} \pl_\rho h_{\epsilon \til{\rho}}(\til{\eta}, \frac{\til{\eta}}{|\til{\eta}|})\Big)dt.\]
where we used \eqref{|tilde eta|} for the second identity.
Setting $s= \epsilon^{-1} \tau^+_g(y,\epsilon^{-1}\eta_0)$ we obtain
\begin{eqnarray}
\label{epsilon^{-1} tau^+_g(y,epsilon^-1eta_0)}
\quad \pi = \epsilon^{-1} \tau^+_g(y,\epsilon^{-1}\eta_0) + \epsilon \int_0^{ \epsilon^{-1} \tau^+_g(y,\epsilon^{-1}\eta_0)} \Big(\int_0^t \frac{\til{\xi}_0}{2} \pl_\rho h_{\epsilon\til{\rho}}
(\til{\eta}, \frac{\til{\eta}}{|\til{\eta}|}) du  + 
\frac{\til{\rho}}{2} \pl_\rho h_{\epsilon \til{\rho}}(\til{\eta}, \frac{\til{\eta}}{|\til{\eta}|})\Big)ds\quad
\end{eqnarray}
which implies 
\begin{equation}\label{taueps}
\tau^+_g(y,\epsilon^{-1}\eta_0)=\epsilon \pi+\mc{O}(\epsilon^2).
\end{equation}
We then obtain a description of the integral curves as $\epsilon\to 0$:
\begin{lemma}
\label{asymp of geod}
Fix $(y_0,\eta_0)\in \partial_- S^*M$ with $|\eta_0|_{h_0}=1$ and set $\tau_\epsilon = \tau^+_g(y_0, \epsilon^{-1} \eta_0)$. For $\epsilon >0$ sufficiently small, the solutions to \eqref{tilde dynamic} for 
$\epsilon s\in[0, \tau_\epsilon]$ are of the form
\[\til{\rho}(s) = \frac{ \tau_\epsilon}{\epsilon\pi} \sin(\frac{\epsilon\pi s}{\tau_\epsilon}) + 
r_\epsilon(s),\ \ \til{\xi}_0(s) =  \cos(\frac{ \epsilon\pi s}{\tau_\epsilon}) + q_\epsilon(s)\]
where $0=r_\epsilon(0) = \dot r_\epsilon(0) = \ddot r_\epsilon(0) = r_\epsilon(\epsilon^{-1}\tau_\epsilon) = \dot r_\epsilon(\epsilon^{-1}\tau_\epsilon) = \ddot r_\epsilon(\epsilon^{-1}\tau_\epsilon)$ and 
\begin{eqnarray}
\label{r_epsilon < epsilon}
\sup\limits_{s \in (0, \epsilon^{-1} \tau_\epsilon)} {|q_\epsilon(s)|}\leq C\epsilon,\ \ \sup\limits_{s \in (0, \epsilon^{-1} \tau_\epsilon)} \frac{|r_\epsilon(s)|} {\sin^3(\frac{\epsilon\pi s}{\tau_\epsilon})}  \leq C\epsilon.
\end{eqnarray}
\end{lemma}
\begin{proof}
We define $r_\eps$ and $q_\eps$ by the expression in the Lemma,
with boundary condition $r_\eps(0) = r_\eps(\epsilon^{-1}\tau_\epsilon) = q_\epsilon(0)=q_\epsilon(\epsilon^{-1}\tau_\epsilon)= 0$. 
Using \eqref{tilde dynamic} and \eqref{|tilde eta|}, we see that $r_\epsilon$ and $q_\epsilon$ must solve
\[\dot r_\epsilon = q_\epsilon,\ \dot q_\epsilon = - \epsilon \til{\rho} \int_0^s \til{\xi}_0 \pl_\rho h(\til{\eta}, \til{\eta}) dt - \epsilon \frac{\til{\rho}^2}{2} \pl_\rho h (\til{\eta},\til{\eta}) + (\frac{\pi\epsilon}{ \tau_\epsilon} - \frac{\tau_\epsilon}{\pi\epsilon})\sin(\frac{\pi \tau \epsilon}{ \tau_\epsilon}) - r_\epsilon.\] 
Note that due to \eqref{epsilon^{-1} tau^+_g(y,epsilon^-1eta_0)} we have that 
$(\frac{\pi\epsilon}{\tau_\epsilon} - \frac{\tau_\epsilon}{\pi\epsilon}) = c_\eps \epsilon$ where
\[c_\eps = \frac{\pi + \epsilon^{-1}\tau_\epsilon}{\pi\epsilon^{-1}\tau_\epsilon}  \int_0^{ \epsilon^{-1} \tau_\epsilon} \Big(\int_0^t \frac{\til{\xi}_0}{2} \pl_\rho h_{\epsilon\til{\rho}}(\tilde \eta, \frac{\tilde\eta}{|\tilde\eta|}) du  + \frac{\til{\rho}}{2} \pl_\rho h_{\epsilon \til{\rho}}(\til{\eta}, \frac{\til{\eta}}{|\til{\eta}|})\Big)dt=\mc{O}(1)\]
as $\epsilon\to 0$. So the system can be simplified to
\[\dot r_\epsilon = q_\epsilon,\,\, \ \dot q_\epsilon = -r_\epsilon + \epsilon k_\epsilon\]
with $k_\epsilon$ given by
\[k_\epsilon (s) = -  \til{\rho} \int_0^s \til{\xi}_0 \pl_\rho h(\til{\eta}, \til{\eta}) dt - 
\frac{\til{\rho}^2}{2} \pl_\rho h (\til{\eta},\til{\eta}) + c_\eps\sin(\frac{\pi s \epsilon}
{\tau_\epsilon}).\]
Solving the ODE systems with vanishing initial conditions one then gets the following representation for $r_\epsilon(s)$:
\[r_\epsilon(s) = \epsilon \Big(-\cos(s) \int_0^s \sin(t) k_\epsilon(t) dt + \sin(s) \int_0^s k_\epsilon(t) \cos(t) dt \Big).\]
Similarly solving the ODE with vanishing end conditions gives that 
\[r_\epsilon(s) = \epsilon \Big(-\cos(s) \int_s^{\epsilon^{-1}\tau_\epsilon} \sin(t) k_\epsilon(t) dt + \sin(s) \int_s^{\epsilon^{-1}\tau_\epsilon} k_\epsilon(t) \cos(t) dt \Big).\]
Note that $k_\epsilon(s)=\mc{O}(s)$ as $s\to 0$ and $k_\epsilon(s)=\mc{O}(|s- \epsilon^{-1}\tau_\epsilon|)$ as $s\to  \epsilon^{-1}\tau_\epsilon$.
From this we see that $r_\eps(s)=\mc{O}(\epsilon s^3)$ as $s\to 0$ and a similar estimate 
as $s\to \epsilon^{-1}\tau_\epsilon$, thus \eqref{r_epsilon < epsilon} is satisfied.
\end{proof}

Note that due to the smoothness of the tensor $h$ up to the boundary $\partial M$, as $\epsilon \to 0$ the equation in \eqref{tilde dynamic} for the $(\til{y}, \til{\eta})$ dynamic converges to the equation  
\[\dot{y}_j = \sum_{k}h^{jk}_0\eta_k,\ \ \  
\dot{\eta}_j = -\frac{1}{2}\pl_{y_j} |\eta|^2_{h_0}\]
where $h_0$ is the metric on $\pl\bbar{M}$. Note that this is nothing but the Hamilton flow equation of the Hamiltonian vector field $H_0$ of the Hamiltonian $\frac{1}{2}|\eta|^2_{h_0}$ on $T^*\pl\bbar{M}$.
In particular, if 
\[(\hat{y}(s),\hat{\eta}(s)):=e^{sH_0}(y_0,\eta_0)\in T^*\pl\bbar{M}\] 
is the integral curve solution of this equation
with initial condition $(\hat{y}(0),\hat{\eta}(0))=(y_0,\eta_0)$, we get from Duhamel formula
\begin{equation}\label{limitinggeodesics}
d_{h_0}((\til{y}(s),\til{\eta}(s)),(\hat{y}(s),\hat{\eta}(s)))=\mc{O}(s\epsilon)
\end{equation}
where we view $(\til{y}(s),\til{\eta}(s))$ as a curve in $T^*\pl\bbar{M}$ and $d_{h_0}$ is the distance for the Sasaki metric associated to $h_0$ in $T^*\pl\bbar{M}$ 
(or equivalently any other Riemannian metric).

\section{X-Ray transform}

We first define the X-ray transform on functions on $S^*M$, and this can be seen as some boundary value of a natural boundary value problem on $\bbar{S^*M}$.
Except in Subsection \ref{trappedsec}, we shall always assume that the metric $g$ is \emph{non-trapping} in the sense that $\tau^+(z)<\infty$ for all $z\in S^*M$, where $\tau^+$ is defined in \eqref{taupm}. Observe that due to \eqref{tau change of var} and the estimates for $\rho(t)$ in Lemma \ref{x asymptotic}, the condition $|\tau_\pm(z)|<\infty$ is equivalent to saying that  $\rho(\varphi_t(z))\to 0$ as $t\to \pm \infty$. 

\subsection{Boundary value problem and X-ray}
Let us consider the boundary value problem 
\[ -Xu=f , \quad u|_{\pl_+S^*M}=0\]
where $f\in C_c^\infty(S^*M)$. Assuming that the metric $g$ is non-trapping, there is a unique smooth solution $u\in C^\infty(S^*M)$ given by 
\[ u(z):= \int_{0}^\infty f(\varphi_t(z))dt.\]
This defines an operator $R_+: C_c^\infty(S^*M)\to C^\infty(S^*M)$ by setting 
$R_+f=\int_0^\infty f\circ \varphi_t \, dt$.
We would like to extend this to functions which are not compactly supported and to points in $z\in \overline{S^*M}$ (i.e. all the way up to the boundary). To this end, observe that by making a change of variable $t = t(\tau)$ of \eqref{tau change of var} one can define the forward/backward resolvent
\begin{equation}\label{Rf}
R_{+} f (z) := \int_{0}^{{\tau^+}(z)}  \frac{f(\bbar{\varphi}_\tau (z))}{\rho^2(\tau)} d\tau,\ \ R_{-}f(z) = -\int_{\tau^-(z)}^0 \frac{f(\bbar{\varphi}_\tau(z))}{\rho^2(\tau)}d\tau.
\end{equation} 
Note that this definition extends to functions $f \in \rho^2 C^\infty(\overline{S^*M})$. Writing $f = \rho^2\bbar{f}$ one gets 
$R f(z) :=\int_{0}^{{\tau^+}(z)}  \bbar{f}(\bbar{\varphi}_\tau (z)) d\tau.$
Again, due to the non-trapping assumption this is a smooth function in $S^*M$. Furthermore, since $\bbar{f}$ and $\bbar{\varphi}$ are both smooth all the way up to $\partial S^*M$, this extends to be a smooth function on $\overline{S^*M}$. 
\begin{definition}\label{defXray}
The X-ray transform of a function $f\in \rho^2 C^\infty(\bbar{S^*M})$ is defined by  
$If := R_+f|_{\partial_- S^*M}=\int_{0}^{{\tau^+}(z)}  \rho^{-2}(\tau)f(\bbar{\varphi}_\tau (z)) d\tau=\int_{-\infty}^\infty f(\varphi_t(z_0))$ where $z_0\in M$ is any point on the integral curve $\cup_{\tau\in (0,\tau^+(z))}\bbar{\varphi}_\tau (z)$. This does not depend on $\rho$.
\end{definition}
First, observe that 
\begin{equation}\label{kerIcobord} 
f\in \ker I\cap \rho^2C^\infty(\bbar{S^*M})\iff R_+f=R_-f.
\end{equation}
By \eqref{liftest}, for symmetric tensors $f\in \rho^k C^\infty(\overline M; S^m({}^{sc}T^*\bbar{M}))\cap \ker (\iota_{\rho^2\partial_\rho})^{\ell+1}$, 
one has $\pi_m^*f\in \rho^{k+m-\ell}C^\infty(\bbar{S^*M})$ and thus, for $k+m-\ell\geq 2$ 
\begin{equation}\label{defIm}
I_m:= I\pi_m^*:\rho^k C^\infty(\overline M; S^m({}^{sc}T^*\bbar{M}))\cap \ker (\iota_{\rho^2\partial_\rho})^{\ell+1}\to  C^\infty(\pl_-S^*M)
\end{equation}
is well-defined and continuous. We call it the X-ray transform on $m$-tensors.
In the next subsections we will study the kernel of $I_m$. 

\subsection{Choice of gauge}
Let $B_\pm: \bbar{S^*M}\to \pl_\pm S^*M$ be the endpoint map defined by 
\[ B_\pm(z):= \bbar{\varphi}_{\tau_\pm(z)}(z).\] 
where $\tau_\pm$ are defined in \eqref{taupm}.
Since $R_+Xf=-f+f\circ B_+$ for $f\in C^\infty(\bbar{S^*M})$, we see that  
\[ IXf= f|_{\pl_+S^*M}\circ B_+-f|_{\pl_-S^*M}\]
and in particular $IXf=0$ if $f\in \rho C^\infty(\bbar{S^*M})$. Now it is also a direct computation to
check that $X\pi_m^*f=\pi_{m+1}^*Df$ if $f\in C^\infty(M;S^m T^*M)$ and where 
$D:=\mc{S}\nabla$ is the symmetrised covariant derivative. As a consequence, 
 if $f \in \rho C^\infty(\bbar{M} ; S^{m-1}({}^{\rm sc}T^*\bbar{M}))$, then $I_mDf=0$.
Since $I_m$ has a natural kernel, we can always put a function $f$ in a certain gauge using that 
kernel.
\begin{lemma}
\label{kill transversal terms}
Let $f\in \rho^k C^\infty(\bbar{M} ; S^m ({}^{\rm sc} T^*\bbar{M}))$  with 
$k \geq 2$ and $m\geq 1$. There exists a tensor $u  \in \rho^{k-1}C^\infty(\bbar{M} ; 
S^{m-1}({}^{\rm sc}T^*\bbar{M}))$ such that $f - Du\in \rho^k C^\infty(\overline M ; S^m({}^{sc} T^*\bbar{M}))$ and $\iota_{\rho^2\partial_\rho}(f - Du) = 0$ near $\partial M$.
\end{lemma}
\begin{proof}
We proceed as in \cite{GGSU} and consider a collar neighbourhood $[0,\epsilon)_\rho \times \partial \bbar{M}\subset \bbar{M}$ and express $f$ in this neighbourhood as $f = \mc{S} (f_j \otimes (\frac{d\rho}{\rho^2})^{m-j})$ where $f_j\in \rho^k C^\infty( \bbar{M}; S^j ({}^{sc}T^* \bbar{M}))$ satisfies $\iota_{\rho^2\partial_\rho} f_j = 0$. We will proceed by induction on $j$.
By a direct computation
\begin{equation}
\label{nabla dx}
\nabla \frac{d\rho}{\rho^2} =\frac{1}{2} \partial_\rho h_\rho- \frac{h_\rho}{\rho} \in \rho 
C^\infty(\bbar{M}; S^2 ({}^{sc}T^*\bbar{M}))\cap \ker \iota_{\pl_\rho}.
\end{equation}
Similarly, if $\omega \in  \rho^k C^\infty(\bbar{M} ; {}^{sc}T^*\bbar{M})$ is tangential to $\partial \bbar{M}$ near $\rho= 0$ (i.e. in $\ker i_{\rho^2\pl_\rho}$) then
\begin{equation}
\label{nabla alpha}
\nabla\omega = d\rho \otimes \pl_\rho \omega + 2\mc{S}(\omega\otimes \frac{d\rho}{\rho}) 
+ \mc{S}(d\rho\otimes A\omega)+\nabla^h\omega 
\end{equation}
for some smooth endomorphism $A\in C^\infty(\bbar{M};{\rm End}(T^*\pl \bbar{M}))$ and $\nabla^h$ denotes the Levi-Civita connection of $h_\rho$ on $\pl\bbar{M}$.  
We begin by setting $q_0(\rho,y) := \int_0^\rho s^{-2} f_0(s,y)ds \in \rho^{k-1}C^\infty(\overline M)$ so that, using \eqref{nabla dx}, 
\[D(q_0 (\frac{d\rho}{\rho^2})^{m-1})(\rho^2\pl_\rho,\dots,\rho^2\pl_\rho) = f_0.\]
This means that $f - D(q_0(\frac{d\rho}{\rho^2})^{m-1}) = \sum\limits_{j = 1}^m \mc{S}(\til{f}_j \otimes (\frac{d\rho}{\rho^2})^{m-j})$ where $\til{f}_j \in \rho^k C^\infty(\overline M ; S^j({}^{sc} T^*\bbar{M}))$ satisfies $\iota_{\rho^2\pl_\rho} \tilde f_j = 0$ near $\pl \bbar{M}$, using again \eqref{nabla dx}.

To complete the induction, we show that if $j \in [1,m-1]$ and $\til{f}_j \in \rho^k 
C^\infty(\overline M ; S^j({}^{sc} T^*\bbar{M}))$ satisfies 
$\iota_{\rho^2\partial_\rho} \til{f}_j = 0$ near $\partial \bbar{M}$, we can construct $q_j \in \rho^{k-1}C^\infty(\overline M ; S^j({}^{sc} T^*\bbar{M}))$ which satisfies $\iota_{\rho^2\pl_\rho} q_j = 0$ near $\partial \bbar{M}$ and solves
\begin{equation}
\label{inductive solve}
D\mc{S}(q_j\otimes (\frac{d\rho}{\rho^2})^{m-j-1})  = \mc{S}(\til{f}_j \otimes (\frac{d\rho}{\rho^2})^{m-j}) + T
\end{equation}
where $T= \mc{S}(T_0 \otimes (\frac{d\rho}{\rho^2})^{m-j-1}) + \mc{S}(T_1\otimes (\frac{d\rho}{\rho^2})^{m-j-2})$ with $T_0 \in \rho^k C^\infty(\overline M ; S^{j+1}({}^{sc} T^*\bbar{M}))$ and $T_1 \in \rho^k C^\infty(\overline M ; S^{j+2}({}^{sc} T^*\bbar{M}))$ 
satisfying $\iota_{\rho^2\pl_\rho} T_0 = \iota_{\rho^2\pl_\rho} T_1 = 0$ (when $j = m-1$ the term involving $T_1$ vanishes).
To this end, using \eqref{nabla alpha} and \eqref{nabla dx}, we compute 
\[ D\mc{S}(q_j\otimes (\frac{d\rho}{\rho^2})^{m-j-1})=\mc{S}(
(\rho^2\pl_\rho q_j+2\rho q_j+\rho^2 Bq_j)\otimes (\frac{d\rho}{\rho^2})^{m-j})+T\]
with $T$ as above if $q_j \in \rho^{k-1}C^\infty(\overline M ; S^j({}^{sc} T^*\bbar{M}))\cap \ker \iota_{\rho^2\pl_\rho}$, for some smooth endomorphism $B\in C^\infty(\bbar{M}; {\rm End}(S^{j} T^*\pl\bbar{M}))$. It then suffices to solve the ODE 
\[ (\pl_\rho +B)(\rho^2q_j)=\til{f}_j\]
which has a solution $q_j\in \rho^{k-1}C^\infty(\overline M ; S^j({}^{sc} T^*\bbar{M}))$.
\end{proof}
Thus, in what follows, for studying the kernel of $I_m$, we will always assume that we are working with tensors in $\ker \iota_{\rho^2\pl_\rho}$.

\subsection{Boundary determinations}
The first property that we will take advantage of is that an element in the kernel of $I_m$ must have its leading behaviour at $\pl \bbar{M}$ satisfying some property on the boundary $(\pl \bbar{M},h_0)$. A consequence of  \eqref{limitinggeodesics} and Lemma \ref{asymp of geod} is the following 
\begin{lemma}
\label{converge to pi transform on boundary}
Let $f\in C^\infty(\overline M ; S^m({}^{sc}{T^*\bbar{M}}))$, then for all $k\geq 2$ and $z_0=(y_0,\eta_0)\in T^*\partial M$ with $|\eta_0|_{h_0} =1$, we have for $\gamma(\tau)=\pi_0(\bbar{\varphi}_\tau(z_0))$
\begin{eqnarray*}
&\lim\limits_{\epsilon\to 0}& \epsilon^{1-k} \int_0^{\tau^+(y_0,\epsilon^{-1}\eta_0)} \rho^{k-2}(\tau) f(\rho^2\dot\gamma(\tau), \dots, \rho^2\dot\gamma (\tau)) d\tau\\ 
&=& \sum\limits_{\ell= 0}^m \int_0^\pi \sin^{k+m-\ell-2}(s) \cos^\ell(s) \iota_{\partial M}^*(\underbrace{\iota_{\rho^2\partial_\rho} \dots\iota_{\rho^2\partial_\rho}}_{\ell} f)(\underbrace{\dot\alpha(s),\dots,\dot\alpha(s)}_{m-\ell})
\end{eqnarray*}
where $\alpha :[0,\pi]\to \partial M$ is the unit speed geodesic for the metric $h_0$ with initial condition $(\alpha(0),\dot\alpha(0)) = (y_0,\eta_0^\sharp)$ with $\eta_0^\sharp\in T_y\pl\bbar{M}$ the dual of $\eta_0$ by $h_0$.
\end{lemma}
\begin{proof}
In coordinates we can write $f = \sum\limits_{\ell =0}^m \sum_{I=(i_1,\dots,i_{(m-\ell)})}f_{\ell,i_{1},\dots, i_{m-\ell}} (\frac{d\rho}{\rho^2})^\ell \frac{dy_{i_1}}{\rho}\dots \frac{dy_{i_{m-\ell}}}{\rho}$ (for $i_k\in \{1,\dots,n-1\}$) and since 
\[\dot \gamma(\tau) = \bbar{\xi}_0(\tau) 
\pl_\rho +  \sum_{k}h^{jk}_\rho(\tau)\eta_j(\tau) \partial_{y_k}\] 
we get 
\begin{eqnarray*}
&& \epsilon^{1-k}\int_0^{\tau^+(y_0,\epsilon^{-1}\eta_0)} \rho^{k-2}(\tau)f(\rho^2\dot\gamma(\tau), \dots, \rho^2\dot\gamma (\tau)) d\tau \\&=& \epsilon^{1-k}\sum\limits_{\ell=0}^m\sum_{I} \int_0^{\tau^+(y_0,\epsilon^{-1}\eta_0)} \rho^{k+m-2-\ell}(\tau)  f_{\ell,i_1,\dots, i_{m-\ell}}(\tau) \bbar{\xi}_0^\ell(\tau) \eta^\sharp_{i_1}(\tau)\dots \eta^\sharp_{i_{m-\ell}}(\tau)d\tau,
\end{eqnarray*}
with $\eta^\sharp$ the dual of $\eta$ by $h_\rho$. 
Make the change of variable $s=\epsilon^{-1}\tau$ and take the limit as $\epsilon\to0$ using \eqref{limitinggeodesics}, Lemma \ref{asymp of geod} and \eqref{taueps}, we obtain the desired result.
\end{proof}
This leading behaviour can be viewed as a sort of ray-transform on the boundary but for geodesic segments of length $\pi$.
We recall that for a closed manifold $(N,h_0)$, the X-ray transform on $m$-symmetric tensors is called s-injective if for each $u\in C^\infty(N;S^mT^*N)$ satisfying 
$\int_\gamma \pi_m^*u=0$ for all closed geodesic $\gamma$ on $N$, then $u=Df$ for some 
$f\in C^\infty(N; S^{m-1}T^*N)$ (and for $m=0$ we ask that $u=0$).
From \cite{JoSB1} and Lemma \ref{converge to pi transform on boundary}, we obtain directly \begin{proposition}
\label{I0 boundary determination}
Suppose $(\partial\bbar{M}, h_0)$ is a Riemannian manifold of dimension $n\geq 2$ which either has injectivity radius strictly bigger than $\pi$, or its X-ray transform on functions is injective or it is a sphere of radius $r\notin \{1/k; k\in \mathbb{N}, k\geq 2\}$. If $f\in \rho^2 C^\infty (\overline M)$ satisfies $I_0f = 0$, then $f\in \rho^\infty C^\infty(\overline M)$.
\end{proposition}
\begin{proof}
We write $f = \rho^2\bbar {f}$ where $\bbar{f}\in C^\infty(\overline M)$ and we have that 
$\bbar{f} =\sum_{k=0}^N \bbar{f}_k\rho^k +\mc{O}(\rho^{N+1})$ with $\bbar{f}_k \in C^\infty(\partial \bbar{M})$. We proceed by induction: if $\bbar{f}_j=0$ for all $j\leq k-1$, we have by 
\[0 = |\eta_0|^{k-1}\int_{0}^{\tau^+(y_0,\eta_0)}  \bbar{f}(\pi_0(\varphi_\tau(y_0,\eta_0))) d\tau \to \int_0^\pi \sin^k(s)\bbar{f}_k(\alpha_{y_0,v}(s))ds\]
as $|\eta_0|\to \infty$,  with $v=\eta_0/|\eta_0|\in S^*_{y_0}\pl\bbar{M}$  and $\alpha_{y,v}(s)$ the geodesic for $h_0$ with initial condition 
$\alpha_{y,v}(0)=y$, $\dot{\alpha}_{y,v}(0)=v$. The proof of  \cite[Theorem 4.2]{JoSB1} implies $\bbar{f}_k = 0$\footnote{Although it is surprisingly not written in the statement of Theorem 4.2 of \cite{JoSB1}, the authors considered the case of the sphere with radius $1$ in the proof.}.
\end{proof}

\textbf{Remark}: For $1$-forms, it is possible to show, using similar arguments as in \cite{JoSB1} and Lemma \ref{kill transversal terms} that if $I_1 \omega = 0$ for $\omega \in \rho^2 C^\infty(\overline M; {}^{sc} T^*\bbar{M})$ then there is $u\in \rho C^\infty(\overline M)$ such that $\omega-du \in \rho^\infty C^\infty(\overline M; {}^{sc} T^*\bbar{M})$ under the assumption that $(\pl\bbar{M},h_0)$ is a closed Riemannnian manifold with injectivity radius ${\rm inj}(h_0)>2\pi$ and such that the X-ray transform on $1$-forms is s-injective. 

\subsection{Resolvent Estimates and Pestov Identity}
\subsubsection{Jacobi Fields Near Infinity}

The curvature estimates in Proposition \ref{curvature decay} allows us to deduce estimates on the Jacobi fields. 
\begin{lemma}
\label{R as map}
There is $C>0$ such that for each $z\in W^\pm_\epsilon$, if 
 $\gamma(t):=\pi_0(\varphi_t(z))$ and $J(t)$ is a smooth vector field along $\gamma$, 
  we  have the following estimate
\[|{ R}_{\gamma(t)}(\dot \gamma(t),J(t))\dot\gamma(t)|_g \leq C \rho^4(t) |J(t)|_g\]
for all $\pm t\geq 0$. 
\end{lemma}
\begin{proof}
We write $z=(\rho,\eta)\in SM$ with $\rho(z)\leq \eps$ and from the geodesic flow equation we have
$\dot\gamma(t)=\bbar{\xi}_0(t)\rho^2(t)\pl_\rho+\rho(t)^2V(t)$ where $V(t)=
\sum_{ij}h^{ij}\eta_i(t)\pl_{y_j}$ in local coordinates and 
with $|V(t)|_h\leq C$ and $d\rho(V(t))=0$ for some uniform $C>0$
using Lemma \ref{x asymptotic}.  By Proposition \ref{curvature decay} we have 
\[|g( R_{\gamma(t)}(\dot \gamma(t), J(t)) \dot\gamma(t),\rho^2\pl_\rho)|\leq C ( |J_0(t)|\rho^6(t)+ 
|J_1(t)|_g\rho^5(t))\]
where $J_0(t)= g(J(t),\rho^2\pl_\rho)$ and $J_1(t)=J(t)-J_0(t) \rho^2\pl_\rho$, while for each $Y(t)\in \ker d\rho$ with
$|Y(t)|_g=1$
\[ | g(R_{\gamma(t)}(\dot \gamma(t), J(t)) \dot\gamma(t),Y(t))|\leq C(|J_0(t)|\rho^5(t)
+|J_1(t)|_g \rho^4(t))\]
which ends the proof.
\end{proof}
\begin{lemma}
\label{jacobi growth}
There is $C>0$ such that for all $z\in W^+_\epsilon$ and $J(t)$ a Jacobi field along $\gamma(t) := \pi_0(\varphi_t(z))$, we have the following estimates:
\begin{equation}
\label{combined initial condition}
\begin{split}
|\dot J(t)|_g &\leq C|J(0)|_g \rho^3(0) + C|\dot J(0)|_g,\\
|J(t)|_g &\leq |J(0)|_g+C |J(0)|_g\rho^3(0)t + C|\dot J(0)|_g t. 
\end{split}
\end{equation}
\end{lemma}
\begin{proof}
Consider the function given by $F(t) := \langle R_{\gamma(t)}(\dot \gamma(t), J(t)) \dot\gamma(t), \frac{\dot J (t)}{|\dot J(t)|_g}\rangle_g$. We have $\frac{d}{dt} |\dot J(t)|_g = F(t)$ since $J$ is a Jacobi field. Define the barrier function $B(t)$ by 
\[\dot B(t) := |F(t)|,\ \ B(0) := |\dot J(0)|_g.\]
One then can easily see that 
\begin{eqnarray}
\label{B is bigger than J dot}
B(t)\geq |\dot J(0)|_g+ \int_0^t \pl_s |\dot J(s)|_gds = |\dot J(t)|_g
\end{eqnarray}
and is non-decreasing. Furthermore, due to the asymptotics of Lemma \ref{x asymptotic} and the estimates of Lemma \ref{R as map} we have
\begin{equation}
\label{F bound}
|F(t)| \leq C \Big(\frac{\rho(0)}{1+ \rho(0)t}\Big)^4 |J(t)|_g.
\end{equation}
Therefore one can write for $t>0$
\[\dot B(t) \leq C \Big(\frac{\rho(0)}{1+ \rho(0)t}\Big)^4 |J(t)|_g =  C \Big( \frac{\rho(0)}{1+ \rho(0)t}\Big)^4\Big( |J(0)|_g + \int_0^t \pl_s |J(s)|_g ds\Big)\]
Differentiating $|J(s)|_g$ and using \eqref{B is bigger than J dot} we obtain
\[\dot B(t) \leq  C\Big(\frac{\rho(0)}{1+ \rho(0)t}\Big)^4 \Big( |J(0)|_g + \int_0^t B(s) ds\Big).\]
And now using the fact that $B(s)$ is increasing we arrive at 
\[\dot B(t) \leq  C\Big(\frac{\rho(0)}{1+ \rho(0)t}\Big)^4 \Big( |J(0)|_g + t B(t)\Big).\]
If $B(0) = |\dot J(0)| = 0$ we have 
\begin{eqnarray*}
B(t) \leq C |J(0)|_g \int_0^t \Big( \frac{\rho(0)}{1+ \rho(0)s}\Big)^4 ds + 
C \int_0^t \Big(\frac{\rho(0)}{1+ \rho(0)s}\Big)^4 sB(s) ds.
\end{eqnarray*}
Then by Gr\"onwall's lemma,
\begin{eqnarray*}
 B(t) \leq |J(0)|_g \rho(0)^3 \int_0^{\rho(0)t} \Big(\frac{1}{1+s}\Big)^4 ds \exp \Big( C\rho(0)^2\int_0^{\rho(0)t} \Big(\frac{1}{1+ r}\Big)^4 r dr\Big) \leq C \rho(0)^3|J(0)|_g.
\end{eqnarray*}
So we obtain 
\begin{eqnarray}
\label{dot J(0) = 0}
|\dot J(t)|_g\leq C|J(0)|_g \rho(0)^3,\ \ \ |J(t)| \leq |J(0)|_g+C\rho(0)^3 |J(0)|_gt
\end{eqnarray}
if $\dot J (0) = 0$. Suppose now $J(0) = 0$. Then $\dot B(t) \leq C \Big( \frac{\rho(0)}{1 + \rho(0)t}\Big)^4 t B(t)$ and we obtain
\begin{equation}\label{J(0) = 0}
\begin{split}
|\dot J(t)|_g &\leq  B(t) \leq |\dot J(0)|_g \exp\Big(C \rho^2(0)\int_0^{\rho(0)t} \Big( \frac{1}{1 + s}\Big)^4sds\Big)\leq C|\dot{J}(0)|_g ,\\ 
|J(t)|_g &\leq   |\dot J(0)|_g \int_0^t \exp\Big( C \rho^2(0) \int_0^{\rho(0)s} \Big(\frac{1}{1+u} \Big)^4 udu\Big) ds\leq C|\dot{J}(0)|_gt
\end{split}
\end{equation}
Consequently combining \eqref{dot J(0) = 0} and \eqref{J(0) = 0} we have the desired estimates.
\end{proof}
\subsubsection{Resolvent mapping properties}
We next describe the solution of $Xu=f$ with boundary conditions $u|_{\pl_\pm S^*M}=0$ when $f$ are symmetric tensors.
\begin{lemma}
\label{resolvent estimate}
Let $f \in \rho^k C^\infty(\overline M; S^m ({}^{\rm sc}T^*\overline M))$ satisfy 
$\iota_{\rho^2\pl_\rho}f = 0$ near $\pl \bbar{M}$. If $m \geq 1$, let $k+m >3$ and define $u_{\pm} := R_{ \pm} \pi_m^*f $ as in \eqref{Rf}. There is $C>0$ such that for all $z = (\rho,y,\bbar{\xi}_0,\eta)\in W^{\pm}_\epsilon$, 
\begin{equation}\label{estimateupm}
\begin{gathered}
|u_{\pm}(z)| \leq C \rho^{k-1} |\rho\eta|^m_{h_\rho},\\ 
\| \nabla^v u_{\pm} (z) \|_G  \leq C \rho^{k-2} |\rho\eta|_{h_\rho}^{m-1},\
 \|\nabla^h u_{ \pm}(z)\|_G \leq C \rho^{k-1}  |\rho\eta|_{h_\rho}^{m-1}.
 \end{gathered}\end{equation} 
The resolvent thus satisfies the estimate
\[|u_{\pm}(z)|+\|\nabla^v u_{ \pm} (z) \|_G+ \|\nabla^h u_{ \pm}(z)\|_G\leq C \rho^{k-2}.\]
If $m=0$, one gets the same estimates but with $\rho^{k-1}$ on the right hand side.
Furthermore, if $f \in \rho^\infty C^\infty(\overline M; S^m ({}^{sc}T^*\overline M))$ then 
$u_{\pm}$ vanishes to infinite order at $\pl_{\pm}S^*M$.
\end{lemma}

\begin{proof}
We only do this for $u = u_+$ as the $u_-$ case is exactly the same. By definition one writes in using the decomposition $\xi=\bbar{\xi}_0d\rho/\rho^2+\eta$ 
of cotangent vectors near $\pl\bbar{S^*M}$ (with $\iota_{\pl_\rho} \eta=0$)
\begin{equation}
\label{explicit resolvent}
u(\rho,y,\bbar{\xi}_0,\eta) = \int_0^\infty f_{\gamma(t)}(\eta^\sharp (t),\dots,\eta^\sharp(t))dt
\end{equation}
where $(\gamma(t),\bbar{\xi}_0(t), \eta(t))\in S^*M$ is the geodesic with initial condition $(\rho,y,\bar\xi_0,\eta)\in S^*M$ and with $h_{\rho}(\eta^\sharp(t),\cdot)/\rho^2(t)=\eta(t)$. The first inequality in \eqref{estimateupm} then follows from Lemma \ref{x asymptotic} and the definition of ${}^{\rm sc}T^*\overline M$.

For the estimates on the derivatives, we see from computing using the chain rule that 
\[\|\nabla^h u(z)\|_G \leq \sup\limits_{\|V\| = 1} \int_0^\infty \|d(\pi_m^*f)_{\varphi_t(z)}\|_G\|(J_{(V,0)}(t), \dot J_{(V,0)}(t)))\|_Gdt\]
where $J_{V,W}(t)$ is the Jacobi field with initial condition $J_{V,W}(0)=V$ and $\dot{J}_{V,W}(0)=W$.
Using \eqref{gradient of lift estimate}, we have
\[
\|\nabla^h u(z)\|_G \leq  C \sup\limits_{\|V\| = 1} \int_0^\infty  \rho(t)^{k}  |\rho(t)\eta(t)|_h^{(m-1)}\|(J_{(V,0)}(t), \dot J_{(V,0)}(t))\|_Gdt
\]
and applying the estimates of Lemma \ref{x asymptotic} and \eqref{combined initial condition} we obtain 
\[\begin{split}
\|\nabla^h u(z)\|_G \leq & C \rho^{k+ m-1} \sup\limits_{\|V\| = 1} \int_0^\infty  \Big(\frac{1}{1 + \rho(0)t}\Big)^{k + m-1}   |\eta(0)|_h^{(m-1)}\|(J_{(V,0)}(t), \dot J_{(V,0)}(t))\|_Gdt \\
\leq &  C\rho^{k-1}|\rho\eta|_h^{m-1}.
\end{split}\]
For obtaining the analogous estimate for $\|\overset v \nabla u(z)\|_g$, we repeat the same process and get
\[\begin{split}
\|\nabla^v u(z)\|_G \leq & C\rho^{k}|\rho\eta|_h^{m-1}\sup\limits_{\|V\| = 1} \int_0^\infty  \Big(\frac{1}{1 + \rho(0)t}\Big)^{k+ m-1}\|(J_{(0,V)}(t), \dot J_{(0,V)}(t))\|_Gdt \\
\leq & C\rho^{k-2}|\rho\eta|_h^{m-1}.
\end{split}\] 
The case $m=0$ is similar with an improvement of one power of $\rho$. The last statement is a direct consequence of the expression $u=R_+f$
 in terms of the smooth flows $\bbar{\varphi}_\tau$ on $\bbar{S^*M}$.
\end{proof}
We then derive the following 
\begin{corollary}\label{primitiveforkerI}
Let $f \in \rho^k C^\infty(\overline M; S^m ({}^{\rm sc}T^*\overline M))$ satisfy 
$\iota_{\rho^2\pl_\rho}f = 0$ near $\pl \bbar{M}$. If $m \geq 1$, we also assume 
$k+m >3$. If $I_mf=0$, there exists $u\in C^\infty(S^*M)$ satisfying $Xu=f$ and the bounds near 
$\rho=0$
\[|u(z)|+\|\nabla^v u (z) \|_G+ \|\nabla^h u(z)\|_G\leq C \rho(z)^{k-1}.\] 
\end{corollary}
\begin{proof}
The condition $I_mf=0$ is equivalent to $R_+\pi_m^*f=R_-\pi^*_mf$ by \eqref{kerIcobord} and thus it suffices to set $u=R_+\pi_m^*f$ and apply Lemma \ref{resolvent estimate}.
\end{proof}
\subsubsection{Case with trapping}\label{trappedsec}
We briefly discuss the case where the curvature of $g$ is negative, in which case the trapped set 
\[ K:= \{z\in S^*M \,|\, \tau^+(z)=+\infty, \tau^-(z)=-\infty\}\]
is a hyperbolic set in the sense of Anosov. The dynamics and resolvent are constructed in \cite{DyGu} and the application to X-ray is done for the compact setting in \cite{Gui}. In that case the trapped set has measure $0$ and the set of $z\in \pl_-S^*M$ for which $\tau^+(z)<\infty$ is an open set of full measure and we will say that $I_mf=0$ if $I_mf(z)=0$ for all such $z\in \pl_-S^*M$.
Using that the trapped set is contained in a strictly convex compact region $\{\rho\geq \eps\}\subset S^*M$  for some $\eps>0$ small, it is straightforward to apply the analysis of \cite{DyGu,Gui} in our setting. 
The argument are mutatis mutandis the same as in the asymptotically hyperbolic case discussed in \cite[Proposition 3.11 and Lemma 3.12]{GGSU}: combined with the analysis near $\pl\bbar{M}$ done to prove Lemma \ref{resolvent estimate}, we obtain
\begin{lemma}\label{trapped}
Assume that $g$ is asymptotically conic with negative curvature, but with a non-trivial trapped set. Then the conclusion of Corollary \ref{primitiveforkerI} hold true exactly as in the non-trapping case.
\end{lemma} 

\subsubsection{Pestov identities}
We next give the Pestov identity for tensors. Let $M_\epsilon:=\{z\in M; \rho(z)\geq \eps\}$, and notice that $M_\eps$ is strictly convex for small $\eps>0$. 
\begin{proposition}
\label{pestov}
Let $f\in \rho^k C^\infty(\overline M ; S^m ({}^{\rm sc}T^*\bbar{M}))$ which satisfies $\iota_{\rho^2\partial_\rho} f = 0$ near $\pl \bbar{M}$ and $I_mf=0$. Assume 
$k > \frac{n}{2}+1$ and set $u := R_+\pi_m^*f$ defined by \eqref{Rf}. The following identity holds
\[ \| \nabla^v \pi_m^* f\|_{L^2}^2 - \| X \nabla^v u\|_{L^2}^2 = (n-1)\| \pi_m^* f\|_{L^2}^2 - \langle \mc{R} \nabla^v u, \nabla^v u\rangle.\]
where $\mc{R}: \mc{Z}\to \mc{Z}$ is defined by $\mc{R}_{(x,\xi)}Z:=R_x(Z,\xi^\sharp)\xi^\sharp$ with $R$ the Riemann curvature tensor.
In general, for all $ u\in \rho^\infty C^\infty(^{\rm sc}S^*M)$ one has 
\[ \| \nabla^v Xu\|^2 - \| X\nabla^v u\|^2 = (n-1)\| Xu\|^2 - \langle {\mathcal R} \nabla^v u, \nabla^v u\rangle_{L^2(S^*M)}.\]
\end{proposition}
\begin{proof}
We will only prove the first identity since the same argument applies to the second one. We use 
the computation of \cite[Proof of Theorem 1]{GGSU} and have for $u\in C^\infty(S^*M)$
\[\begin{split} 
\| \nabla^v Xu\|_{L^2(S^*M_\eps)}^2 - \| X \nabla^v u\|_{L^2(S^*M_\eps)}^2 = & (n-1)\| Xu\|_{L^2(S^*M_\eps)}^2 - \langle {\mathcal R} \nabla^v u, \nabla^v u\rangle_{L^2(S^*M_\eps)}\\
& +\int_{\pl S^*M_\eps}\big(\cjg \nabla^vu,\nabla^hu\cjd-(n-1)uXu\big)\mu_\eps
\end{split}\]
where $\mu_\eps=\iota^*_{\{\rho=\eps\}}\iota_X\mu$ and $\cjg \nabla^vu,\nabla^hu\cjd$ is understood as $g(\mc{K}\nabla^vu,d\pi \nabla^hu)$.

We now argue that when $\epsilon \to 0$ the boundary terms vanish and the terms involving interior integrals converge to the integral on all of $S^*M$. 
Treating the expression term by term we see first using \eqref{pimf} that 
\[\| \pi_m^*f\|_{L^2(S^*M_\eps)}^2 \leq C\int_{\rho \geq \epsilon}\rho^{2k} {\rm dvol}_g\leq C'\]
if $k>n/2$ and $\| \pi_m^*f\|_{L^2(S^*M_\eps)}^2\to \| \pi_m^*f\|_{L^2(S^*M)}^2$ as $\eps\to 0$.
For the term involving $\langle {\mathcal R}^v\nabla u,\nabla^v u\rangle_{L^2(S^*M_\eps)}$, this can be estimated by using the pointwise estimate
$|g {\mathcal R} \nabla^v u, \nabla^v u)| \leq \|\nabla^v u\|^2_G| K(\xi^\#,\nabla^v u)|$:
using Lemma \ref{resolvent estimate} and Proposition \ref{curvature decay}, we obtain  $
|g({\mathcal R} \overset{v}\nabla u, \overset{v}\nabla u)| \leq  C\rho^{2k-2} |\rho \eta|_{h}^{2(m-1)}$ if $m\geq 1$ and $|g( {\mathcal R} \nabla^v u, \nabla^v u)| \leq \rho^{2k}$ if $m = 0$.
If $k>n/2+1$, we get as $\epsilon\to 0$ that
\[\langle {\mathcal R} \nabla^v u, \nabla^v u\rangle_{L^2(S^*M_\eps)}\to \langle {\mathcal R} \nabla^v u, \nabla^v u\rangle_{L^2(S^*M )}.\]
We now look at  $\| \nabla^v Xu\|_{L^2(S^*M_\eps)}^2$ (recall $Xu=\pi_m^*f$): this is equal to 
$\| \Delta^v\pi_m^*f\|_{L^2(S^*M_\eps)}^2$ where $\Delta^v$ is the vertical Laplacian in the fiber,
but $f=\sum_{2j\leq m}\mc{S}(\otimes^j g\otimes f_j)$ for some trace-free $f_j\in \rho^k C^\infty(\overline M ; S^{m-2j} ({}^{\rm sc}T^*\bbar{M}))$, which then satisfy 
\[\Delta^v\pi_{m-2j}^*f_j=(m-2j)(m-2j+n-2)\pi_{m-2j}^* f_j.\] 
This implies that $\| \nabla^v Xu\|_{L^2(S^*M_\eps)}^2\to \| \nabla^v Xu\|_{L^2(S^*M)}^2$ as $\eps\to 0$ if $k>n/2$.
For the $ \| X \nabla^v u\|_{L^2}^2$ term we use the identity $X \nabla^v  u = \nabla^v\pi^*_mf -  \nabla^h u$
in conjunction with Lemma \ref{resolvent estimate} to get $|X\nabla^v  u|_G^2 \leq 
C \rho^{2k-2}$, thus $ \| X \nabla^v u\|_{L^2(S^*M_\eps)}^2\to \| X \nabla^v u\|_{L^2(S^*M)}^2$
as $\eps \to 0$ if $k>n/2+1$.

We conclude with the boundary terms: by Lemma \ref{resolvent estimate} we have at $\pl S^*M_\eps$
\[ |g(\nabla^vu,\nabla^hu)|+|\pi_m^*f|.|u| \leq C\eps^{2k-3} + C\eps^{2k-1}
\]
thus for small $\eps$
\[ \int_{\pl S^*M_\eps}(|g(\nabla^vu,\nabla^hu)|+|\pi_m^*f|.|u|)\mu_\eps\leq 
C\eps^{2k-2-n}\]
which converges to $0$ as $\eps\to 0$.
This ends the proof for the first statement. The second statement of the Lemma goes the same way, where it suffices to use the fast vanishing of $u$ as $\rho\to 0$.
 \end{proof}
\noindent\textbf{Remark:} Note that in the case when $m\geq 1$ and the dimension $n\geq 2$ then the condition $k > 3-m$ needed in Lemma \ref{resolvent estimate} is always valid if we assume as in the Proposition that $k > \frac{n}{2}+1$. 

\subsection{Injectivity of X-ray transforms}
We begin with a Carleman estimate which will be useful in proving injectivity of $I_1$ (here $\cjg t\cjd:=(1+t^2)^{1/2}$ and $W^{2,\infty}(\RR;\RR^{n-1})$ is the sobolev space of $\RR^{n-1}$ valued bounded functions with two derivatives bounded) 
\begin{lemma}
\label{carleman}
Let $R \in \langle t\rangle^{-2}L^\infty(\RR; {\rm End}(\RR^{n-1}))$, then there is $C>0$ such that for all $U\in \langle t\rangle^{-\infty} W^{2,\infty}(\RR;\RR^{n-1})$ one has the following estimate for $N$ sufficiently large:
\[\|e^{N \log (1+t^2)} \partial_t^2 e^{-N \log (1+t^2)}U + R(t)U\|_{L^2}^2 \geq N\|\dot U\|_{L^2}^2+ CN^2\|(1+t^2)^{-1}U\|_{L^2}^2.\]
\end{lemma}
\begin{proof}
A direct computation using integration by parts give
\begin{eqnarray}
\label{exact identity}
\|e^{N\log (1+t^2)} \partial_t e^{-N\log (1+t^2)} U \|_{L^2}^2 = \int_\rr |\dot U|^2 dt + 2N\int_\rr \frac{1+ (2N-1)t^2}{(1+t^2)^2} |U|^2 dt.
\end{eqnarray}
So for $N\geq2$ we yield  that $\|e^{N\log (1+t^2)} \partial_t e^{-N\log (1+t^2)} U \|^2 \geq \int \frac{N}{1+t^2}|U|^2$. Using this inequality in conjunction with \eqref{exact identity} we have
\begin{eqnarray*}
\nonumber\|e^{N\log (1+t^2)} \partial_t^2 e^{-N\log (1+t^2)} U \|_{L^2}^2 &\geq& N\int (1+t^2)^{-1} \big|e^{N\log (1+t^2)} \partial_te^{-N\log (1+t^2)} U\big|^2\\
&\geq & N \int \frac{|\dot U|^2}{1+t^2}+2N \int_\R \frac{1+(2N-3)t^2}{(1+t^2)^3}|U|^2dt\\
& \geq & N \int \frac{|\dot U|^2}{1+t^2}+2N \int_\R \frac{1}{(1+t^2)^2}|U|^2dt
\end{eqnarray*}
We see that for $N$ sufficiently large the potential term $RU$ can be absorbed to the right hand  side (by the last term $2N\int \cjg t\cjd^{-4}|U|^2dt$) to obtain the desired inequality.
\end{proof}
\begin{lemma}
\label{no vanishing jacobi fields}
Let $\gamma$ be a complete geodesic curve on $M$ with $\lim\limits_{t\to\pm \infty} \gamma(t) \in \partial M$. If $J$ is a Jacobi field along $\gamma$ satisfying $\|J(t)\|_g \leq C\langle t\rangle^{-\alpha}$ for some $\alpha >0$ and $g(J,\dot{\gamma})=0$, then $J = 0$.
\end{lemma}
\begin{proof}
Let $\{\dot{\gamma}, Y_1,\dots, Y_{n-1}\}$ be a unitary orthonormal frame which is parallel along $\gamma$, and recall that 
$|\dot{\gamma}(t)\pm\rho^2(\gamma(t))\pl_\rho|_g=\mc{O}(\rho(\gamma(t)))$ as $t\to \pm \infty$ by   using Lemma \ref{x asymptotic}, so $(d\rho/\rho^2)(Y_j(t))=\mc{O}(\rho(\gamma(t)))$.
We write $J = \sum U_\ell Y_\ell$, the equation 
$\ddot J = R(\dot\gamma, J)\dot \gamma$ becomes $\ddot U(t) + R(t) U(t) = 0$ with $U(t) \in \langle t\rangle^{-\alpha} L^\infty(\RR; \RR^{n-1})$ and some $R\in \langle t\rangle^{-4}L^\infty(\rr; 
{\rm End}(\rr^{n-1}))$ by using Lemma \ref{R as map}. We first show that $U(t), \dot U(t) \in \langle t\rangle^{-\infty} L^\infty(\rr; \RR^{n-1})$. 
Indeed, we write for $t<T$,
\[\dot U (t)  = - \int_t^T \ddot U (s)ds + \dot U(T) = \int_t^TR(s) U(s)ds + \dot U(T).\] 
This implies by integrating
\[|\dot U(T)| t \leq |U(t)| + |U(0)| + C\int_0^t  \int_r^\infty   \langle s\rangle^{-(4+\alpha)} dsdr.\]
Taking the $\limsup$ as $T \to \infty$ we obtain
\[\limsup_{T\to\infty}|\dot U(T)| t \leq |U(t)| + |U(0)| + C\int_0^t  \int_r^\infty   \langle s\rangle^{-(3+\alpha)} dsdr.\]
Taking $t\to +\infty$ and using the fact that $U\in L^\infty(\RR)$ we get that $\limsup\limits_{T\to\infty} |\dot U(T)| = 0$. Therefore, we obtain 
\[|\dot U (t) | = \int_t^\infty |R(s) U(s)|ds  \leq C \int_t^\infty \langle s \rangle^{-(4+\alpha)}ds \leq C \langle t\rangle^{-(3+\alpha)} \]
as $t\to +\infty$. Same estimate holds for when $t\to -\infty$ by the analogous argument. Now write
\[|U(t)| \leq \int_t^\infty |\dot U(s)| ds \leq C \langle t\rangle^{-(2 + \alpha)}\]
as $t\to \infty$ and similarly for $t\to-\infty$. We can repeat the argument and deduce that $U \in \langle t\rangle^{-\infty} L^\infty$ by induction, which means $\dot U(t) = \int_{-\infty}^t R(s) U(s) ds$ is in $\langle t\rangle^{-\infty} L^\infty$.
Now that we have $U \in \langle t\rangle^{-\infty}W^{2,\infty}(\RR;\RR^{n-1})$ we can apply Lemma \ref{carleman} to deduce that $U = 0$.
\end{proof}
We are now in position to prove Theorem \ref{injectivity of tensors}, the injectivity of X-ray transform on tensors.

\begin{proof}[Proof of Theorem \ref{injectivity of tensors}] We first show i): by Corollary \ref{primitiveforkerI} (or Lemma \ref{trapped} for the trapping case with negative curvature), we get $\mc{O}(\rho^{k-1})$ pointwise bounds on the smooth function $u:=R_+\pi_0^*=R_-\pi_0^*f\in C^\infty(S^*M)$, and on the derivatives $|\nabla^vu|_G$, $|\nabla^hu|_G$.
We apply Pestov identity from Proposition \ref{pestov} and get
\[0=  \| X \nabla^v u\|_{L^2}^2 + (n-1)\| \pi_0^* f\|_{L^2}^2 - \langle \mc{R} \nabla^v u, \nabla^v u\rangle.\]
Then we follow the proof of \cite[Theorem 1]{GGSU}: we let $Z\in C^\infty(S^*M;\mc{Z})\cap \rho^{k-1}L^\infty$ so that $|XZ|_G\in \rho^{k-1}L^\infty$ and $|X^2Z|_G\in \rho^{k-1}L^\infty$, and define the quadratic form
\[
A(Z)=\|XZ\|^2_{L^2(S^*M)}-\cjg \mc{R}Z,Z\cjd_{L^2(S^*M)}.
\]
By the same argument as in the proof of \cite[Theorem 1]{GGSU}, we have $A(Z)\geq 0$ for such $Z$: the method is to first replace $Z$ by $Z_\eps:=\chi(\rho/\eps)Z$ where $\chi(s)=1$ for $s\geq 1$ and $\chi(s)=0$ for $s\leq 1/2$, then we get $A(Z_\eps)\geq 0$ since $g$ has no conjugate points and $A(Z_\eps)$ is the integrated index form on geodesics, and by letting $\eps\to 0$ we get $A(Z)=\lim_{\eps\to 0} A(Z_\eps)$ using $|X(\chi(\rho/\eps))|\leq C\eps$ and the bounds on $Z$. It then suffices to apply this with $Z=\nabla^vu$ (we use  $[X,\nabla^v]u=-\nabla^hu$ and $[X,\nabla^h]u=\mc{R}\nabla^vu$ to get bounds on $XZ$ and $X^2Z$ from bounds on $|\nabla^vu|$ and $|\nabla^hu|$) and we deduce that 
$\| X \nabla^v u\|_{L^2}^2- \langle \mc{R} \nabla^v u, \nabla^v u\rangle\geq 0$, so $f=0$.\\

For ii), we may assume, using Proposition \ref{kill transversal terms} that $\iota_{\rho^2\partial_\rho} f = 0$ near $\partial \bbar{M}$. We can then argue as in \cite[Proof of Theorem 1]{GGSU} to show that Proposition \ref{pestov} implies that $X^2 {\nabla}^v u = {\mathcal R} {\nabla}^v u$ if $u=R_+\pi_1^*f=R_-\pi_1^*f$ (using $I_1f=0$). When restricted to a geodesic curve $\gamma\subset M$, $J:= {\nabla}^v u|_\gamma $ is then a Jacobi field along $\gamma$ with (for $x=\gamma(0)\in M$)
\[\|J(\gamma(t))\|_g= \|\nabla^v u|_{\gamma_t(x)} \|_G \leq C \rho(t)^{k-2} \leq C\langle t\rangle^{-k+2},\ \ t\to\pm\infty\]
by using Lemma \ref{resolvent estimate} and Lemma \ref{x asymptotic}. 
Now apply Lemma \ref{no vanishing jacobi fields} to deduce that $J(\gamma(t)) = 0$ identically. Since $\gamma$ is an arbitrary complete geodesic with end points on $\partial M$, we have that $\nabla^v u = 0$ and therefore $u=\pi_0^*q$  for some $q \in \rho^{k-1} C^\infty(\overline M)$ and 
$Xu=\pi_1^*f$ implies $dq=f$.\\

For iii) we follow the idea of \cite{PSU} (see also \cite{GGSU}). First by Proposition \ref{kill transversal terms} we may assume $\iota_{\pl_\rho} f = 0$ near $\partial M$. 
Let $ u = R_+ \pi^*_m f=R_-\pi^*_mf$, which is smooth in $S^*M$ and belongs to $\rho^k L^\infty(\overline{S^*M})$ with similar estimates for its vertical/horizontal derivatives by Proposition \ref{primitiveforkerI}. 
We write $u= \sum\limits_\ell u_\ell$ where $u_\ell$ are eigenmodes of the vertical Laplacian $\Delta^v={\nabla^v}^*\nabla^v$. Define $\til{u} := u - \sum_{\ell\leq m-1} u_\ell$, then $X \til{u} = \pi^*_m f - \sum_{\ell\leq m-1} Xu_\ell$ has no eigenmodes above $m$ since $ \pi^*_m f = \sum_{\ell \leq m}f_\ell$. Furthermore, $X\til{u} = \sum\limits_{l \geq m} X u_l$ has no eigenmodes below $m-1$. Therefore, $X\til{u} = (X\til{u})_m + (X\til{u})_{m-1}$.  We apply Proposition \ref{pestov} to $\til{u}$ and get 
\[\begin{split}
\lambda_m \| (X\til{u})_m\|^2 + \lambda_{m-1} \|(X\til{u})_{m-1} \|^2 = & 
\| X\nabla^v \til{u}\|^2 - 
\langle {\mathcal R}  \nabla^v \til{u}, \nabla^v \til{u}\rangle + \\
& +(n-1) (\| (X\til{u})_m\|^2 +\| (X\til{u})_{m-1}\|^2 ).\end{split}
\]
where $\lambda_ m = m(m+n-2) $ and $\lambda_{m-1} = (m-1)(m+n-3)$.
Now if the curvature is non-positive, one has that $X \til{u} = 0$: indeed, 
\begin{eqnarray*}
\lambda_m \| (X\til{u})_m\|^2 + \lambda_{m-1} \|(X\til{u})_{m-1} \|^2& \geq& \| X\nabla^v \til{u}\|^2  + (n-1) (\| (X\til{u})_m\|^2 +\| (X\til{u})_{m-1}\|^2)\\
&\geq& \Big( \frac{(m-1)(m+n-2)^2}{m+n-3} + (n-1)\Big) \| (X\til{u})_{m-1}\|^2  \\
&&+ \Big( \frac{m(m+n-1)^2}{m+n-2}  + (n-1)\Big)\|(X\til{u})_{m}\|^2 
\end{eqnarray*}
where we used the computation in \cite[Lemma 4.3]{PSU} to deal with $\|X\nabla^v\til{u}\|^2$. 
This implies that $X \til{u} = 0$ with $\til{u}$ decaying to order $\mc{O}(\rho^k)$ as $\rho\to 0$. 
Thus necessarily $\til u = 0$ and the proof is complete.
\end{proof}

\section{Renormalized length and scattering map}

\subsection{Scattering map}\label{scatteringsec}
We define the scattering map using the rescaled flow $\bbar{\varphi}_\tau$.
\begin{definition}\label{nontrap}
For non-trapping asymptotically conic manifolds $(M,g)$, the scattering map $S_g: T^*\pl\bbar{M}\to T^*\pl\bbar{M}$ is defined by 
\[ S_g(z):= \bbar{\varphi}_{\tau^+(z)}(z).\]
In the case of trapping, the scattering map is defined on the subset $\{z\in T^*\pl\bbar{M}; \tau^+(z)<\infty\}$.
\end{definition}
We notice that $S_g$ is defined using a choice of normal form (in order to identify $\pl_\pm S^*M$ to $T^*\pl\bbar{M}$), but by \eqref{changeofcoord}, a change of normal form yields the 
coordinate change $(y,\eta)\mapsto (y,\eta+d\omega_0)$ on $T^*\pl\bbar{M}$ where $\omega_0\in C^\infty(\pl\bbar{M})$, which means that $S_g$ is simply being conjugated by this transformation of $T^*\pl\bbar{M}$.  

\begin{lemma}\label{scatteringlarge}
Let $(y_0,\eta_0)\in T^*\pl\bbar{M}$, then 
\begin{equation}
\label{scattering relation limit}
\lim_{\eps\to 0 }\epsilon.S_g(y_0,\epsilon^{-1} \eta_0) = |\eta_0|_{h_0}.\phi_\pi \Big(y_0,\frac{\eta_0}{|\eta_0|_{h_0}}\Big)
\end{equation}
where we defined the fiber dilation action $c.(y,\eta):=(y,c\eta)$ if $c>0, (y,\eta)\in T^*\pl\bbar{M}$ and 
$\phi_t:T^*\pl\bbar{M}\to T^*\pl\bbar{M}$ is the Hamilton flow of the Hamiltonian 
$\frac{1}{2}|\eta|^2_{h_0}$ at time $t$. 
\end{lemma}
\begin{proof}
Let us first take $|\eta_0|_{h_0}=1$. We use the bound \eqref{limitinggeodesics} at $s=\epsilon^{-1}\tau_\epsilon=\epsilon^{-1}\tau_g^+(y_0,\epsilon^{-1}\eta_0)$, which by \eqref{taueps} is equal to $\pi+\mc{O}(\epsilon)$, to deduce that as $\epsilon \to 0$,
\[
\epsilon.S_g(y_0,\epsilon^{-1} \eta_0) \to \phi_\pi (y_0,\eta_0).
\]
For the general case, it suffices to write 
$\epsilon.S_g(y_0,\epsilon^{-1} \frac{\eta_0}{|\eta_0|_{h_0}})= \frac{\epsilon'}{|\eta_0|_{h_0}}.S_g(y_0,{\epsilon'}^{-1} \eta_0)$ with $\epsilon':=\epsilon|\eta_0|_{h_0}\to 0$.
\end{proof}
\subsection{Renormalized length}
In this section we define the notion of a rescaled length $L_g(\gamma)$ for non-trapped  geodesics $\gamma : (-\infty,\infty) \to M$. The first observation we make is that for $\gamma(t)=\pi_0(\varphi_t(z))$ with $z\in S^*M$, we have $\rho(\gamma(t)) \leq C \langle t\rangle^{-1}$ when $t\to\pm \infty$ by Lemma \ref{x asymptotic}. This means that the quantity $L_g^\lambda(\gamma) = \int_{-\infty}^\infty \rho^\lambda(\varphi_t(z)) dt $ is finite if ${\rm Re}(\lambda) >1$. The goal is to extend this function holomorphically near $\lambda =0$:
\begin{proposition}
\label{rescaled geod length}
For each complete non-trapped geodesic $\gamma$, the function $L^\lambda_g(\gamma)$ has a meromorphic extension to ${\rm Re}(\lambda) >-1$ which is holomorphic at $\lambda =0$. Furthermore, the renormalized length $L_g(\gamma) := L^0_g(\gamma)$ is also given by
\[L_g(\gamma) = \lim\limits_{\epsilon \to 0} \Big(\ell_g(\gamma\cap \{\rho >\epsilon\}) - 2\epsilon^{-1}\Big)\]
where $\ell_g$ denotes the length with respect to $g$.
\end{proposition}
\begin{proof} 
First, observe that, writing $\gamma (t)=(\rho(t),y(t))$ in the product decomposition near $\pl \bbar{M}$, we have by \eqref{formluaX} and the bounds Lemma \ref{x asymptotic} that there are functions $a,b,c_j$ such that
\begin{equation}
\label{class of curves}
 \begin{gathered}
 \dot{\rho}(t) = a(t) \rho(t)^2,\quad \dot{a}(t)= b(t) \rho(t)^3,\quad \dot{y}_j(t) = c_j(t) \rho(t)^2, \\
 \textrm{ with }\dot{c}_j(t) = \mc{O}(\rho(t)^2), \quad |b(t)|=\mc{O}(1), \quad \lim\limits_{t\to\pm\infty} |a(t)| = 1,  \quad |a(t)| \leq 1. 
\end{gathered}\end{equation}
We use  the rescaled variable $\tau(t) := \int_{-\infty}^t \rho(\gamma(r))^2 dr$ and denote its inverse by $t(\tau)$. Direct computation yields that for curves satisfying \eqref{class of curves}, we get in the rescaled variable that $\bbar{\rho} (\tau) := \rho(\gamma(t(\tau)))$ has the following expression near $\tau^-:=\lim_{t\to-\infty}\tau(t)=0$:
\[
\bbar{\rho}(\tau) = \int_{0}^\tau a(t(r))dr \textrm{ and }a(t(\tau))=-1+\int_{0}^\tau b(t(r))\bbar{\rho}(r)dr=-1+\mc{O}(|\tau|^2)
\]
which implies that
\begin{equation}\label{tau asymptotic for general curve} 
\bbar{\rho}(\tau)=\tau+\mc{O}(|\tau|^3).
\end{equation}
The analogous asymptotic holds for $\tau$ near $\tau^+=\lim_{t\to +\infty}\tau(t)$ with $\tau^+-\tau$ as leading term.
We can then write
\begin{equation}
\label{L^lambda_g}
L^\lambda_g(\gamma) := \int_{-\infty}^\infty \rho^\lambda(\gamma(s))  ds
=  \int_{0}^{\tau_+} \bbar{\rho}^{\lambda-2}(\tau)   d\tau.
\end{equation}
Using the asymptotic \eqref{tau asymptotic for general curve}, it is direct to see that for a fixed $\tau_0\in(\tau^-,\tau_+)$
\[ L^\lambda_g(\gamma)=\frac{\tau_0^{\la-1}}{\la-1}+ \frac{(\tau_+-\tau_0)^{\la-1}}{\la-1}+H_{\la,\tau_0}\]
with $H_{\la,\tau_0}$ holomorphic in ${\rm Re}(\la)>-1$.
To show that $L_g(\gamma)$ can be obtained as the asymptotic limit of the length, we first note that since $\gamma$ intersects $\{\rho=\eps\}$ transversally for $\epsilon >0$ small enough. 
By \eqref{tau asymptotic for general curve}, the equation $\bbar{\rho}(\tau)=\eps$ has two solutions $\tau^\eps_\pm$ for $\eps>0$ small, satisfying
 $\tau_\mp^\eps = \tau_\mp\pm\eps+\mc{O}(\eps^3)$. We obtain, using in addition \eqref{tau asymptotic for general curve},
\[\ell_g(\gamma\cap \{\rho>\epsilon\}) = \int_{\tau^\eps_-} ^{\tau^\eps_+} {\bbar{\rho}}^{-2}(\tau) d\tau=
2\eps^{-1}-\tau_0^{-1}-(\tau_+-\tau_0)^{-1}+H_{0,\tau_0}+o(1)
\]
as $\eps\to 0$, which proves the claim.
\end{proof}

The reason of choosing $L_g$ instead of $\tau^+_g$ for the renormalized length is that $L_g$ is in some sense a good invariant geometric quantity, which only depends on the first jet at infinity of the boundary defining function (and in a simple way, as shown in \eqref{different bdf for rescaled length}), while $\tau^+_g$ really depends on the choice of $\rho$ in the bulk $M$, thus is not intrinsic.
We now investigate how the rescaled length depends on the boundary defining function. Suppose $\rho$ and $\til\rho$ are the boundary defining function for two coordinate systems under which the scattering metric $g$ is in normal form. Then $\til\rho = \rho + a \rho^2$ for some function $a\in C^\infty(\overline M)$. Denote by $L_g$ and $\til L_g$ to be the rescaled length with respect to $\rho$ and $\til\rho$. Let $\gamma(t)$ be a unit speed geodesic whose trajectory is the same as the rescaled flow (with respect to $x$) $\bbar\varphi_\tau(y_0,\eta_0)$ for some $(y_0,\eta_0) \in \partial_-S^*M$. By definition the rescaled length with respect to the boundary defining function $\til \rho$ is given by 
$${\til L}_g (\gamma)= \int_{-\infty}^\infty {\til\rho}^\lambda(\gamma(t))dt\mid_{\lambda = 0} = \int_{-\infty}^\infty \left(\rho(\gamma(t)) + a(\gamma(t)) \rho^2(\gamma(t)) \right)^\lambda dt \mid_{\lambda = 0}$$
We make a change of variable $\tau(t) = \int_{-\infty}^t \rho^{-2}(\gamma(t)) dt$ as in \eqref{L^lambda_g}, we get that
\[{\til L}_g (\gamma)= \int_0^{\tau(\infty)}  \bbar \rho^{\lambda -2}(\tau) ( 1 + \bbar a(\tau) \bbar \rho(\tau))^{\lambda}d\tau\mid_{\lambda = 0}\]
where $\bbar a(\tau) = a(\bbar\varphi_\tau(y_0,\eta_0))$. Split the integral into three parts we obtain, for any $\tau_0 >0$,
\[{\til L}_g(\gamma) = \int_0^{\tau_0} + \int_{\tau_0}^{\tau(\infty) - \tau_0} + \int_{\tau(\infty) -\tau_0}^{\tau(\infty)}\bbar \rho^{\lambda -2}(\tau) ( 1 + \bbar a(\tau) \bbar \rho(\tau))^{\lambda} d\tau \mid_{\lambda = 0}\]
The middle integral extends trivially to $\lambda = 0$ to become $\int_{\tau_0}^{\tau(\infty) - \tau_0} \bbar \rho^{-2}(\tau)d\tau$. The integral along $(0,\tau_0)$ can be treated by writing 
$$(1+ \bbar a (\tau) \bar \rho(\tau))^{\lambda} = 1 + \lambda \bbar a(\tau) \bbar\rho(\tau) + \lambda {\mathcal O}(\bbar \rho^2(\tau))$$
so that
$$\int_0^{\tau_0} \bbar \rho^{\lambda -2}(\tau) ( 1 + \bbar a(\tau) \bbar \rho(\tau))^{\lambda} d\tau \mid_{\lambda = 0} =  \int_0^{\tau_0}\bbar \rho^{\lambda -2}(\tau) \mid_{\lambda = 0} + \lambda \int_0^{\tau_0}\bbar \rho^{\lambda -1}(\tau)\bbar a(\tau) \mid_{\lambda = 0}.$$
Recall from \eqref{tau asymptotic for general curve} that $\bbar\rho(\tau ) = \tau( 1+ O(\tau))$ which then gives
$$\int_0^{\tau_0} \bbar \rho^{\lambda -2}(\tau) ( 1 + \bbar a(\tau) \bbar \rho(\tau))^{\lambda} d\tau \mid_{\lambda = 0} = \int_0^{\tau_0}\bbar \rho^{\lambda -2}(\tau) \mid_{\lambda = 0} + \lambda \int_0^{\tau_0}\tau^{\lambda-1} (1+{\mathcal O}(\tau))^{\lambda -1}\bbar a(\tau) \mid_{\lambda = 0}.$$
It is easy to check that the meromorphic extention of the second integral is holomorphic at $\lambda =0$ and is equal to $\bbar a(0)$. Therefore,
\[
\int_0^{\tau_0} \bbar \rho^{\lambda -2}(\tau) ( 1 + \bbar a(\tau) \bbar \rho(\tau))^{\lambda} d\tau \mid_{\lambda = 0} = \int_0^{\tau_0} \bbar \rho^{\lambda -2}(\tau)d\tau \mid_{\lambda = 0} + \bbar a(0).
\]
Similarly the integral on the interval $(\tau(\infty) - \tau_0, \tau(\infty))$ can be treated the same way to give
\[
\int_{\tau(\infty)- \tau_0}^{\tau(\infty)} \bbar \rho^{\lambda -2}(\tau) ( 1 + \bbar a(\tau) \bbar \rho(\tau))^{\lambda} d\tau \mid_{\lambda = 0} = \int_{\tau(\infty) - \tau_0}^{\tau(\infty)} \bbar \rho^{\lambda -2}(\tau)d\tau \mid_{\lambda = 0} + \bbar a(\tau(\infty)).
\]
We see therefore that if $\gamma$ is a unit speed geodesic whose trajectory is given by $\bbar{\varphi}_\tau(p_0,\eta_0)$ with $(p_0,\eta_0) \in \partial_- S^*M$, then 
\begin{eqnarray}
\label{different bdf for rescaled length}
{\til L}_g(\gamma) - L_g(\gamma) = a(p_0) + a({S}_g(p_0,\eta_0)).
\end{eqnarray}
This allows us to obtain the following:
\begin{lemma}
\label{L uniquely def}
Let $g_1$ and $g_2$ be asymptotically conic metrics with the same scattering maps so that $\rho_0^{-2} |d\rho_0|_{g_j} = 1 + {\mathcal O}(\rho_0^2)$, $j = 1,2$ for some boundary defining function $\rho_0$. If the rescaled lengths of $g_1$ and $g_2$ agree with respect to a boundary defining function $\rho_0$ satisfying the conditions of Definition \ref{defAE}, then they agree for all such boundary defining functions. 
\end{lemma}

\subsection{Determination of boundary metric from $S_g$}
In this section, we will study some cases where the scattering map determines the metric at the boundary. It does not play a role in the main Theorems. This part only uses geodesics staying arbitrarily close to $\pl\bbar{M}$ (those corresponding to $|\eta|$ very large), and the arguments thus work whether there is trapping or not, since such geodesics are never trapped.
First, we notice that if  $c>0$, we have for each $(y,\eta)\in T^*\pl\bbar{M}$
\[ |\eta|_{ch_0}.\phi^{ch_0}_{\pi}\Big(y,\frac{\eta}{|\eta|_{ch_0}}\Big)=c|\eta|_{h_0}.\phi^{h_0}_{c\pi}\Big(y,\frac{\eta}{c|\eta|_{h_0}}\Big)=|\eta|_{h_0}.\phi^{h_0}_{\pi}\Big(y,\frac{\eta}{|\eta|_{h_0}}\Big)
\]
where $\phi_t^{ch_0}$ is the Hamilton flow of $|\eta|_{ch_0}^2/2$. This shows that, using only \eqref{scattering relation limit}, we can not determine the boundary metric $h_0$ but at best 
only a multiple of $h_0$ can be recovered.

\begin{lemma}\label{ergodic}
Let $g,g'$ be two asymptotically conic metrics on $M$ in normal form so that, up to pulling-back by a diffeomorphism, $g=d\rho^2/\rho^4+h_\rho/\rho^2$ and $g'=d\rho^2/\rho^4+h'_\rho/\rho^2$.
If the geodesic flow of the boundary metric $(\pl\bbar{M},h_0)$ at time $\pi$, i.e. the map $\phi^{h_0}_\pi:
S^*\pl\bbar{M}\to S^*\pl\bbar{M}$,  is ergodic with respect to the Liouville measure, and if $S_g=S_{g'}$, then  $h_0'=ch_0$ for some $c>0$. This is in particular true if $h_0$ has negative curvature or more generally if it has mixing geodesic flow.
\end{lemma}
\begin{proof}
First, we notice that the normal form of Lemma \ref{normalform} for $g$ and $g'$ (with associated boundary defining function $\rho$ and $\rho'$) allows to construct a diffeomorphism $\psi:\bbar{M}\to \bbar{M}$ so that, near $\pl\bbar{M}$, $\psi^*\rho'=\rho$ and $\psi^*g'=\frac{d\rho^2}{\rho^4}+\frac{h'_\rho}{\rho^2}$. We can then use \eqref{scattering relation limit} to deduce that for all 
$(y,\eta)\in T^*\pl\bbar{M}$ with $|\eta|_{h_0}=1$,
\[ |\eta|_{h_0'}.\phi'_\pi\Big(y,\frac{\eta}{|\eta|_{h_0'}}\Big)=\phi_\pi(y,\eta),\]
where $\phi'_t$ is the Hamilton flow of $\frac{1}{2}|\eta|^2_{h_0'}$.
This implies that the map $(y,\eta)\mapsto |\eta|_{h_0'(y)}$ is invariant by $\phi_\pi$ on the unit 
tangent bundle $S^*\pl\bbar{M}$ of $(\pl\bbar{M},h_0)$, it is thus constant by the ergodicity assumption on $\phi_\pi$. This shows that $h_0'=ch_0$ for some $c>0$. If $\phi_t$ is mixing, one has (assuming 
Liouville measure $\mu$ on $S^*\pl\bbar{M}$ has mass $1$ for notational simplicity) for $u,v\in C^\infty(S^*\pl\bbar{M})$
\[ \lim_{t\to \infty} \int_{S^*\pl\bbar{M}}u\circ \phi_t.v \, d\mu=\int_{S^*\pl\bbar{M}}u\, d\mu \int_{S^*\pl\bbar{M}} v\, d\mu\] 
and the same is then true with $t\in \pi\NN\to \infty$. This implies that if $u\circ \phi_\pi=u$,
then $\int u^2 \, d\mu=(\int u\, d\mu)^2$, and by the case of equality in H\"older we get that $u$ is constant.
\end{proof}
We also the same result under conditions on the injectivity radius of 
$(\pl \bbar{M},h_0)$. 
\begin{lemma}
\label{boundary determination of metric}
Let $g,g'$ be two asymptotically conic metrics on $M$ and denote their boundary metrics by $h_0$ and $h'_0$. Assume that the radius of injectivity of $h_0$ is larger than $\pi$ and that the scattering maps $S_g$ and $S_{g'}$ agree. Then $h_0=c\psi^*h_0'$ for some diffeomorphism $\psi$ on $\pl\bbar{M}$ and some $c>0$.
\end{lemma}
In particular, we note that the assumptions of Lemma \ref{boundary determination of metric}
are valid when $(\pl\bbar{M},h_0)$ is a sphere of an Euclidean space of radius $R>1$
\begin{proof}
As before, we use the normal form of Lemma \ref{normalform} for $g$ and $g'$.
We first show that $h_0$ and $h'_0$ are conformally related.
Let $\phi_t$ and $\phi'_t$ be the Hamilton flow of $\frac{1}{2}|\eta|^2_{h_0}$ and $\frac{1}{2}|\eta|^2_{h'_0}$ on $T^*\pl\bbar{M}$.
Since the radius of injectivity of $(\pl\bbar{M},h_0)$ is greater than $\pi$, 
for each $y_0\in \pl M$ the geodesic sphere $S_{y_0}(\pi)$ of center $y_0$ and radius $\pi$
is an embedded submanifold in $\pl\bbar{M}$.
By using \eqref{scattering relation limit}, we get that the corresponding geodesic sphere $S_{y_0}'(\pi)$ for $h_0'$ agrees with $S_{y_0}(\pi)$.
Let $(y_0,\eta_0)\in T^*\partial \bbar{M}$ and set $(y_1,\eta_1) = \phi_{-\pi}(y_0,\eta_0)$, then 
$(y_1,\eta_1) = \phi'_{-\pi}(y_0,\eta_0)$ by \eqref{scattering relation limit}. 
We see that by Gauss lemma,  $\eta_0$ belongs to the normal bundle of the geodesic sphere $S_{y_1}(\pi)=S'_{y_1}(\pi)$.  
Since $T_{y_0} S_{y_1}(\pi) = T_{y_0} S'_{y_1}(\pi)$ we get $\cjg \eta_0,\eta\cjd_{h_0(y_0)} = 0 \iff \cjg\eta_0,\eta\cjd_{h'_0(y_0)} = 0$. This shows that $h_0' = e^{2f} h_0$ for some smooth function $f$.
We need to show that $f$ is constant. Let $\hat{f}=f-\frac{1}{{\rm vol}_{h_0}(M)}\int_M f\, {\rm dvol}_{h_0}$ so that $h_0'=e^{2c}e^{2\hat{f}}h_0$ with $c=\frac{1}{{\rm vol}_{h_0}(M)}\int_M f\, {\rm dvol}_{h_0}$. Note that $\hat{f}$ vanishes somewhere.
Let  $y\in \partial \bbar{M}$ be a point such that $\hat{f}(y)  =0$. Then for all $\eta\in T_{y}\partial \bbar{M}$ such that $|\eta|_{h_0} =1= e^{-c}|\eta|_{h_0'}$, we have using \eqref{scattering relation limit}
\[\begin{split}
1= & |\phi_{\pi}'(y,e^{-c}\eta)|_{h_0'} = e^{c+\hat{f}(\pi_0 (\phi_\pi(y,\eta)))} |\phi_{\pi} '(y,e^{-c}\eta)|_{h_0}  \\
= & e^{\hat{f}(\pi_0( \phi_{\pi}(y,\eta)))} |\phi_{\pi} (y,\eta)|_{h_0}=e^{\hat{f}(\pi_0 (\phi_\pi(y,\eta)))}.
\end{split}\]
Therefore, we have 
\begin{eqnarray}
\label{distance 1 means f=0}
\hat{f}(y) = 0 \implies {\hat{f}(\pi_0(\phi_\pi(y,\eta)))} = 0,\ \ \forall \eta \in S^*_y \partial \bbar{M}
\end{eqnarray}
where $S^*\partial \bbar{M}$ is the unit tangent bundle for $h_0$. 
Let $y_1\in \partial \bbar{M}$ be a point such that $\hat{f}(y_1) =0$ and $s>0$ be small enough such that $\pi+s$ is smaller than the injectivity radius of $(\partial \bbar{M}, h_0)$. If $y_2$ is a point for which $d_{h_0}(y_1,y_2) =s$ then there exists a unit vector $\eta_1 \in S_{y_1}^* \partial \bbar{M}$ such that $\pi_0(\phi_s(y_1,\eta_1)) = y_2$. This means that 
\[d_{h_0}(y_2, \pi_0(\phi_{\pi}(y_1,\eta_1))) = \pi-s,\ \ d_{h_0}(y_2, \phi_{-\pi}(y_1,\eta_1))  = \pi+s.\]
By continuity of the map $y\mapsto d_{h_0}(y,y_2)$ on $S_{y_1}(\pi)$, this means that there exists a point $y_0\in \pl\bbar{M}$ such that $d_{h_0}(y_0,y_1) = d_{h_0}(y_0,y_2) = \pi$. Therefore by \eqref{distance 1 means f=0} we have that $0 = \hat{f}(y_1) = \hat{f}(y_0) = \hat{f}(y_2)$. Since $s>0$ is an arbitrarily small number this means that $\hat{f}(y)$ vanishes in a small neighbourhood of $y_1$. The proof is complete by an open-closed argument.
\end{proof}

\subsection{Determination of the metric jets at $\pl\bbar{M}$ from $S_g$}
We consider two asymptotically conic metrics $g,g'$, and as before the normal form of Lemma \ref{normalform} for $g$ and $g'$ (with associated boundary defining function $\rho$ and $\rho'$) allows to construct a diffeomorphism $\psi:\bbar{M}\to \bbar{M}$ so that, near $\pl\bbar{M}$, $\psi^*\rho'=\rho$ and $\psi^*g'=\frac{d\rho^2}{\rho^4}+\frac{h'_\rho}{\rho^2}$. Up to 
replacing $g'$ by $\psi^*g'$, we can thus assume that $g,g'$ are both in normal form with the same boundary defining function $\rho$, i.e.
\[g=\frac{d\rho^2}{\rho^4}+\frac{h_\rho}{\rho^2}, \quad g'=\frac{d\rho^2}{\rho^4}+\frac{h'_\rho}{\rho^2}\]
and we will assume that the boundary metrics coincide: $h_0=h_0'$.
We consider the Taylor expansion at $\rho=0$ of the dual metrics $h^{-1}_\rho$ and ${h'_\rho}^{-1}$ 
 to $h_\rho$ and $h'_{\rho}$:
\begin{equation}\label{exphrho}
h^{-1}_\rho=\sum_{j=0}^{m}\rho^j h_j+\mc{O}(\rho^{m+1}), \quad 
{h'_\rho}^{-1}=\sum_{j=0}^{m}\rho^j h'_j+\mc{O}(\rho^{m+1})
\end{equation}
for $m\in \NN$, and we define for each $j$
\[T_j:=h_j-h_j'.\]
Here we view $h_j,h_j',T_j$ as homogeneous functions of order $2$ on $T^*\pl\bbar{M}$ and,
by abuse of notations, $h_0$ denotes both the metric on $T\pl\bbar{M}$ and $T^*\pl\bbar{M}$. 

We shall apply perturbation theory in the regime $\epsilon\to 0$ to the system \eqref{tilde dynamic}. Observe that for each $\epsilon >0$ the ODE system \eqref{tilde dynamic} is given by a 1-parameter smooth family of  vector fields $\til{X}_\epsilon$ given by
\begin{equation}\label{tildeXeps} 
\til{X}_\eps:= \til{\xi}_0\pl_{\til{\rho}}-\til{\rho}\big(|\til{\eta}|^2_{h_{\eps\til{\rho}}}+\frac{\eps\til{\rho}}{2}(\partial_\rho |\til{\eta}|^2_{h_\rho})|_{\rho=\eps\til{\rho}}
\big)\pl_{\til{\xi}_0}+
H_{\eps\til{\rho}}.
\end{equation}
The variables $\til{\rho},\til{\xi}_0,(\til{y},\til{\eta})$ belong to $[0,1]\x [-1,1]\x T^*\pl\bbar{M}$, and the vector field $H_{\eps\til{\rho}}$ is the Hamilton vector field of $(\til{y},\til{\eta})\mapsto \frac{1}{2}h^{-1}_{\eps\til{\rho}}(\til{\eta},\til{\eta})=\frac{1}{2}|\til{\eta}|^2_{h_{\eps\til{\rho}}}$
on $T^*\pl\bbar{M}$ with respect to the Liouville symplectic form. In local coordinates one has 
\[H_{\rho}=\sum_{j,k=1}^{n-1}h_{\rho}^{jk}\til{\eta}_k\pl_{\til{y}_j}-\frac{1}{2}\sum_{j=1}^{n-1}\pl_{\til{y}_j}|\til{\eta}|^2_{h_\rho}\pl_{\til{\eta}_j}.\]
Remark that a priori the integral curves solving the ODE \eqref{tilde dynamic} belong to the hypersurface $\{\til{\xi}_0^2+\til{\rho}^2|\til{\eta}|^2_{h_{\eps\til{\rho}}}=1\}$ but the expression defining the vector field $\til{X}_\eps$ extends smoothly to a neighborhood of that hypersurface.

The vector fields $\til{X}_{\epsilon}$ and $\til{X}'_{\epsilon}$ corresponding to the metrics $g$ and $g'$ have a smooth expansion in powers of $\eps$: 
\[\til{X}_{\epsilon}=\sum_{j=0}^m\eps^jX_j+\mc{O}(\eps^{m+1}),\quad \til{X}'_{\epsilon}=\sum_{j=0}^m\eps^jX'_j+\mc{O}(\eps^{m+1})\]
with $X_j, X_j'$ smooth vector fields. 
Under the assumption that $h_j=h'_j$ for $j\leq m-1$, we get in addition that $X_j=X'_j$ for $j\leq m-1$ and  
\begin{equation}
\label{tilX and tilX'}
X_m-X'_m=-\big(\frac{m}{2}+1\big){\til{\rho}}^{m+1}T_{m}(\til{\eta},\til{\eta})\pl_{\bbar{\xi}_0}+\til{\rho}^{m} (H_{m}-H'_m)
\end{equation}
where $H_m$ (resp. $H_m')$ is the Hamilton vector field of $\frac{1}{2}h_m(\til{\eta},\til{\eta})$  (resp. $\frac{1}{2}h_m(\til{\eta},\til{\eta})$) on $T^*\pl\bbar{M}$, and $H_m-H_m'$ is the Hamilton field of $T_m(\til{\eta},\til{\eta})$.
We introduce on $T^*\partial \bbar{M}$ the coordinate system $(E,\til{y},\hat \eta)$ where $E := |\til\eta|^2_{h_0}$ and $\hat \eta = \til \eta/\sqrt{E}$.

We set $c_\eps(s)$ and $c'_\eps(s)$ to be the trajectories of $\til{X}_\eps$ and $\til{X}'_\eps$ respectively with the same initial condition 
\[\big(\til{\rho},\til{\xi}_0,E,\til{y},\hat{\eta}\big)=(0,1,1,y_0,\eta_0) \in [0,1]\times [-1,1] \times \RR^+\times S^*\partial M.\]
These solutions have a Taylor expansion in powers of $\eps$ of the form
\[ c_\eps(s)=\sum_{j=1}^m\eps^j c_j(s)+\mc{O}(\eps^{m+1}), \quad c'_\eps(s)=\sum_{j=1}^m\eps^j c'_j(s)+\mc{O}(\eps^{m+1}).\]

\begin{lemma}\label{solutionlinearised}
Assume that $g,g'$ are two asymptotically conic metrics written in normal form such that their boundary jets $h_j$ and $h_j'$ are equal up to $j\leq m-1$ for some $m\geq 1$.
If the scattering maps $S_g$ and $S_{g'}$ agree,
then for all $(y_0,\eta_0)\in S^*\pl\bbar{M}$: 
\begin{equation}\label{energyvar} 
\int_0^\pi \sin(s)^m H_0T_m(e^{sH_0}(y_0,\eta_0))ds=0
\end{equation}
if $H_0$ is the Hamilton field of $\frac{1}{2}|\eta|^2_{h_0}$ on $T^*\pl\bbar{M}$.
If in addition $H_0T_m=0$, then
\begin{equation}\label{rhocm}
\til{\rho}(c_m(\pi))-\til{\rho}(c'_m(\pi))=-(\frac{m}{2}+1)\int_0^\pi \sin^{m+2}(s)T_m(e^{sH_0}(y_0,\eta_0))ds,
\end{equation}
\begin{equation}\label{equcos}
\int_0^\pi \cos(s)\sin(s)^{m+1}T_m(e^{sH_0}(y_0,\eta_0))ds=0,
\end{equation}
\begin{equation}\label{equdirectionH0}
\int_0^\pi (\sin(s)^m-(\frac{m}{2}+1)\sin(s)^{m+2})T_m(e^{sH_0}(y_0,\eta_0))ds=0.
\end{equation}
\end{lemma}
\begin{proof}
Observe that when $\epsilon = 0$, the vector field $\til{X}_0=\til{X}_\epsilon|_{\epsilon=0}$ is
\begin{equation}
\label{tilX_0}
\til{X}_0 = \til{\xi}_0\pl_{\til{\rho}}-\til{\rho}|\til{\eta}|^2_{h_{0}} \partial_{\til{\xi}_0}+
H_{0}
\end{equation}
and its integral curves with initial condition $(\til{\rho},\til{\xi}_0,\til{y},\til{\eta})|_{s=0}=(0,1,y_0,\eta_0)$ when $|\eta_0|_{h_0}=1$
are given by 
\[c_0(s)=(\sin(s),\cos(s),e^{sH_0}(y_0,\eta_0)).\] 
Note that since $H_0$ is the geodesic vector field on $T^*\pl\bbar{M}$ for the metric $h_0$,
\begin{eqnarray}
\label{dE(H_0) = 0}
dE(H_0) = 0.
\end{eqnarray}
In the coordinates $(\til{\rho},\til{\xi}_0, E,\til{y},\hat \eta)$, the first order linearization of $\til{X}_0$ takes the convenient block form:
\begin{eqnarray}
\label{dtilX_0}
d\til{X}_0=\left(\begin{array}{c c c c}
0 & 1 & 0 & 0\\
-E & 0 & -\til{\rho} & 0\\
 0 & 0 & 0 & 0 \\
 0 & 0 & \pl_E H_0 & d_{\til{y},\hat{\eta}}H_0
\end{array}
\right),
\end{eqnarray}
where we view vector fields as column vectors and the last line correspond to the pair of coordinates $(\til{y},\hat \eta)\in T^*\pl\bbar{M}$: for example, 
\[\til{X}_0=\left(\begin{array}{c}
\til{\xi}_0\\
 -\til{\rho}|\til{\eta}|^2_{h_{0}}\\
0 \\ 
H_0
\end{array}\right),\] 
the columns in \eqref{dtilX_0} are the partial derivatives 
$(\pl_{\til{\rho}}, \pl_{\til{\xi}_0},\pl_E,\pl_{(\til{y},\hat{\eta})})$.
The vector fields $\til{X}_{\epsilon}$ and $\til{X}'_{\epsilon}$ corresponding to the metrics $g$ and $g'$ have a smooth expansion in powers of $\eps$: 
\[\til{X}_{\epsilon}=\sum_{j=0}^m\eps^jX_j+\mc{O}(\eps^{m+1}),\quad \til{X}'_{\epsilon}=\sum_{j=0}^m\eps^jX'_j+\mc{O}(\eps^{m+1})\]
with $X_j, X_j'$ smooth vector fields. 
Under the assumption that $h_j=h'_j$ for $j\leq m-1$, we get in addition that $X_j=X'_j$ for $j\leq m-1$ and  
\begin{equation}
\label{tilX and tilX'}
X_m-X'_m=-\big(\frac{m}{2}+1\big){\til{\rho}}^{m+1}T_{m}(\til{\eta},\til{\eta})\pl_{\bbar{\xi}_0}+\til{\rho}^{m} (H_{m}-H'_m)
\end{equation}
where $H_m$ (resp. $H_m')$ is the Hamilton vector field of $\frac{1}{2}h_m(\til{\eta},\til{\eta})$  (resp. $\frac{1}{2}h'_m(\til{\eta},\til{\eta})$) on $T^*\pl\bbar{M}$, and $H_m-H_m'$ is the Hamilton field of $T_m(\til{\eta},\til{\eta})$.

We set $c_\eps(s)$ and $c'_\eps(s)$ to be the trajectories of $\til{X}_\eps$ and $\til{X}'_\eps$ respectively with the same initial condition 
\[\big(\til{\rho}_0,\til{\xi}_0,E,\til{y},\hat{\eta}\big)=(0,1,1,y_0,\eta_0) \in [0,1]\times [-1,1] \times \RR^+\times S^*\partial M.\]
By Taylor expanding in powers of $\epsilon$ the equations $\til{X}_\eps(c_\eps(s))=\dot{c}_\eps(s)$ and $\til{X}'_\eps(c'_\eps(s))=\dot{c}_\eps'(s)$, we obtain using $X_j=X_j'$ for $j\leq m-1$
\[c_j(s)=c'_j(s) \textrm{ for }j\leq m-1\]
and the $\eps^m$ term yields the equation
\[ \dot{c}_m(s)-\dot{c}'_m(s)=X_m(c_0(s))-X'_m(c_0(s))+dX_0(c_0(s)).(c_m(s)-c_m'(s)).\]
Writing $e_m(s):=c_m(s)-c_m'(s)$ and using \eqref{tilX and tilX'}, we obtain the linear ODE system
\begin{equation}\label{equationek}
\dot{e}_m(s)=X_m(c_0(s))-X'_m(c_0(s))+dX_0(c_0(s)).e_m(s).
\end{equation}
To solve this equation, we introduce the matrix solution of 
\begin{eqnarray}
\label{ODE of R}
\dot{R}(s)=d\til{X}_0(c_0(s))R(s), \quad R(0)={\rm Id}
\end{eqnarray}
which can be solved explicitly (in the $(\til{\rho},\til{\xi}_0,E,(\til{y},\hat{\eta}))$ coordinates) as 
\begin{eqnarray}
\label{R matrix}
R(s)=\left(\begin{array}{c c c c}
\cos(s) & \sin(s) & a_1(s) & 0\\
-\sin(s) & \cos(s) & a_2(s) & 0\\
 0 & 0 & 1 & 0 \\
 0 & 0 & K(s) & L(s)
\end{array}
\right).
\end{eqnarray}
The function $L(s)$ solves the ODE $\dot{L}(s)=dH_{0}(e^{sH_0}(y,\hat{\eta}))L(s)$ with $L(0)={\rm Id}$ on $\{E=1\}=S^*\pl\bbar{M}$, and $a_1,a_2,K$ are smooth functions that do not play any role for later. Note that $L(s)=d\phi_s$ if $\phi_s=e^{sH_0}$ is the geodesic flow on $S^*\pl \bbar{M}$.
The function $e_m(s)$ is then given by 
\begin{equation}\label{ems}
e_m(s)=R(s)\int_0^s R(t)^{-1}(X_m(c_0(t))-X_m'(c_0(t)))dt.
\end{equation}
Let $\tau_\eps$ and $\tau'_\eps$ be the positive solutions of $\til{\rho}(c_\eps(\tau_\eps))=0$ and $\til{\rho}(c'_\eps(\tau'_\eps))=0$; we note that $\tau_0=\tau'_0=\pi$. 
Expanding the equation in powers of $\eps$, we obtain that 
$\tau_\eps=\tau'_\eps+\eps^m\tau_m+\mc{O}(\eps^{m+1})$
and
\[\til{\rho}(c_m(\pi))-\til{\rho}(c'_m(\pi))+(\tau_m-\tau'_m)d\til{\rho}.\til{X}_0(c_0(\pi))=0.\]
Since $\til{\xi}_0(c_0(\pi))=-1$, this gives 
\begin{equation}\label{equtauj}
(\tau_m-\tau'_m)=\til{\rho}(c_m(\pi))-\til{\rho}(c'_m(\pi)).
\end{equation}
The identity $S_g=S_g'$ implies that 
$c_\eps(\tau_\eps)=c_\eps'(\tau'_\eps)$ and taking the $\eps^{m}$ coefficient of the Taylor expansion of this equation, we deduce that 
\[0= c_m(\pi)-c_m'(\pi)+ \til{X}_0(c_0(\pi))(\tau_m-\tau_m'),\]
which can be rewritten using \eqref{equtauj} under the form
\begin{equation}\label{empi}
e_m(\pi)+(\til{\rho}(c_m(\pi))-\til{\rho}(c'_m(\pi)))\til{X}_0(c_0(\pi))=0.
\end{equation}
Combining \eqref{empi} and \eqref{ems} we deduce that
\begin{equation}\label{empifinal}
R(\pi)\int_0^\pi R(t)^{-1}(X_m(c_0(t))-X_m'(c_0(t)))dt+(\til{\rho}(c_m(\pi))-\til{\rho}(c'_m(\pi)))\til{X}_0(c_0(\pi))=0.
\end{equation}
We can write $R(s)^{-1}$ under the matrix form
\begin{eqnarray}
\label{Rinv matrix}
R(s)^{-1}=\left(\begin{array}{c c c c}
\cos(s) & -\sin(s) & b_1(s) & 0\\
\sin(s) & \cos(s) & b_2(s) & 0\\
 0 & 0 & 1 & 0 \\
 0 & 0 & -L(s)^{-1}K(s) & L(s)^{-1}
\end{array}
\right)
\end{eqnarray}
for some functions $b_j(s)$.
We consider the $\pl_E$ (i.e fiber radial) component of \eqref{empifinal}: since $dE.\til{X}_0=0$ and  
$dE.(\til{X}_m-\til{X}'_m)=\til{\rho}^mdE.(H_m-H_m')=-2\til{\rho}^mH_0T_m$ by \eqref{tilX and tilX'} (viewing $T_m$ as a homogeneous function of degree $2$ in $\til{\eta}$), 
this leads to 
\[ 0=\int_0^\pi \sin(s)^mH_0T_m(e^{sH_0}(y_0,\eta_0))ds,\]
which is \eqref{energyvar}.

Assume now that $T_{m}$ is a Killing 2-tensor, i.e. that $H_0T_m=0$. This implies  
\[(H_{m}- H_{m}') h_0 = - H_0T_m = 0.\]
In other words, in the coordinate system $(\til\rho, \til\xi_0, E, \tilde y, \hat \eta)$, the vector field $\til{X}_{m} - \til{X}_{m}'$ given by \eqref{tilX and tilX'} lies in the kernel of $d E$, it thus has the matrix form in our coordinates
\[\til{X}_m- \til{X}_{m}'=\left(\begin{array}{c}
0\\
 -\big(\frac{m}{2}+1\big){\til{\rho}}^{m+1}T_{m}(\til{\eta},\til{\eta})\\
0 \\ 
\til{\rho}^m(H_m-H_m')
\end{array}\right).\]
Identifying the $\pl_{\til{\xi}_0}$ component of \eqref{empifinal}, we get 
\[ \int_0^\pi (\frac{m}{2}+1)\sin^{m+1}(s)\cos(s)T_m(e^{sH_0}(y_0,\eta_0))ds=0\]
 implying equation \eqref{equcos}.

Identifying the $\pl_{\til{\rho}}$ component of \eqref{empifinal} and using $\til{X}_0(c_0(\pi))=-\pl_{\til{\rho}}+H_0$ and \eqref{tilX and tilX'}, we obtain \eqref{rhocm} if $H_0T_m=0$.

To obtain \eqref{equdirectionH0}, we consider the $H_0$ component of \eqref{empifinal}. Since $d(e^{sH_0})_{y,\eta}.H_0(y,\eta)=H_0(e^{sH_0}(y,\eta))$ for all $y,\eta\in T^*\pl\bbar{M}$, the direction $H_0$ is preserved by the linearisation of the geodesic flow $H_0$ on  $\pl\bbar{M}$, and the same holds for 
$\ker \la$ if $\la:=\sum_{j}\til{\eta}_jdy_j$ is the Liouville $1$-form on $T^*\pl \bbar{M}$. 
We then have $L(s)H_0(y_0,\eta_0)=H_0(e^{sH_0}(y_0,\eta_0))$, that is for $f$ smooth and $Y\in \ker \lambda$ 
\[ L(s)^{-1}(f H_0+ Y)(\phi_s(y_0,\eta_0))=f(\phi_s(y_0,\eta_0)) H_0(y_0,\eta_0)+
d\phi_s^{-1}(\phi_s(y_0,\eta_0))Y(\phi_s(y_0,\eta_0))\] 
if $\phi_s=e^{sH_0}$ is the geodesic flow on $S^*\pl\bbar{M}$. The last line of the system \eqref{empifinal} becomes
\[ \int_0^\pi \sin(s)^mL(s)^{-1}(H_m-H_m')(c_0(s))ds+(\til{\rho}(c_m(\pi))-\til{\rho}(c'_m(\pi)))\til{X}_0(c_0(0))=0\]
Applying $\la$ to this equation gives 
\[\til{\rho}(c_m(\pi))-\til{\rho}(c'_m(\pi)) + \int_0^\pi \sin(s)^{m}\la((H_m-H_m')(e^{sH_0}(y_0,\eta_0)))ds =0.\]
A direct calculation gives that $\la(H_m-H_m')=T_m$ so we conclude that
\[\til{\rho}(c_m(\pi))-\til{\rho}(c'_m(\pi)) + \int_0^\pi \sin(s)^{m}T_m(e^{sH_0}(y_0,\eta_0)))ds =0.\]
Substituting \eqref{rhocm} for $\til{\rho}(c_m(\pi))-\til{\rho}(c'_m(\pi))$ into the above equation yields \eqref{equdirectionH0}. The proof is complete.
\end{proof}

\begin{corollary}\label{negcurv}
Assume that $g,g'$ are asymptotically conic metrics with the same boundary metric $h_0$ and assume that the geodesic flow $e^{2\pi H_0}$ of $h_0$ at time $2\pi$ is ergodic on $S^*\pl\bbar{M}$ (which is in particular true if $h_0$ has negative curvature or more generally if it has mixing geodesic flow).
If the scattering operator $S_{g'}=S_g$ agree, then there is a smooth diffeomorphism $\psi:\bbar{M}\to \bbar{M}$ fixing the boundary so that $\psi^*g$ and $g$ agree to infinite order at the boundary. 
\end{corollary}
\begin{proof}
Assume $g$ and $g'$ agree to order $m$, i.e. $h_j=h_j'$ for all $j<m$ using the notation 
\eqref{exphrho}. We will show that $T_m=h_m-h_m'$ must vanish.
We apply $H_0$ to \eqref{energyvar} and integrate by part (using $H_0(f\circ e^{sH_0})=\pl_s(f\circ e^{sH_0})$) to get for all $(y,\eta)\in S^*\pl\bbar{M}$
\[0=\int_0^\pi \sin(s)^{m}H_0^2T_m(e^{sH_0}(y,\eta))ds=
-m\int_0^\pi \cos(s)\sin(s)^{m-1}H_0T_m(e^{sH_0}(y,\eta))ds.\]
Applying again this method and using \eqref{energyvar}, we obtain for $m>1$ that
\[0=\int_0^\pi \cos(s)^2\sin(s)^{m-2}H_0T_m(e^{sH_0}(y,\eta))ds.\]
Using \eqref{energyvar}, this gives $\int_0^\pi\sin(s)^{m-2}H_0T_m(e^{sH_0}(y,\eta))ds=0$.
We repeat the operation, and one gets when $m$ is even that for all $(y,\eta)\in S^*\pl\bbar{M}$
\[0=\int_0^\pi H_0T_m(e^{sH_0}(y,\eta))ds= T_m(e^{\pi H_0}(y,\eta))-T_m(y,\eta).\]
If $e^{\pi H_0}$ is ergodic, then necessarily $T_m=c\,h_0$ for some $c\in \RR$ and $H_0T_m=0$. If $m$ is odd, we end up with $\int_0^\pi \sin(s)H_0T_m(e^{sH_0}(y,\eta))ds=0$, which gives after another application of $H_0^2$ and two integrations by parts
\begin{equation}\label{Tmsym} 
(H_0T_m)\circ e^{\pi H_0}=-H_0T_m.
\end{equation}
Here again if $e^{2\pi H_0}$ is ergodic, as $(H_0T_m)\circ e^{2\pi H_0}=H_0T_m$, we conclude that $H_0T_m=c'\, h_0$ for some $c'\in \RR$. But integrating this identity over $S^*\pl\bbar{M}$ shows that $c'=0$ and thus $T_m=c\, h_0$ for some $c$ by ergodicity of $e^{2\pi H_0}$. In all case we have $T_m=c\, h_0$ for some $c\in\RR$. 
Next, we apply the identity \eqref{equdirectionH0} and obtain 
\[c \int_0^\pi (\sin(s)^m -(\frac{m}{2}+1)\sin(s)^{m+2}))ds=0.\]  
Since $\int_0^\pi \sin(s)^{m+2}ds=\frac{m+1}{m+2}\int_0^\pi \sin(s)^{m}ds$, we conclude that 
$c=0$ if $m\geq 2$. To deal with the case $m=1$, we can change the normal form by using Lemma \ref{normalform}: this amounts to change the boundary defining function $\rho$ in that Lemma to 
$\hat{\rho}=\rho(1+c_0\rho+\mc{O}(\rho^2))$ for some $c_0\in\RR$ so that, by \eqref{changeofrho}, 
$g$ in this normal form becomes $\frac{ds^2}{s^4}+\frac{\hat{h}_s}{s^2}$ with
\[ \hat{h}_s-h_s=2s c_0h_0+\mc{O}(s^2).\] 
The change of normal form amounts to pulling-back $g$ by a smooth diffeomorphism $\psi$ on $\bbar{M}$, fixing $\pl\bbar{M}$ pointwise and so that $\psi^*\hat{\rho}=\rho$.  
Since we know that $h_1-h_1'=c \,h_0$ for some $c$, we can choose $c_0=c/2$ so that in a normal form near $\pl\bbar{M}$, the expansion of $\psi^*g$ and $g'$ in normal form agree to order $2$, i.e and we are reduced to the case $m=2$ dealt with above (we also use that $S_{\psi^*g}=S_g$ for such diffeomorphism $\psi$ by the remark following Definition \ref{nontrap}).  
\end{proof}

In a companion paper in collaboration with Mazzucchelli \cite[Theorem 7.5]{GMT}, we prove a boundary determination similar to Corollary \ref{negcurv} from the scattering map for the case where the boundary is the canonical sphere with curvature $+1$.

\subsection{Deformation Rigidity of Rescaled Lens Map}
Assume now that, on $M$, we have a family of non-trapping asymptotically conic metrics of the form $g(s) = \frac{d\rho^2}{\rho^4} + \frac{h(s)}{\rho^2}$ near $\pl\bbar{M}$  for $s\in (-1,1)$ (we set $g = g(0)$) such that $L_{g(s)} = L_{g}$ and $S_{g(s)} = S_{g}$ for each $s$. Furthermore we will assume that $h(s) = h(0) + \mc{O}(\rho^\infty)$. We will denote by prime the derivative with respect to $s$ at $s=0$.
Observe that if $h(s)$ is a smooth family of tensors which are smooth up to the boundary, then $g'(\cdot,\cdot)\in \rho^\infty C^\infty(\overline M ; S^2({}^{\rm sc} T^*M))$ which is annihilated by $\iota_{\partial_\rho}$ for $\rho>0$ small. Therefore, by the estimates of Lemma \ref{x asymptotic}, for each geodesic $\gamma(t)$ of $g$,
$|g'(\dot \gamma, \dot \gamma)|= \mc{O}(t^{-\infty})$ as $|t|\to\infty$.
The goal of this section is to show the
\begin{proposition}\label{impliesI2=0}
Let $g(s)$ be a smooth family of non-trapping asymptotically conic metrics of the form $g(s) = \frac{d\rho^2}{\rho^4} + \frac{h(s)}{\rho^2}$ near $\pl\bbar{M}$.
If $S_{g(s)}=S_g$ and $L_{g(s)}=L_g$ for all $s$, then $g'=\pl_sg(s)|_{s=0}$ satisfies $I_2(g')=0$
if $I_2$ is the X-ray transform associated to $g$.
\end{proposition}
\begin{proof}
We denote the projection on $\bbar{M}$ of the (non-trapped) integral curves of the rescaled geodesic vector fields $\bbar{X}_{s}$ of $g(s)$ by $\bbar{\gamma}_s(\tau,z)$, if $z\in \pl_-S^*M$ is the 
initial value at time $\tau=0$ (these are simply the geodesics of $g(s)$ with time rescaled).
Let $\tau_+(s,z)$ be the time to that $\bbar{\varphi}^s_{\tau_+(s,z)}(z)\in \pl_+S^*M$ if 
$\bbar{\varphi}_t^s(z)$ denotes the flow of $\bbar{X}_{s}$ at time $\tau\geq 0$. We let $z'=S_{g(s)}(z)=S_{g}(z)\in \pl_+S^*M$. By assumptions, we have $\bbar{X}_s=\bbar{X}_0+\mc{O}(s\rho^\infty)$ when viewed as smooth vector fields on 
$[0,\eps]_{\rho}\x [-1,1]_{\bbar{\xi}_0}\x T^*\pl\bbar{M}_{y,\eta}$. This implies that for each 
 $N>0$ and $\tau\geq 0$ small
\[ \bbar{\varphi}_\tau^s(z)=\bbar{\varphi}_\tau^0(z)+\mc{O}(s\max_{u\leq\tau}\rho(\bbar{\gamma}_s(u,z))^N), \quad \bbar{\varphi}_{-\tau}^s(z')=\bbar{\varphi}_{-\tau}^0(z')+\mc{O}(s\max_{u\leq\tau}\rho(\bbar{\gamma}_s(-u,z'))^N).\]
Let $\bbar{\gamma}'(\tau,z):=\pl_s\bbar{\gamma}_s(\tau,z)|_{s=0}$ and dot denotes the $\tau$ derivative. Then we obtain 
\begin{equation}\label{gammaOtau^N}
\bbar{\gamma}(\tau,z)=\mc{O}(\tau^N),\quad \dot{\bbar{\gamma}}(\tau,z)=\mc{O}(\tau^N)
\end{equation} 
uniformly for $\tau$ small, with the similar bounds for 
$\bbar{\gamma}'(-\tau,z')$ and $\dot{\bbar{\gamma}}'(-\tau,z')$. We recall that $L_{g(s)}$ is obtained from the formula \eqref{L^lambda_g} in terms of the curve $\bbar{\gamma}_s(\tau,z)$, and we shall vary \eqref{L^lambda_g} with respect to $s$.
Let $\eps\in(0,\tau_+(s,z)/4)$ be small. 
Using that $\rho^4g(s)(\dot{\bbar{\gamma}}(\tau,z),\dot{\bbar{\gamma}}(\tau,z))=1$ we compute 
for ${\rm Re}(\la)>1$ (using $\pl_s(\rho^2g(s))|_{s=0}=\rho^2g'=h'$)
\[ \begin{split}
\pl_s \Big[\int_0^\eps \rho(\bbar{\gamma}_s(\tau,z))^{\la-2}d\tau\Big]\Big|_{s=0}=& \pl_s \Big[\int_0^\eps \rho(\bbar{\gamma}_s(\tau,z))^{\la+2}g(s)(\dot{\bbar{\gamma}}_s(\tau,z),\dot{\bbar{\gamma}}_s(\tau,z))d\tau\Big]\Big|_{s=0}\\
=& \la\int_0^\eps \rho(\bbar{\gamma}(\tau,z))^{\la-3} d\rho(\bbar{\gamma}'(\tau,z)) d\tau\\
 & +\int_0^\eps \rho(\bbar{\gamma}(\tau,z))^{\la}\pl_s(\rho^2g(\dot{\bbar{\gamma}}_s(\tau,z),\dot{\bbar{\gamma}}_s(\tau,z)))|_{s=0}d\tau\\
& + \int_0^\eps \rho(\bbar{\gamma}(\tau,z))^{\la}h'(\dot{\bbar{\gamma}}(\tau,z),\dot{\bbar{\gamma}}(\tau,z)))d\tau .
\end{split}\]
Using that  $h'\in \rho^\infty C^\infty(\bbar{M};S^2(T^*\bbar{M}))$ and the bounds \eqref{gammaOtau^N}, we deduce that the integrals above all extend holomorphically to $\la\in \CC$ since $\rho(\bbar{\gamma}(\tau,z))=\mc{O}(\tau)$ as $\tau\to 0$, moreover this extension is uniformly $\mc{O}(\eps^N)$ for $\la\in \CC$ in compact sets. The same argument and estimates also applies to 
\[\pl_s|_{s=0} \Big[\int_0^\eps \rho(\bbar{\gamma}_s(-\tau,z'))^{\la-2}d\tau\Big].\] 
Next we compute (with $\tau_+(z):=\tau_+(0,z))$)
\[ \begin{split}
\pl_s\Big[\int_\eps^{\tau_+(s,z)-\eps} \rho(\bbar{\gamma}_s(\tau,z))^{\la-2}d\tau\Big]\Big|_{s=0} =& \pl_s \Big[\int_\eps^{\tau_+(s,z)-\eps} \rho(\bbar{\gamma}(\tau,z))^{\la+2}g(\dot{\bbar{\gamma}}_s(\tau,z),\dot{\bbar{\gamma}}_s(\tau,z))d\tau\Big]\Big|_{s=0}\\
& +\la\int_\eps^{\tau_+(z)-\eps}  \rho(\bbar{\gamma}(\tau,z))^{\la-3} d\rho(\bbar{\gamma}'(\tau,z)) d\tau\\
 & +\int_\eps^{\tau_+(z)-\eps} \rho(\bbar{\gamma}(\tau,z))^{\la}h'(\dot{\bbar{\gamma}}(\tau,z),\dot{\bbar{\gamma}}(\tau,z)))d\tau.
\end{split}\]
where we denoted $h'=\rho^2g'$ (defined globally on $\bbar{M}$ and tangential to boundary near $\pl\bbar{M}$). We thus obtain 
\[\begin{split}
\Big(\pl_s \Big[\int_0^\eps \rho(\bbar{\gamma}_s(\tau,z))^{\la-2}d\tau\Big]\Big|_{s=0}\Big)\Big|_{\la=0}= & 
\pl_s \Big[\int_\eps^{\tau_+(s,z)-\eps} \rho(\bbar{\gamma}(\tau,z))^2g(\dot{\bbar{\gamma}}_s(\tau,z),\dot{\bbar{\gamma}}_s(\tau,z))d\tau\Big]\Big|_{s=0}\\
&+ \int_\eps^{\tau_+(z)-\eps} h'(\dot{\bbar{\gamma}}(\tau,z),\dot{\bbar{\gamma}}(\tau,z))d\tau+\mc{O}(\eps^N).
\end{split}\]
As $\eps\to 0$, the second term converges to $I_2(g')$ by Definition \ref{defXray}. Making the change of variable $\tau \mapsto t(\tau,z)=\int_\eps^\tau \rho^{-2}(\bbar{\gamma}(r,z))dr$ we have \[\int_\eps^{\tau_+(s,z)-\eps} \rho(\bbar{\gamma}(\tau,z))^2g(\dot{\bbar{\gamma}}_s(\tau,z),\dot{\bbar{\gamma}}_s(\tau,z))d\tau=
\int_{0}^{t_s(\eps)}g(\dot{\gamma}_s(t),\dot{\gamma}_s(t))dt
\]
where $\gamma_s(t):=\bbar{\gamma}_s(\tau,z)$ and $t_s(\eps):=t(\tau_+(s,z)-\eps,z)$. This is the energy functional of the curve $\gamma_s(t)$ with respect to $g=g(0)$, and since $\gamma_0(t)$ is a geodesic for $g$, we get by the variation formula for the energy
\[\pl_s \Big[\int_0^\eps \rho(\bbar{\gamma}_s(\tau,z))^{\la-2}d\tau\Big]\Big|_{s=0}=g(\gamma'(\eps,z),\dot{\gamma}_0(t_0(\eps))-g(\gamma'(\eps,z'),\dot{\gamma}(0))=\mc{O}(\eps^N)\]
for all $N$. Letting $\eps\to 0$, we conclude that 
\[\pl_s L_{g(s)}(z)|_{s=0}=I_2(g')\]
and that proves the Proposition in the non-trapping case. 
\end{proof}

\begin{proof}[Proof of Theorem \ref{def rigid}]
We now prove deformation rigidity following the argument of \cite{GGSU}. First 
by Corollary \ref{negcurv} we can assume that $g(s)=g(0)+\mc{O}(\rho^\infty)$.
By Proposition \ref{impliesI2=0} applied at the point $s_0$, we have that at each $s_0$, $I_2^{g(s_0)}(g'(s_0))=0$.
By Theorem \ref{injectivity of tensors}, there exists $q(s)\in \rho^\infty C^\infty (M; 
{^{\rm sc}T}^*M)$ such that $g'(s) = D^{g(s)}q(s)$, or equivalently 
\begin{eqnarray}
\label{g'(s) is lie derivative}
g'(s) = {\mathcal L}_{\frac{1}{2}q(s)^\sharp} g(s),
\end{eqnarray}
with $q(s)^\sharp$ the vector field dual to $q(s)$ by $g(s)$.  Integrating the vector field produces the desired family of diffeomorphisms.
\end{proof}

\end{document}